\documentclass[11pt]{imsart}
\RequirePackage{amsthm,amsmath,amsfonts,amssymb}
\RequirePackage[numbers,sort&compress]{natbib}
\RequirePackage[colorlinks,citecolor=blue,urlcolor=blue]{hyperref}
\RequirePackage{graphicx}

\usepackage{latexsym, amssymb, amscd, amsxtra, amsmath}
\usepackage{graphics, graphicx, color}
\usepackage{ifpdf}
\usepackage[caption=false]{subfig}
\usepackage{multirow}
\usepackage{hyperref}
\usepackage{verbatim}
\usepackage{mathtools}
\usepackage{graphicx, lscape}
\usepackage{algorithm, algorithmicx, algpseudocode}
\usepackage{lipsum}
\usepackage{hyperref}
\startlocaldefs
\theoremstyle{plain}

\newtheorem{theorem}{Theorem}[section]
\newtheorem{lemma}[theorem]{Lemma}
\newtheorem{proposition}[theorem]{Proposition}
\newtheorem{corollary}[theorem]{Corollary}
\theoremstyle{definition}
\newtheorem{definition}[theorem]{Definition}
\newtheorem{example}{Example}

\newtheorem{remark}{Remark}

\newcommand{\norm}[1]{\left\Vert#1\right\Vert_2}
\newcommand{\op}[1]{\left\Vert#1\right\Vert}
\newcommand{\fnorm}[1]{\left\Vert#1\right\Vert_F}
\newcommand{\abs}[1]{\left\vert#1\right\vert}

\newcommand{\Real}{\mathbb R}

\newcommand{\W}{\mathfrak{W}}

\def\argmax{\mathop{\rm argmax}}

\def\ang{\mbox{Angle}}

\def\tr{\mbox{trace}}

\def\diag{\mbox{diag}}
\def\rank{\mbox{rank}}

\def\diag{\mbox{Diag}}

\def\av{\mathbf a}
\def\bv{\mathbf b}
\def\cv{\mathbf c}
\def\dv{\mathbf d}
\def\ev{\mathbf e}
\def\fv{\mathbf f}
\def\gv{\mathbf g}
\def\hv{\mathbf h}

\def\gv{\mathbf g}
\def\kv{\mathbf k}

\def\nv{\mathbf n}

\def\pv{\mathbf p}
\def\qv{\mathbf q}
\def\rv{\mathbf r}

\def\tv{\mathbf t}
\def\uv{\mathbf u}
\def\vv{\mathbf v}
\def\wv{\mathbf w}
\def\xv{\mathbf x}

\def\zv{\mathbf z}

\def\Av{\mathbf A}

\def\Cv{\mathbf C}
\def\Dv{\mathbf D}

\def\Hv{\mathbf H}
\def\Iv{\mathbf I}
\def\Jv{\mathbf J}
\def\Kv{\mathbf K}
\def\Lv{\mathbf L}
\def\Mv{\mathbf M}

\def\Ov{\mathbf O}

\def\Qv{\mathbf Q}
\def\Rv{\mathbf R}
\def\Sv{\mathbf S}

\def\Uv{\mathbf U}
\def\Vv{\mathbf V}
\def\Wv{\mathbf W}
\def\Xv{\mathbf X}
\def\Yv{\mathbf Y}
\def\Zv{\mathbf Z}

\newtheorem{Assumption}{Assumption}

\newcommand{\zetav}{\mbox{\boldmath$\zeta$}}

\newcommand{\iotav}{\mbox{\boldmath{$\iota$}}}

\newcommand{\muv}{\mbox{\boldmath{$\mu$}}}
\newcommand{\nuv}{\mbox{\boldmath{$\nu$}}}
\newcommand{\xiv}{\mbox{\boldmath{$\xi$}}}

\newcommand{\Phiv}{\mbox{\boldmath{$\Phi$}}}

\newcommand{\psiv}{\mbox{\boldmath{$\psi$}}}
\newcommand{\omegav}{\mbox{\boldmath{$\omega$}}}
\newcommand{\Omegav}{\mbox{\boldmath{$\Omega$}}}
\newcommand{\Sigmav}{\mbox{\boldmath{$\Sigma$}}}
\newcommand{\Lambdav}{\mbox{\boldmath{$\Lambda$}}}

\newcommand{\Xiv}{\mbox{\boldmath{$\Xi$}}}

\newcommand{\Ac}{\mathcal{A}}
\newcommand{\Bc}{\mathcal{B}}

\newcommand{\Dc}{\mathcal{D}}

\newcommand{\Lc}{\mathcal{L}}

\newcommand{\Nc}{\mathcal{N}}

\newcommand{\Sc}{\mathcal{S}}
\newcommand{\Tc}{\mathcal{T}}
\newcommand{\Uc}{\mathcal{U}}
\newcommand{\Vc}{\mathcal{V}}

\newcommand{\Xc}{\mathcal{X}}
\newcommand{\Yc}{\mathcal{Y}}

\def\1v{\mathbf 1}
\def\0v{\mathbf 0}

\def\mdp{\wv_{\textup{MDP}}}
\def\smdp{\wv_{\textup{SMDP}}}
\allowdisplaybreaks

\endlocaldefs

\begin{document}

\begin{frontmatter}
\title{Optimal Test-Data Piling in HDLSS Classification with Covariance Heterogeneity}
\runauthor{Kim, Ahn, and Jung}
\runtitle{Optimal Test-Data Piling with Covariance Heterogeneity}

\begin{aug}
\author[A]{\fnms{Taehyun}~\snm{Kim}\ead[label=e1]{tk3036@columbia.edu}},
\author[B]{\fnms{Jeongyoun}~\snm{Ahn}\ead[label=e2]{jyahn@kaist.ac.kr}}
\and
\author[C]{\fnms{Sungkyu}~\snm{Jung}\ead[label=e3]{sungkyu@snu.ac.kr}}
\address[A]{Department of Statistics,
Columbia University\printead[presep={,\ }]{e1}}

\address[B]{Department of Industrial and Systems Engineering,
KAIST\printead[presep={,\ }]{e2}}

\address[C]{Department of Statistics and Institute for Data Innovation in Science, Seoul National University\printead[presep={,\ }]{e3}}
\end{aug}

\begin{abstract}
This work addresses a longstanding question in high-dimensional linear classification: Is perfect classification achievable in heterogeneous covariance structures? We focus on the phenomenon of data piling, where projected data points collapse onto discrete values. We provide a comprehensive characterization of two distinct types of data piling. The first type of data piling refers to the phenomenon where projecting the training data onto a certain direction yields exactly two distinct values—one for each class. This occurs universally when the data dimension $p$ exceeds the sample size $n$. The second type concerns independent test data and arises asymptotically as $p \to \infty$ with fixed $n$. While previous work established the existence of such double data piling under homogeneously spiked covariance structures using negatively ridged classifiers, our analysis extends to the more general and realistic case of heterogeneous covariance. We identify an optimal direction among all piling directions that maximizes the separation between test data piles, which is called the Second Maximal Data Piling direction. An algorithm based on data splitting is proposed to compute this direction using only training data. Our analysis reveals a key insight: the main obstacle to discovering this direction is the imbalance of the tail eigenvalues, rather than differences in spike count, spike magnitude, or the alignment of leading eigenspaces. Extensive simulations confirm our theoretical results and demonstrate the effectiveness of the proposed classifier across a wide range of high-dimensional scenarios.
\end{abstract}



\end{frontmatter}


\section{Introduction}\label{sec:intro}
Recently, it has been observed that extremely complicated models which interpolate training data can also achieve nearly zero generalization error \citep{Zhang2017, Belkin2019}. This contradicts an important concept of the classical statistical learning framework in which a careful regularization is inevitable to reduce generalization error. This phenomenon, called \textit{benign overfitting}, has been confirmed both empirically and theoretically in the context of linear regression, and it turned out that a nearly zero or even negative ridge parameter can be optimal in the overparameterized regime \citep{Bartlett2020, Kobak2020, Holzmuller2020, Hastie2022}. 

In this work, we investigate the high-dimensional binary classification problem in which benign overfitting occurs under generalized heterogeneous covariance models. To gain theoretical insight into this phenomenon, we take the High-Dimension Low-Sample-Size (HDLSS) asymptotic regime of \citet{Hall2005} where the dimension of data $p$ tends to grow while the sample size $n$ is fixed. 
HDLSS data have often been found in many scientific fields, such as bioinformatics including next-generation sequencing data analysis, chemometrics, and image processing. Such HDLSS data are often best classified by linear classifiers \citep{Shao2011, Qiao2009, Niyazi2024}.
For binary classification with $p > n$, \citet{Ahn2010} observed the data piling phenomenon, that is, projections of training data onto a direction vector are identical for each class. Among such directions, the \textit{maximal data piling} direction uniquely gives the largest distance between the two piles of training data. The maximal data piling direction is defined as 
\begin{align}\label{eq:MDP}
    \mdp := \argmax\limits_{\wv : \| \wv \|_2 = 1}{\wv^\top\Sv_{B}\wv } \text{ subject to } \wv^\top\Sv_{W}\wv = 0,
\end{align}
where $\Sv_{W}$ and $\Sv_{B}$ are the $p \times p$ within-class and between-class scatter matrices of training dataset $\Xc$, respectively. \citet{Ahn2010} demonstrated that a classification rule using $\mdp$ as the normal vector to a discriminative hyperplane achieves better classification performance than classical linear classifiers when there are significantly correlated variables in high dimensions.

However, the maximal data piling direction has not been considered as an effective classifier since it depends too much on training data, resulting in poor generalization performances \citep{marron2007, Lee2013}. In general, while the training data are piled on $\mdp$, independent test data are not piled on $\mdp$. Recently, \citet{Chang2021} revealed the existence of the \textit{second data piling} direction, which gives a data piling of independent test data, under the HDLSS asymptotic regime.
In addition, they showed that a negatively ridged linear discriminant vector, projected onto a low-dimensional subspace, can be a \textit{second maximal data piling} direction, which yields a maximal asymptotic distance between two piles of independent test data. 

A second data piling direction is defined asymptotically as $p \to \infty$, unlike the first data piling of training dataset $\Xc$ for any fixed $p > n$. For a sequence of directions $\left\{\wv\right\} = (\wv^{(1)}, \ldots, \wv^{(p-1)}, \wv^{(p)}, \wv^{(p+1)}, \ldots)$, in which $\wv^{(q)} \in \Real^{q}$ for $q \in \mathbb{N}$, we write $\wv \in \Real^{p}$ for the $p$th element of $\left\{\wv \right\}$. Let $Y, Y'$ be independent random vectors from the same population of $\Xc$, and write $\pi(Y) = k$ if $Y$ belongs to class $k$. We formally define the second data piling below.

\begin{definition}[Second Data Piling]\label{def:SDP}
    The collection of all sequences of second data piling directions is defined as
    \begin{align*}
        \Ac := \left\{\left\{\wv \right\} \in \W_\Xc  : \text{ for any } Y, Y' \text{ with } \pi(Y) = \pi(Y'),~p^{-1/2}\wv^\top(Y - Y') \xrightarrow{P} 0 \text{ as } p \to \infty \right\}
    \end{align*}
    where $\W_\Xc = \left\{\left\{\wv\right\} : \wv \in \Sc_{\Xc}, \| \wv \|_2 = 1 \text{ for all } p \right\}$, and $\Sc_{\Xc} = {\rm span}(\Sv_W) \cup {\rm span}(\Sv_B)$ is the sample space. 
\end{definition}

There are infinitely many sequences of second data piling directions, as will be evident in later sections. Among those, analogously to $\mdp$ in (\ref{eq:MDP}), we can consider an \emph{optimal} direction which maximizes the asymptotic distance between the two piles of independent test data.

\begin{definition}[Second Maximal Data Piling]\label{def:SMDP}
    If $\left\{\vv \right\} \in \Ac$ satisfies 
    \begin{center}
        $\left\{ \vv \right\} \in \argmax\limits_{\left\{\wv\right\} \in \Ac }{D(\wv)}$,
    \end{center}
    where $D(\wv)$ is the probability limit of $p^{-1/2}|\wv^\top(Y_{1} - Y_{2})|$ for $\pi(Y_{k}) = k$ $(k = 1, 2)$, then we call $\{\vv\}$ a sequence of second maximal data piling directions. 
\end{definition}


\citet{Chang2021} showed that the second maximal data piling direction exists and by using it, asymptotic perfect classification of independent test data is possible. They assumed that the population mean difference is as large as $\|\muv_{(1)} - \muv_{(2)} \|_2 = O(p^{1/2})$ where $\muv_{(k)} \in \Real^p$ ($k =1, 2$) is the population mean vector of the $k$th class and both populations share a common, or \textit{homogeneous}, spiked covariance matrix. The spiked covariance model, first introduced by \citet{johnstone2001}, refers to high-dimensional population covariance matrix structures in which a few eigenvalues of the matrix are much larger than the other nearly constant eigenvalues \citep{ahn2007, Jung2009a, Shen2016}. More precisely, \citet{Chang2021} assumed the common covariance matrix $\Sigmav$ has $m$ spikes, that is, $m$ eigenvalues increase at the order of $p^\beta$ as $p \to \infty$ with some $\beta \in [0, 1]$, while the other eigenvalues are nearly constant, averaging to $\tau^2 > 0$. 

With such assumptions, \citet{Chang2021} showed that if $\Sigmav$ has weak spikes (that is, $0 \leq \beta < 1$), then projections of independent test data are asymptotically piled on two distinct points on $\mdp$ as $p \to \infty$, similar to the projections of the training data. In fact, \citet{Hall2005} also observed that second data piling occurs under this setting with classical classifiers such as support vector machine (SVM) \citep{Vapnik1995} or distance-weighted discrimination (DWD) \citep{marron2007}. However, if $\Sigmav$ has strong spikes (that is, $\beta = 1$), then projections of independent test data tend to be respectively distributed along two parallel affine subspaces in a low-dimensional subspace $\Sc = {\rm span}(\hat\uv_1, \ldots, \hat\uv_m, \mdp) \subset \Sc_{\Xc}$, where $\hat\uv_i$ is the $i$th eigenvector of $\Sv_W$. See Figure~\ref{fig:intro_figure} for an illustration. 
Furthermore, $\vv_{-\tau^2}$, a ridged linear discrimination vector with the choice of the negative ridge parameter $-\tau^2$, projected onto $\Sc$, is asymptotically orthogonal to these affine subspaces. 
As can be inspected from  Figure~\ref{fig:intro_figure}, independent test data are asymptotically piled onto $\vv_{-\tau^2}$, which is a second maximal data piling direction. 

\begin{figure}
    \centering
    \includegraphics[width=0.8\linewidth]{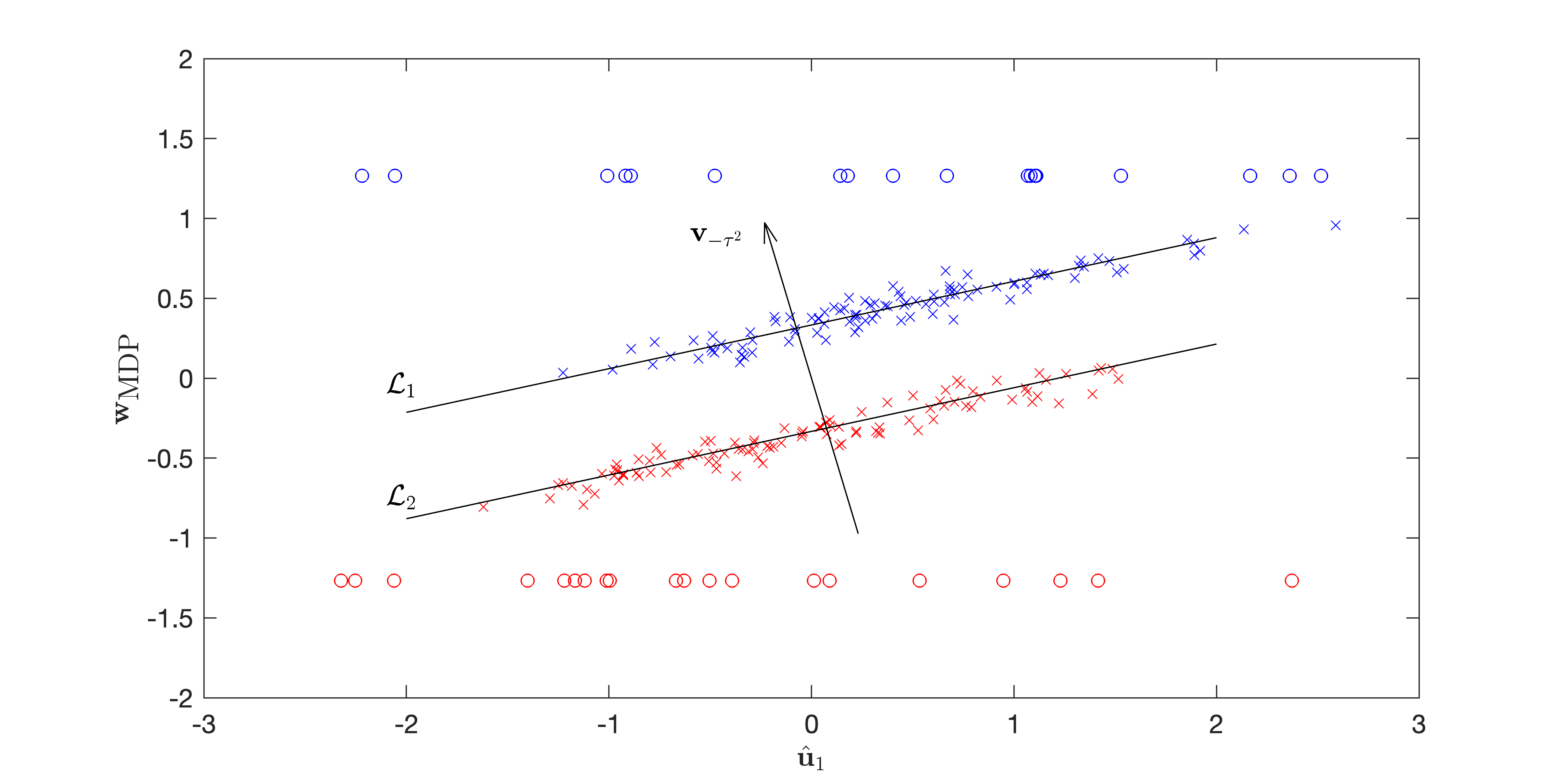}
    \caption{Double data piling phenomenon for homogeneous covariance model with one strong spike ($m = 1$). 
    Training data (class 1: blue circles, class 2: red circles) are piled on two distinct points on $\mdp$, while the projections of independent test data (class 1: blue crosses, class 2: red crosses) are distributed along parallel lines $\Lc_1$ and $\Lc_2$, respectively, when projected to $\Sc = {\rm span}(\hat\uv_1, \mdp)$. Both $\Lc_1$ and $\Lc_2$ are nearly orthogonal to $\vv_{-\tau^2}$. 
    }
    \label{fig:intro_figure}
\end{figure}

We note that, under the Gaussian assumption and homogeneous covariance model, the optimal separating hyperplane is normal to $\Sigmav^{-1}(\muv_{(1)} - \muv_{(2)})$. Under the heterogeneous covariance assumptions, however, the optimal decision boundary is no longer a plane but a quadric surface: 
$\xv^\top \Cv\xv + \bv^\top \xv + a = 0$, where $\Cv = \Sigmav_{(2)}^{-1} - \Sigmav_{(1)}^{-1}$, $\bv = 2(\Sigmav_{(1)}^{-1}\muv_{(1)} - \Sigmav_{(2)}^{-1}\muv_{(2)})$, $a = \muv_{(2)}^\top \Sigmav_{(2)}^{-1}\muv_{(2)} - \muv_{(1)}^\top \Sigmav_{(1)}^{-1}\muv_{(1)}$ and $\Sigmav_{(k)}$ is the covariance matrix of the $k$th class ($k = 1, 2$). However, the normal vector at the center of $\muv_{(1)}$ and $\muv_{(2)}$, $\xv_0 = \tfrac{1}{2}(\muv_{(2)} + \muv_{(1)})$, is simplified to 
\begin{align*}
    \nv = \Cv \xv_0 + \bv = -(\Sigmav_{(2)}^{-1} + \Sigmav_{(1)}^{-1}) (\muv_{(2)} - \muv_{(1)}),
\end{align*}
meaning that when the signal is strong enough, one can approximate the optimal decision boundary with a separating hyperplane, orthogonal to $\nv$. 

In this work, we develop a novel understanding of the second data piling direction, under generalized \textit{heterogeneous} spiked covariance models. Moreover, we show that a second maximal data piling direction can also be obtained purely from the training data, which can be seen as an estimate of the normal vector $\nv$. It turns out that the second maximal data piling direction is the direction closest to $\mdp$ in the nullspace of leading eigenspace. Based on this observation, we propose a data-splitting approach to obtain an estimate of the second maximal data piling direction, and confirm that the associated linear classification rule can achieve perfect classification. Furthermore, we show that the projected ridged linear discriminant vector---previously shown to achieve perfect classification in the homogeneous case---no longer becomes a second data piling direction in general for any ridge parameter.

We assume that for $k = 1, 2$, $\Sigmav_{(k)}$ has $m_{k}$ spikes such that $m_{k}$ eigenvalues increase at the order of $p^{\beta_{k}}$ as $p \to \infty$ while the other eigenvalues are nearly constant, averaging to $\tau_{k}^2 > 0$. We say that $\Sigmav_{(k)}$ has strong spikes if $\beta_{k} = 1$, or weak spikes if $0 \leq \beta_{k} < 1$. Also, we say that two covariance matrices have equal tail eigenvalues if $\tau_{1} = \tau_{2}$, or unequal tail eigenvalues if $\tau_{1} \ne \tau_{2}$.

\subsection{Main Contributions}
We provide a complete characterization of the second data piling phenomenon under generalized heterogeneous spiked covariance models, which includes a wide range of scenarios, covering covariance matrices with either strong spikes ($\beta_{k} = 1$) or weak spikes ($\beta_{k} < 1$), and either equal tail eigenvalues ($\tau_{1} = \tau_{2}$) or unequal tail eigenvalues ($\tau_{1} \ne \tau_{2}$). We show that projections of independent test data tend to be respectively distributed along two affine subspaces, denoted $\Lc_{1}$ and $\Lc_{2}$, in a low-dimensional subspace $\Sc \subset \Sc_{\Xc}$ (Theorem~\ref{thm:test data piling}). The `signal' subspace $\Sc$, spanned by some sample eigenvectors of $\Sv_W$ and $\mdp$, is obtained by removing the noisy directions from $\Sc_{\Xc}$, and is also characterized for each of the scenarios. 
Unlike the homogeneous case, the affine subspaces $\Lc_{1}$ and $\Lc_{2}$ need not be parallel to each other. 
However, we reveal that there are two parallel affine subspaces of greater dimension, containing each of these affine subspaces, in $\Sc$. Also, we find a somewhat counter-intuitive phenomenon that the eigenvectors of $\Sv_W$ corresponding to the largest eigenvalues do not always contribute to $\Sc$, but some other seemingly unimportant eigenvectors capture important variability. 

Based on the characterization of the second data piling phenomenon, we provide a unified view on the second \textit{maximal} data piling direction: The second maximal data piling direction can be obtained by projecting $\mdp$ onto the nullspace of the common leading eigenspace (Theorems~\ref{thm:SMDP piling distance} and~\ref{thm:SMDP characterization}). Building on this insight and a data-splitting approach, we propose Second Maximal Data Piling (SMDP) algorithms to estimate a second maximal data piling direction, and to compute discrimination rules based on the estimated directions (Algorithms~\ref{alg:1} and~\ref{alg:2}). The resulting classifiers achieve asymptotic perfect classification (Theorems~\ref{thm:data-splitting approach piling distance} and~\ref{thm:data-splitting approach SMDP}).

We further investigate exact conditions under which $\mdp$ or the projected ridged linear discriminant vector with a negative ridge parameter is a second maximal data piling direction under the heterogeneous covariance model (Section~\ref{sec:estimation ridge}). These directions were previously shown to have this property under the homogeneous covariance model with weak and strong spikes, respectively. It turns out that, while differences in the leading eigenspaces of the two classes may seem important, it is the difference in the tail eigenvalues (i.e., $\tau_1^2$ and $\tau_2^2$) that leads to the main difficulty in high-dimensional binary classification. For the case of weak spikes (i.e., $\beta_{1}, \beta_{2} < 1$), $\mdp$, which corresponds to the \textit{ridgeless} minimum-norm estimator in the context of linear regression setting, still yields second maximal data piling. We further show that a bias-corrected version of the original maximal data piling classification rule of \citet{Ahn2010} achieves asymptotic perfect classification for this case (Appendix~\ref{app:weak spikes} in the supplemental materials). However, for the case of strong spikes 
with equal tail eigenvalues (i.e., $\beta_{1} =\beta_{2} = 1$ and $\tau_{1} = \tau_{2}$), we show that the \textit{negatively} ridged linear discriminant vector, projected onto $\Sc$, can be a second maximal data piling direction even under heterogeneous covariance models. Moreover, the original projected ridge classification rule of \citet{Chang2021} achieves perfect classification in more general settings (Appendix~\ref{app:strong spikes with equal tails} in the supplemental materials).  However, for the case of strong spikes with unequal tail eigenvalues (i.e., $\beta_{1} = \beta_{2} = 1$ and $\tau_{1} \ne \tau_{2}$), the projected ridged linear discriminant vector may not be a second maximal data piling direction, or may not even yield second data piling for any ridge parameter (Section~\ref{sec:estimation ridge} and Appendix~\ref{app:strong spikes with unequal tails} in the supplemental materials). Table~\ref{table:summary} provides a summary of the linear classification rules discussed in this paper.

\begin{table}
\caption{Summary of the linear classification rules with respect to double data piling phenomenon discussed in this paper. See (\ref{eq:SMDP-I-rule}) and (\ref{eq:SMDP-II-rule}) for $\phi_{\textup{SMDP-I}}$ and $\phi_{\textup{SMDP-II}}$, respectively. The other classification rules are provided in Appendices~\ref{app:weak spikes}--\ref{app:strong spikes with unequal tails} in the supplemental materials. Classification rules that can achieve asymptotic perfect classification under the corresponding setting are marked by $\checkmark$. Classification rules that can achieve asymptotic perfect classification under the corresponding setting only with further specific conditions are marked by $\triangle$. 
}
\centering
\begin{tabular}{cccccccc}
\noalign{\smallskip}\noalign{\smallskip}
\multirow{2}{*}{Setting}  & \multicolumn{2}{c}{$0 \leq \beta_1, \beta_2 < 1$}  & \multicolumn{2}{c}{$\beta_1 = \beta_2 = 1$}  & \multicolumn{2}{c}{$0 \leq \beta_2 < \beta_1 = 1$} & \multirow{2}{*}{Reference}\\
\cline{2-7}
 & $\tau_1 = \tau_2$ & $\tau_1 \ne \tau_2$ & $\tau_1 = \tau_2$ & $\tau_1 \ne \tau_2$ & $\tau_1 = \tau_2$ & $\tau_1 \ne \tau_2$ &\\
\hline
$\phi_{\textup{SMDP-I}}$ & $\checkmark$ & $\checkmark$ & $\checkmark$ & $\checkmark$ & $\checkmark$ & $\checkmark$ & \multirow{2}{*}{Section~\ref{sec:estimation of SMDP}} \\
\cline{1-7}
$\phi_{\textup{SMDP-II}}$ & $\checkmark$ & $\checkmark$ & $\checkmark$ & $\checkmark$ & $\checkmark$ & $\checkmark$ & \\
\hline
$\phi_{\textup{MDP}}$ & $\checkmark$ & &  & & & & \multirow{2}{*}{Appendix~\ref{app:weak spikes}} \\
 \cline{1-7}
$\phi_{\textup{b-MDP}}$ & $\checkmark$ & $\checkmark$ &  & & & & \\
 \hline
$\phi_{\textup{PRD},\alpha}$ & $\checkmark$ & & $\checkmark$ & & $\checkmark$ & & Appendix~\ref{app:strong spikes with equal tails} \\
 \hline
$\phi_{\textup{b-PRD},\alpha}$ & $\checkmark$ & $\checkmark$ & $\checkmark$ & $\triangle$ & $\checkmark$ & $\checkmark$ & Appendix~\ref{app:strong spikes with unequal tails} \\
\end{tabular}
\label{table:summary}
\end{table}


        


                

\subsection{Related Works}
There has been relatively scarce work on binary classification problems with heterogeneous covariances, 
which are potentially more applicable to real-world data.
\citet{aoshima2019} proposed a distance-based classifier, while \citet{ishii2022} proposed geometrical quadratic discriminant analysis for this problem. Both assume not only the dimension of data $p$ but also training sample sizes of each class $n_{1}$ and $n_{2}$ tend to infinity to achieve perfect classification. \citet{ishii2020} proposed another distance-based classifier which achieves perfect classification even when $n_1$ and $n_2$ are fixed, but limited to the one-component covariance model (with $m_{1} = m_{2} = 1$). All of these works are based on a data transformation technique, which essentially projecting the independent test data onto the nullspace of the leading eigenspace. Our results are also based on a similar idea of removing the leading eigenspace, but 
we reveal the relationship between the maximal data piling direction of training data and the second maximal data piling direction of independent test data. Recently, it has been shown that double data piling phenomenon also occurs in multi-category classification problems for homogeneous covariance models \citep{kim2024double}. In this work, we focus on binary classification problems for generalized heterogeneous covariance models.

The benign overfitting phenomenon has been extensively studied in linear regression models \citep{Bartlett2020, Holzmuller2020, Hastie2022}. \citet{Kobak2020} showed that the optimal ridge parameter can be zero or negative when a one-component covariance model is assumed in the overparameterized regime, and \citet{Tsigler2020} further showed that negative regularization can achieve smaller generalization error than nearly zero regularization under specific spiked covariance models. \citet{Wu2020} also provided general conditions for which the optimal ridge parameter is negative in the overparameterized regime. Our findings are consistent with the above results concerning linear regression, and provide conditions when negative regularization is needed in the context of linear classification. 

\subsection{Organization}
The rest of this paper is organized as follows. In Section~\ref{sec:model}, we define the generalized heterogeneous spiked covariance models. In Section~\ref{sec:test data piling}, we characterize the second data piling phenomenon under the heterogeneous covariance models for $\beta_{1} = \beta_{2} = 1$ (Discussions for $\beta_{1} < 1$ or $\beta_{2} < 1$ are given in the supplemental materials). In Section~\ref{sec:estimation of SMDP}, 
we propose SMDP algorithms to estimate an SMDP direction. In Section~\ref{sec:estimation ridge}, we reveal exact conditions under which the projected ridged linear discriminant vector with a negative ridge parameter yields second data piling. In Section~\ref{sec:numerical studies}, we empirically evaluate the performance of the SMDP algorithms and classification rules based on the projected ridged linear discriminant vector through simulation studies and a real data example. Technical details, case-by-case discussions of the second data piling phenomenon, and proofs of the main theorems are provided in the supplemental materials.

\section{Heterogeneous Covariance Models}\label{sec:model}
We assume that for $k = 1, 2$, $X|\pi(X)=k$ follows an absolutely continuous distribution on $\Real^{p}$ with mean $\muv_{(k)}$ and covariance matrix $\Sigmav_{(k)}$. Also, we assume $\mathbb{P}(\pi(X) = k) = \pi_{k},$ where $\pi_{k} > 0 $ and $\pi_{1} + \pi_{2} = 1$. Write the eigen-decomposition of $\Sigmav_{(k)}$ by $\Sigmav_{(k)}=\Uv_{(k)}\Lambdav_{(k)}\Uv_{(k)}^\top$, where $\Lambdav_{(k)}=\diag{(\lambda_{(k),1},\ldots,\lambda_{(k),p})}$ in which the eigenvalues are arranged in descending order, and $\Uv_{(k)}=[\uv_{(k),1},\ldots,\uv_{(k),p}]$ for $k = 1, 2$. We make the following assumptions for generalized heterogeneous spiked covariance models.

\begin{Assumption}\label{assume:1}
For the population mean difference vector $\muv = \muv_{(1)} - \muv_{(2)}$, there exists $\delta > 0$ such that $p^{-1/2}\norm{\muv} \to \delta$ as $p\to\infty$.
\end{Assumption}

\begin{Assumption}\label{assume:2}
For a fixed integer $m_{k} \ge 0$, $\sigma_{k,i}, \tau_{k,i} > 0$ and $\beta_{k} \in [0, 1]$ $(k = 1, 2)$, eigenvalues $\lambda_{(k),i} = \sigma_{k, i}^2 p^{\beta_{k}} + \tau_{k,i}^2$ for $1 \leq i \leq m_{k}$ and $\lambda_{(k),i} = \tau_{k,i}^2$ for $m_{k}+1 \leq i \leq p$. Also, $\{\tau_{k,i}^2 : k = 1, 2,~i = 1, 2, \ldots \}$ is uniformly bounded and $p^{-1}\sum_{i=1}^{p} \tau_{k,i}^2 \rightarrow \tau_{k}^2$ as $p \to \infty$ for some $\tau_{k}^2 > 0$. Without loss of generality, we assume $\tau_1 \ge \tau_2$.
\end{Assumption}

Assumption~\ref{assume:1} ensures that nearly all variables are meaningfully contributing to discrimination \citep{Hall2005, Qiao2009, jung2018}. Assumption~\ref{assume:2} allows heterogeneous covariance matrices for different classes, including the homogeneous case, that is, $\Sigmav_{(1)} = \Sigmav_{(2)}$. We assume for $k = 1, 2$, $\Sigmav_{(k)}$ has $m_{k}$ spikes, that is, $m_{k}$ eigenvalues increase at the order of $p^{\beta_{k}}$ as $p \to \infty$ while the other eigenvalues are nearly constant as $\tau_{k}^2$. 
We call the first $m_{k}$ eigenvalues and their corresponding eigenvectors \emph{leading} eigenvalues and eigenvectors of the $k$th class for $k = 1, 2$. Note that if $\beta_{k} > 1$, then $\Sigmav_{(k)}$ has extremely strong signals within the leading eigenspace, and thus the classification problem becomes trivial \citep{Jung2009a}. Hence, we pay attention to the cases of $\beta_{1}, \beta_{2} \in [0, 1]$. \citet{aoshima2019} found that real high-dimensional data often follow the strongly spiked eigenvalue model in which $\beta_k \ge 1/2$. In particular, we assume $\beta_{1} = \beta_{2} = 1$ throughout the main article, which is the most interesting and challenging setting. 
The other cases where $\beta_{1} < 1$ or $\beta_{2} < 1$ are addressed in the supplemental materials. 

\begin{remark}\label{rmk:asymptotic regime}
    Assumption~\ref{assume:2} may be relaxed so that the first $m_{k}$ eigenvalues have different orders of magnitude, for example, for some $N_{k} > 0$, we may assume $\lambda_{(k),i} = \sigma_{k, i}^2 p^{\beta_{k,i}} + \tau_{k,i}^2$ for $1 \ge \beta_{k,1} \ge \ldots \ge \beta_{k,N_{k}} \ge 0$. However, asymptotic results under this relaxed model are equivalent with the results under Assumption~\ref{assume:2} with $m_{k} = \sum_{i=1}^{N_k} \1v(\beta_{k,i} = 1)$ under the HDLSS asymptotic regime \citep[see also][Remark~2.1]{Chang2021}. 
\end{remark}

We write a $p \times m_k$ orthogonal matrix of leading eigenvectors of each class as $\Uv_{(k),1} = [\uv_{(k),1}, \ldots, \uv_{(k), m_{k}}]$ for $k = 1, 2$. 
Let $\Uc$ be the subspace spanned by the leading eigenvectors from both classes, that is,
\begin{equation}
    \begin{aligned}
        \Uc := {\rm span}([\Uv_{(1),1}, \Uv_{(2),1}]).
    \end{aligned}
\end{equation}
We call $\Uc$ the common leading eigenspace of both classes. We assume that the dimension of $\Uc$, 
\begin{align*}
    m := \dim{(\Uc)},
\end{align*}
is a fixed constant for all $p$. 
Note that $\max{(m_{1}, m_{2})} \leq m \leq m_{1}+m_{2}$. In particular, if $m_{1} = m_{2} = 0$, then $\Uc = \0v_p$ and $m = 0$. As indicated in Remark~\ref{rmk:asymptotic regime}, asymptotic results with $m_{k} = 0$ and $\beta_k = 1$ ($k = 1, 2$) are equivalent to those under the weak spikes model ($\beta_{k} < 1$). 

We assume that training dataset $\Xc$ consists of $X_{k1}, \ldots, X_{kn_{k}}$, which are independent and identically distributed random variables where $\pi(X_{kj}) = k$ for $k = 1, 2$ and $1 \leq j \leq n_k$. Let $\Xv_{k} = [X_{k1},\ldots,X_{kn_k}]$ for $k = 1, 2$. We assume $n_{1}$ and $n_{2}$ are fixed and denote $\eta_k = n_k/n$ for $k = 1, 2$ where $n = n_1+n_2$. We write class-wise sample mean vectors $\bar{X}_k = {n_k}^{-1}\sum_{j=1}^{n_k} X_{kj}$ and the total sample mean vector $\bar{X} = \eta_1 \bar{X}_1 + \eta_2 \bar{X}_2$. Also, the within-class scatter matrix is $\Sv_W = \Sv_1 + \Sv_2$, where, for $k = 1, 2$, $\Sv_k = (\Xv_k - \bar{\Xv}_k)(\Xv_k - \bar{\Xv}_k)^\top $ and $\bar{\Xv}_k = \bar{X}_k \1v_{n_k}^\top$. We write an eigen-decomposition of $\Sv_W$ by $\Sv_W = \hat\Uv\hat\Lambdav\hat\Uv^\top$, where $\hat\Lambdav=\diag{(\hat\lambda_{1},\ldots,\hat\lambda_{p})}$ in which the eigenvalues are arranged in descending order, and $\hat\Uv=[\hat\uv_1, \ldots, \hat\uv_{p}]$. Since $\hat\lambda_1 \ge \cdots \ge \hat\lambda_{n-2} \ge \hat\lambda_{n-1} = \cdots = \hat\lambda_{p} = 0$ with probability $1$, we can write $\Sv_W = \hat\Uv_1\hat\Lambdav_{1}\hat\Uv_1^\top$ where $\hat\Uv_1 = [\hat\uv_1, \ldots, \hat\uv_{n-2}]$ and $\hat\Lambdav_{1} = \diag{(\hat\lambda_1, \ldots, \hat\lambda_{n-2})}$. Also, we write $\hat\Uv_2 = [\hat\uv_{n-1}, \ldots, \hat\uv_{p}]$. We denote the sample space as $\Sc_{\Xc}$, which is the $(n-1)$-dimensional subspace spanned by $X_{kj} - \bar{X}$ for $k = 1, 2$ and $1 \leq j \leq n_k$. Note that the sample space $\Sc_{\Xc}$ can be equivalently expressed as ${\rm span}(\hat\uv_1, \ldots, \hat\uv_{n-2}, \mdp)$ \citep{Ahn2010, Chang2021}. We denote the sample mean difference vector as $\dv = \bar{X}_1 - \bar{X}_2$. 

We regulate the dependency of the true principal components $\zv_{kj} = \Lambdav_{(k)}^{-1/2}\Uv_{(k)}^\top(X_{kj} - \muv_{(k)}) \in \Real^{p}$ ($k = 1, 2$, $1 \leq j \leq n_k$) by utilizing the $\rho$-mixing condition. This allows us to make use of the law of large numbers applied to $p \to \infty$ as done in \citet{Hall2005} and \citet{Jung2009a}. We refer the reader to \citet{Kolmogorov1960} and \citet{Bradley2005} for a detailed explanation of the $\rho$-mixing condition. 

\begin{Assumption}\label{assume:3}
The elements of the $p$-vector $\zv_{kj} $ have uniformly bounded fourth moments, and for each $p$, $\zv_{kj}$ consists of the first $p$ elements of an infinite random sequence $$(z_{k,1},z_{k,2},...)_{j},$$ which is $\rho$-mixing under some permutation.
\end{Assumption}

We define ${\rm Angle}(\wv_1, \wv_2) := \arccos\{\wv_1^\top \wv_2 / \left(\|\wv_1\|_2 \|\wv_2 \|_2\right)\}$ for $\wv_1, \wv_2 \in \Real^p  \setminus \left\{ \0v_p \right\}$. For $\wv \in \Real^p  \setminus \left\{ \0v_p \right\}$ and a subspace $\Vc$ of $\Real^p$, let $P_{\Vc}\wv$ be the orthogonal projection of $\wv$ onto $\Vc$ and define ${\rm Angle}(\wv, \Vc) : = \arccos\{\wv^\top P_{\Vc}\wv / \left(\|\wv\|_2 \|P_{\Vc}\wv \|_2\right)\}$. We also assume that limiting angles between leading eigenvectors and the population mean difference vector $\muv$ exist as $p \to \infty$. Related technical assumptions are deferred to Appendix~\ref{app:assumptions} in the supplemental materials (see Assumptions~\ref{assume:4} and~\ref{assume:5} therein). Let 
\begin{align*}
    \varphi := \lim_{p \to \infty} {\rm Angle}(\muv, \Uc) \in (0, \pi/2]
\end{align*}
denote the limiting angle between $\muv$ and $\Uc$. Throughout, we assume that $\varphi \ne 0$, that is, we do not consider the case where $\muv$ lies within the common leading eigenspace $\Uc$. Lastly, we use the convention that if the dimension $m = 0$, then $\varphi = \pi/2$.

\section{Data Piling of Independent Test Data}\label{sec:test data piling}
The \emph{first data piling} refers to the phenomenon where projections of training data onto a direction vector $\wv \in \Real^p$ are piled on two distinct points in binary classification \citep{Ahn2010}. The first data piling occurs whenever $p > n-2$. \citet{Ahn2010} revealed the existence of the \emph{first maximal data piling} direction, $\mdp$ in (\ref{eq:MDP}), which uniquely maximizes the distance between two piles of training data among directions exhibiting first data piling. While $\mdp$ yields the first data piling, it does not generally yield the \emph{second data piling} of independent test data (see Figure~\ref{fig:intro_figure}).

In this section, we show that independent test data, projected onto a low-dimensional signal subspace $\Sc$ within the sample space $\Sc_{\Xc}$, tend to be respectively distributed along two affine subspaces as $p$ increases. It should be emphasized that the second data piling phenomenon occurs asymptotically as $p \to \infty$, while the first data piling phenomenon occurs for any fixed $p > n-2$. For the homogeneous case ($\Sigmav_{(1)} = \Sigmav_{(2)}$), \citet{Chang2021} showed that there are two affine subspaces, each with dimension $m = m_1 = m_2$, such that they are parallel to each other. We show that if $\Sigmav_{(1)} \ne \Sigmav_{(2)}$, these affine subspaces are not in general parallel to each other, but there exist parallel affine subspaces, of greater dimension, containing each of these affine subspaces. 

In Section~\ref{sec:test data piling examples}, we illustrate this phenomenon under different conditions on the tail eigenvalues of the covariance matrices. In Section~\ref{sec:test data piling signal subspace}, we formally characterize the signal subspace $\Sc$, which captures important variability of independent test data, for each scenario of covariance matrices. In Section~\ref{sec:test data piling main Theorem}, we provide the main theorem (Theorem~\ref{thm:test data piling}) that characterizes the second data piling phenomenon of independent test data under generalized heterogeneous spiked covariance models. Throughout, let $\Yc_k$ be an independent test data of the $k$th class whose element $Y \in \Yc_k$ satisfies $\pi(Y) = k$ for $k = 1, 2$ and is independent to training data $\Xc$. Write $\Yc = \Yc_1 \cup \Yc_2$.

\subsection{Illustration of Second Data Piling}\label{sec:test data piling examples}
In this subsection, we illustrate the phenomenon of data pile of independent test data. 
As a simple example, consider the one-component covariance model (i.e., $m_1 = m_2 = 1$) as follows:
\begin{equation}\label{eq:one-comp model}
    \begin{aligned}
    &\Sigmav_{(1)} = \sigma_{1,1}^2 p \uv_{(1),1}\uv_{(1),1}^\top + \tau_1^2 \Iv_{p}; \\
    &\Sigmav_{(2)} = \sigma_{2,1}^2 p \uv_{(2),1}\uv_{(2),1}^\top + \tau_2^2 \Iv_{p}.
    \end{aligned}
\end{equation}
Note that, if $\sigma_{1,1} = \sigma_{2,1}$, $\uv_{(1),1} = \uv_{(2),1}$ and $\tau_1 = \tau_2$, then this model corresponds to the homogeneous covariance model of Figure~\ref{fig:intro_figure}. 

We will see that how differences in the magnitude of signals ($\sigma_{1,1}^2$ and $\sigma_{2,1}^2$), the leading eigenvectors ($\uv_{(1),1}$ and $\uv_{(2),1}$) and the tail eigenvalues ($\tau_1^2$ and $\tau_2^2$) affect the second data piling phenomenon. Among these, the magnitude of signals makes the least distinction between the homogeneous and heterogeneous cases with respect to the second data piling phenomenon. For example, suppose that $\uv_{(1),1} = \uv_{(2),1}$ and $\tau_1 = \tau_2$ but $\sigma_{1,1} \ne \sigma_{2,1}$. Then, similarly in Figure~\ref{fig:intro_figure} where $\Sigmav_{(1)} = \Sigmav_{(2)}$, it can be checked that projections of independent test data $\Yc$ onto 
$\Sc = {\rm span}(\hat\uv_1, \mdp)$ 
tend to be distributed along two parallel lines $\Lc_1$ and $\Lc_2$, while those of training data $\Xc$ are piled on two distinct points along $\mdp$. The only difference is that the spread of the independent test data $\Yc_k$ along $\Lc_k$ depends on the value of $\sigma_{k,1}$ for $k = 1, 2$. 

Hence, we pay close attention to the impact of differences in the leading eigenvectors and the tail eigenvalues on the second data piling phenomenon. 
In particular, we examine the case of heterogeneous leading eigenvector ($\uv_{(1),1} \ne \uv_{(2),1}$) under the model (\ref{eq:one-comp model}), while setting $\tau_1 = \tau_2$ in Example~\ref{ex:strong_eq_2} and $\tau_1 \ne \tau_2$ in Example~\ref{ex:strong_ne_2}. 
We will see that these examples are quite different from the homogeneous case in Figure~\ref{fig:intro_figure} and also from one another. 


\begin{example}[$\uv_{(1),1} \ne \uv_{(2),1}$ and $\tau_{1} = \tau_2$]\label{ex:strong_eq_2}

\begin{figure}
    \centering
    \includegraphics[width=1\linewidth]{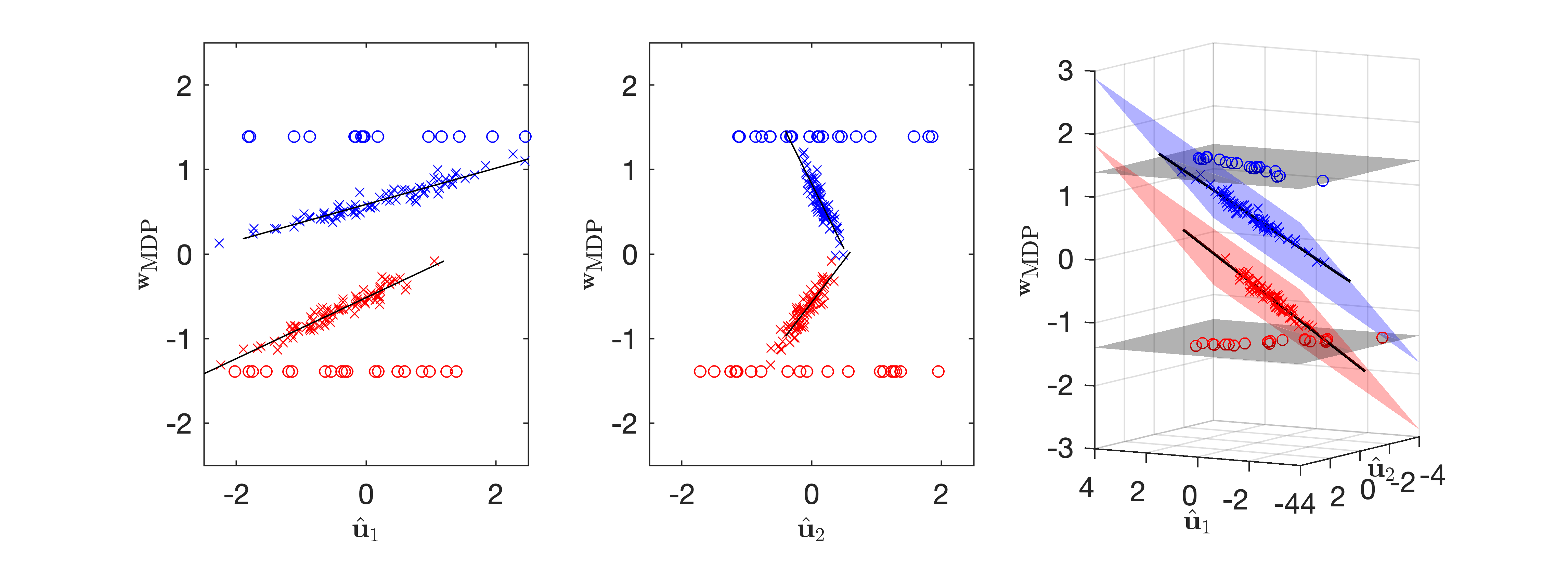}
    \caption{
    Projections of training data $\Xc$ and independent test data $\Yc$ under the model in Example~\ref{ex:strong_eq_2}, onto  $\Sc_1 = {\rm span}(\hat\uv_1, \mdp)$, $\Sc_2 = {\rm span}(\hat\uv_2, \mdp)$ and $\Sc = {\rm span}(\hat\uv_1, \hat\uv_2, \mdp)$. Second data piling occurs along the direction orthogonal to $\Lc_1$ and $\Lc_2$ (blue and red planes). 
    }
    \label{fig:strong_eq_2}
\end{figure}

In this case, the angle between the sample eigenvector $\hat\uv_i$ and the common leading eigenspace $\Uc = {\rm span}(\uv_{(1),1}, \uv_{(2),1})$ converges to a random quantity between $0$ and $\pi/2$ for $i = 1, 2$. In contrast, the other sample eigenvectors are strongly inconsistent with $\Uc$ in the sense that ${\rm Angle}(\hat\uv_i, \Uc) \xrightarrow{P} \pi / 2$ as $p \to \infty$ for $3 \leq i \leq n-2$ (Detailed explanations are given in Section~\ref{sec:test data piling signal subspace} and Appendix~\ref{app:asymptotic properties of sample covariance matrix} of the supplemental materials).
In Figure~\ref{fig:strong_eq_2}, independent test data $\Yc$ projected onto $\Sc_1 = {\rm span}(\hat\uv_1, \mdp)$ and $\Sc_2 = {\rm span}(\hat\uv_2, \mdp)$ are also concentrated along lines, but in both subspaces these lines are not parallel to each other (unlike in the homogeneous case in Figure~\ref{fig:intro_figure}). However, within the $3$-dimensional subspace 
$\Sc = {\rm span}(\hat\uv_1, \hat\uv_2, \mdp)$, 
there are two parallel $2$-dimensional planes $(\Lc_1, \Lc_2)$ that include these lines, one for each line. In fact, $\Yc_1$ is distributed along the direction $P_{\Sc}\uv_{(1),1}$, while $\Yc_2$ is distributed along the direction $P_{\Sc}\uv_{(2),1}$. Thus, these lines are asymptotically contained in $2$-dimensional affine subspaces that are parallel to $P_{\Sc}\Uc := {\rm span}(P_{\Sc}\uv_{(1),1},  P_{\Sc}\uv_{(2),1})$.
\end{example}

In the next example in which $\tau_1 \ne \tau_2$, we observe that asymptotic properties of eigenvectors of $\Sv_W$ are different from the case of $\tau_1 = \tau_2$, and that unequal tail eigenvalues affects the behavior of data piling of independent test data.

\begin{example}[$\uv_{(1),1} \ne \uv_{(2),1}$ and $\tau_{1} \ne \tau_2$]\label{ex:strong_ne_2}

\begin{figure}[tp]
    \centering
    \includegraphics[width=0.8\linewidth]{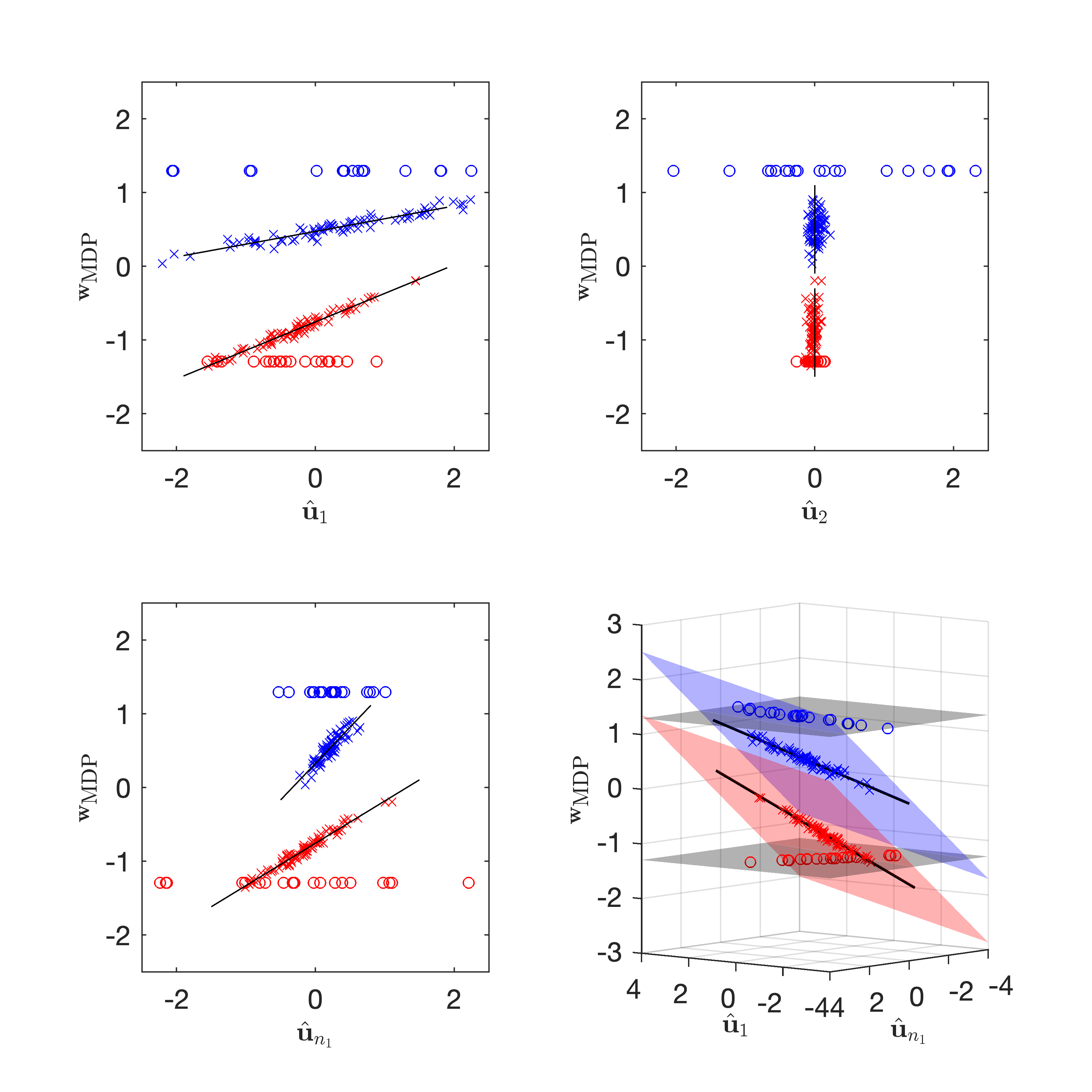}
    \caption{
    Projections of training data $\Xc$ and independent test data $\Yc$ under the model in Example~\ref{ex:strong_ne_2}, onto  $\Sc_1 = {\rm span}(\hat\uv_1, \mdp)$, $\Sc_2 = {\rm span}(\hat\uv_2, \mdp)$, 
    $\Sc_{n_1} = {\rm span}(\hat\uv_{n_1}, \mdp)$ and $\Sc_{1,n_1} = {\rm span}(\hat\uv_1, \hat\uv_{n_1}, \mdp)$. Second data piling occurs along the direction orthogonal to $\Lc_1$ and $\Lc_2$ (blue and red planes) in $\Sc_{1,n_1}$.
    %
    %
    }
    \label{fig:strong_ne_2}
\end{figure}
Without loss of generality, assume $\tau_1 > \tau_2$.
Similarly to Example~\ref{ex:strong_eq_2}, $\hat\uv_1$ captures the largest variation within the common leading eigenspace $\Uc$ from the data. 
In the equal tail case (Example~\ref{ex:strong_eq_2}), $\hat\uv_2$ always captures the remaining variation in $\Uc$. However, for the unequal tail case, the eigenvector that captures the remaining variation could be $\hat\uv_{n_1}$ instead of $\hat\uv_2$. 
This situation---$\Sc={\rm span}(\hat\uv_1, \hat\uv_{n_1}, \mdp)$ (rather than ${\rm span}(\hat\uv_1, \hat\uv_{2}, \mdp)$) captures the significant variations of independent test data $\Yc$---is displayed in Figure~\ref{fig:strong_ne_2}. 

Perhaps surprisingly, which eigenvector, $\hat\uv_2$ or $\hat\uv_{n_1}$, captures the variation is decided at random. 
%
%
%
%
To understand this phenomenon, recall the geometric representation of HDLSS data: \citet{Jung2012a} showed that
HDLSS data from a strongly spiked covariance model (that is, $\beta_k = 1$) can asymptotically be decomposed into random and deterministic parts; the random variation remains in ${\rm span}(\uv_{(k),1})$, while the deterministic simplex structure with edge length $\tau_k \sqrt{p}$ remains in the orthogonal complement of ${\rm span}(\uv_{(k),1})$. For sufficiently large $p$, $\hat\uv_1$ explains the most important variation within $\Uc = {\rm span}(\uv_{(1),1}, \uv_{(2),1})$ in the data from both classes. If the remaining variation within $\Uc$ in the data from both classes is larger than $\tau_1^2p$, then this variation is captured by $\hat\uv_{2}$, while $\hat\uv_{3}, \ldots, \hat\uv_{n_1}$ (and $\hat\uv_{n_1+1}, \ldots, \hat\uv_{n-2}$) explain the deterministic simplex of data with edge length $\tau_1\sqrt{p}$ (and $\tau_2\sqrt{p}$) from the first (and second) class, respectively. In contrast, if the remaining variation is smaller than $\tau_1^2p$, then the roles of $\hat\uv_{2}$ and $\hat\uv_{n_1}$ are reversed, thus $\hat\uv_{n_1}$ explains the remaining variation. In Section~\ref{sec:test data piling signal subspace} and Appendix~\ref{app:asymptotic properties of sample covariance matrix} of the supplemental materials, we will see that this variation may be either larger or smaller than $\tau_1^2p$ depending on the true leading principal components scores of training data $\Xc$. 
\end{example}

\begin{remark}
While we have presented illustrations only for the case $\uv_{(1),1} \ne \uv_{(2),1}$ in this subsection for brevity and clarity, the case $\uv_{(1),1} = \uv_{(2),1}$ but $\tau_1 \ne \tau_2$ is also illustrated in Figure~\ref{fig:one_comp_m=1} in Section~\ref{sec:SMDP theory}, which is used in providing insight into understanding the theoretical second maximal data piling direction. 
\end{remark}

\subsection{Characterization of the Signal Subspace $\Sc$}\label{sec:test data piling signal subspace}
In Section~\ref{sec:test data piling examples}, we have observed that there exists a low-dimensional `signal' subspace $\Sc$ such that projections of independent test data $\Yc$ onto $\Sc$ tend to lie on parallel affine subspaces, one for each class. The subspace $\Sc$ depends on the situations, and was ${\rm span}(\hat\uv_1, \mdp)$ in Figure~\ref{fig:intro_figure}, ${\rm span}(\hat\uv_1, \hat\uv_2, \mdp)$ in Example~\ref{ex:strong_eq_2}, and either ${\rm span}(\hat\uv_1, \hat\uv_2, \mdp)$ or ${\rm span}(\hat\uv_1, \hat\uv_{n_1},\mdp)$ in Example~\ref{ex:strong_ne_2}. 

We now formally characterize the signal subspace $\Sc$ for general cases. For this, we will utilize the index set $\Dc \subset \left\{1, \ldots, n-2\right\}$ collecting the indices of the sample eigenvectors of $\Sv_W$ that explain variations contained in the common leading eigenspace $\Uc$. In the examples mentioned above, $\Dc = \{1\}$, $\{1,2\}$, and either $\{1,2\}$ or $\{1,n_1\}$, respectively.


\begin{table}
\caption{The index set $\Dc$ for the strong spike case ($\beta_1 = \beta_2 = 1$). The random number $k_0$ depends on the true principal component scores of training data $\Xc$, and is defined in Lemma~\ref{lem:asymptotic property of Sw unequal tails} in Appendix~\ref{app:asymptotic properties of sample covariance matrix} of the supplemental materials. See Appendix~\ref{app:strong and weak spikes} in the supplemental materials for cases with weak spikes.}
\centering
\begin{tabular}{c|cc}
\noalign{\smallskip}\noalign{\smallskip}
$\tau_1, \tau_2$ & {$\Dc$} & {$|\Dc|$}\\
\hline
 {$\tau_1 = \tau_2$} & {$\left\{1, \ldots, m \right\}$} & {$m$} \\ \hline
  {$\tau_1 > \tau_2$} & {$\left\{1, \ldots, k_0, k_0+(n_1-m_1), \ldots, n_1+m_2-1  \right\}$} & {$m_1+m_2$} \\
\end{tabular}
%
\label{table:D}
\end{table}

As demonstrated in Section~\ref{sec:test data piling examples} for single-spike scenarios, the asymptotic behavior of the sample eigenvectors of $\Sv_W$ varies significantly depending on whether the tail eigenvalues are equal or differ. In the more general case with $m_1$ and $m_2$ spikes, these differences in asymptotic behavior lead to distinct characterizations of $\Dc$, which are summarized in Table~\ref{table:D}.
%
Detailed derivations of this result are given in Appendix~\ref{app:asymptotic properties of sample covariance matrix} of the supplemental materials. Here, we provide a summary of some important observations for strong spike models. 
\begin{itemize}
    \item If $\tau_1 = \tau_2$, then only the first $m$ leading eigenvectors of $\Sv_W$ can explain the variation within the common leading eigenspace $\Uc$ (see Example~\ref{ex:strong_eq_2} for the single spike case). 

    \item If $\tau_1 > \tau_2$, then there are exactly $m_1+m_2$ sample eigenvectors of $\Sv_W$ that explain the variation within the common leading eigenspace $\Uc$, regardless of the value of $m$. While the first $m_1$ sample eigenvectors are always capable of explaining this variation, some non-leading eigenvectors (after the $n_1$th) may capture variability within $\Uc$ that could otherwise be explained by other leading eigenvectors beyond the $m_1$th. That is, $k_0$ is a random number between $m_1$ and $m_1+m_2$. In particular, if $m = m_1$, then $k_0$ in Table~\ref{table:D} is $m_1$ with probability $1$, that is, $\Dc$ always consists of the first $m_1$ leading eigenvectors and $m_2$ non-leading eigenvectors 
    (e.g., see Figure~\ref{fig:one_comp_m=1} in Section~\ref{sec:SMDP theory} for the single spike case). In general, $k_0$ is a random number depending on true leading principal component scores of $\Xc$, which is difficult to estimate in practice.
    Under the single spike model in Example~\ref{ex:strong_ne_2}, $\mathbb{P}(k_0 = i)$ $(i = 1, 2)$ depends on the magnitude of signals $\sigma_{1,1}^2$ and $\sigma_{2,1}^2$, the tail eigenvalues $\tau_1^2$ and $\tau_2^2$, and the angle between $\uv_{(1),1}$ and $\uv_{(2),1}$. See Figure~\ref{fig:k0plot} for an illustration. In particular, when $\uv_{(1),1} = \uv_{(2),1}$, then $k_0=1$ and $\Dc = \{1, n_1 \}$ almost surely.

    \begin{figure}
        \centering
        \includegraphics[width=1\linewidth]{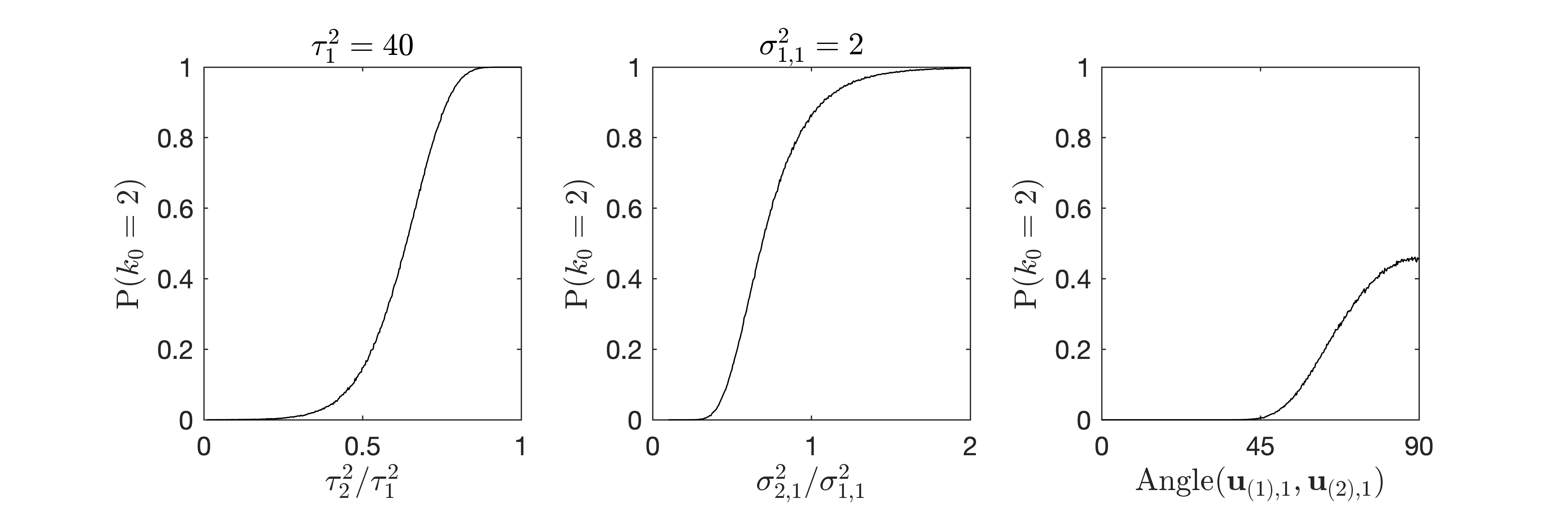}
        \caption{Average estimates of $\mathbb{P}(k_0 = 2)$, the probability that $\hat\uv_2$ captures variability within $\Uc$ instead of $\hat\uv_{n_1}$ under the one-component model in Example~\ref{ex:strong_ne_2} and Figure~\ref{fig:strong_ne_2}. $\mathbb{P}(k_0 = 2)$ is empirically estimated under the Gaussian assumption with $(\sigma_{1,1}^2, \sigma_{2,1}^2) = (2, 1)$, $(\tau_1^2, \tau_2^2) = (40, 20)$, ${\rm Angle}{(\uv_{(1),1}, \uv_{(2),1})} = \pi/4$ and $(n_1, n_2) = (20, 20)$. The panels show the average estimates with different $\tau_2$ (left), $\sigma_{2,1}$ (middle), ${\rm Angle}{(\uv_{(1),1}, \uv_{(2),1})}$ (right) while the other parameters are fixed, respectively.}
        \label{fig:k0plot}
    \end{figure}
\end{itemize}

In general, we define the signal subspace $\Sc$ as 
\begin{align}\label{eq:S}
\Sc := {\rm span}(\left\{\hat\uv_i \right\}_{i \in \Dc}, \mdp),
\end{align}
where $\Dc$ is given in Table \ref{table:D}. The signal subspace $\Sc$ is a low-dimensional subspace of the sample space $\Sc_{\Xc}$, and is obtained by removing  noisy directions in $\Sc_{\Xc}$. In case of $m = 0$, we use the convention of $\Dc = \emptyset$, that is, $\Sc = {\rm span}(\mdp)$.

\subsection{Main Theorem}\label{sec:test data piling main Theorem}
We formally establish that projections of $\Yc$ onto the signal subspace $\Sc$, defined in (\ref{eq:S}), are distributed along parallel affine subspaces, one for each class, and that these affine subspaces do not overlap. Write the scaled training data piling distance as
\begin{align}\label{eq:kMDP}
    \kappa_{\textup{MDP}} := p^{-1/2}\|\mdp^\top (\bar{X}_1 - \bar{X}_2) \|_2.
\end{align} 
For $Y \in \Yc$ and a subspace $\Sc$ of $\Real^p$, let $Y_{\Sc} = p^{-1/2} P_{\Sc}Y$, which is a scaled projection of $Y$ onto $\Sc$. Similarly, write ${\bar{X}}_{\Sc} = p^{-1/2} P_{\Sc}\bar{X}$. Lastly, write an orthogonal basis of the common leading eigenspace $\Uc$ as $\uv_1, \ldots, \uv_m$, and set $\Uv_{1,\Sc} = [\uv_{1,\Sc}, \ldots, \uv_{m, \Sc}]$ where $\uv_{i,\Sc} = P_{\Sc}\uv_i$ for $i = 1, \ldots, m$.

\begin{theorem}\label{thm:test data piling}
Suppose Assumptions~\ref{assume:1}---\ref{assume:5} hold. Let $\Sc = {\rm span}(\left\{ \hat\uv_i \right\}_{i \in \Dc}, \mdp)$, where $\Dc$ is defined in Table~\ref{table:D}. Let
\begin{align*}
   \Lc_k := \left\{ \Uv_{1, {\Sc}} \tv + \nu_k \mdp + \bar{X}_{\Sc} : \tv \in \Real^{m} \right\} 
\end{align*}
for $k = 1, 2$ where $\nu_1 = \kappa_{\textup{MDP}}^{-1}(\eta_2(1-\cos^2\varphi)\delta^2 - (\tau_1^2 - \tau_2^2)/n )$ and $\nu_2 = \kappa_{\textup{MDP}}^{-1}(-\eta_1(1-\cos^2\varphi)\delta^2 - (\tau_1^2 - \tau_2^2)/n)$. Then for any independent observation $Y \in \Yc$ and for any $\epsilon > 0$,
\begin{align*}
\lim\limits_{p \to \infty} \mathbb{P} \left( \inf\limits_{a \in \Lc_k} \|Y_{\Sc} - a \| > \epsilon | \pi(Y) = k \right) = 0
\end{align*}
for $k = 1, 2$. 
\end{theorem}

We remark that, in general, the test data from class 1, $\Yc_1$ projected onto $\Sc$, are distributed along an $m_1$-dimensional affine subspace $\Lc_1'$. Similarly, projections of $\Yc_2$ are distributed along an $m_2$-dimensional affine subspace $\Lc_2'$. 
For each $k = 1, 2$, 
$\Lc_k'$ is parallel to ${\rm span}([P_{\Sc}\uv_{(k),1}, \ldots, P_{\Sc}\uv_{(k),m_k}])$, the subspace spanned by the leading eigenvectors of the class $k$, projected onto the signal subspace $\Sc$. While $\Lc_1'$ and $\Lc_2'$ are not necessarily parallel to each other, each of them are contained in the $m$-dimensional affine subspace $\Lc_k$ (appeared in Theorem~\ref{thm:test data piling}), respectively. 
We provide case-by-case discussions to offer a clearer understanding of Theorem~\ref{thm:test data piling}.

\begin{itemize}
    \item (Weak spikes) If $m = 0$, then Theorem~\ref{thm:test data piling} remains valid with 
    \begin{align*}
        \Sc = {\rm span}(\mdp)
    \end{align*}
    and $\Lc_k = \{ \nu_k \mdp + \bar{X}_{\Sc}\}$ for $k = 1, 2$ and $\cos^2\varphi = 0$. That is, projections of $\Yc_1$ and $\Yc_2$ are asymptotically piled on two distinct points on $\mdp$, one for each class. Note that the piling locations of $\Yc_1$ and $\Yc_2$ do not coincide with each other, nor with those of $\Xc$. Also, $\mdp$ yields the \emph{second} data piling for both cases of equal and unequal tail eigenvalues ($\tau_1 = \tau_2$ and $\tau_1 \neq \tau_2$). In Appendix~\ref{app:weak spikes} of the supplemental materials, we establish that $\mdp$ is indeed a second maximal data piling direction for this case. There, we also define a bias-corrected maximal data piling classification rule $\phi_{\textup{b-MDP}}$ based on $\mdp$ that achieves asymptotic perfect classification regardless of whether $\tau_1 = \tau_2$ or $\tau_1 \ne \tau_2$. 
    
    \item (Strong spikes with equal tail eigenvalues) If $m \ge 1$ and $\tau_1 = \tau_2 =: \tau$, then projections of $\Yc$ onto the $(m+1)$-dimensional subspace 
    \begin{align*}
        \Sc = {\rm span}(\hat\uv_1, \ldots, \hat\uv_m, \mdp)
    \end{align*}
    are distributed along two $m$-dimensional affine subspaces $\Lc_1$ and $\Lc_2$, which become parallel to each other, and also to $P_{\Sc}\Uc := {\rm span}(\uv_{1,\Sc}, \ldots, \uv_{m,\Sc})$, as $p$ increases. Two special cases were discussed in Figure~\ref{fig:intro_figure} (where $\uv_{(1),1} = \uv_{(2),1}$) and Example~\ref{ex:strong_eq_2} (where $\uv_{(1),1} \ne \uv_{(2),1}$) under the one-component spiked covariance model (\ref{eq:one-comp model}). Generally, $\dim\Sc - m = 1$ and there exists a unique direction within $\Sc$ which is orthogonal to both $\Lc_1$ and $\Lc_2$.

    \item (Strong spikes with unequal tail eigenvalues) If $m \ge 1$ and $\tau_1 > \tau_2$, then projections of $\Yc$ onto a carefully chosen signal space $\Sc$ are asymptotically contained in $m$-dimensional parallel affine subspaces. 
    However, in this case, $\Sc$ may not be the subspace spanned by the first $m$ eigenvectors of $\Sv_W$ and $\mdp$. As an illustrative example, consider the one-component spiked covariance model (\ref{eq:one-comp model}) where $m_1 = m_2 = 1$. For the case of $\uv_{(1),1} = \uv_{(2),1}$ (i.e., $m = 1$), $\hat\uv_{n_1}$ always explains the variation of data along the common leading eigenvector $\uv_1 = \uv_{(1),1} = \uv_{(2),1}$ while $\hat\uv_{2}$ does not. That is, for the common eigenspace case, the signal subspace is always the $3$-dimensional subspace $\Sc = {\rm span}(\hat\uv_1, \hat\uv_{n_1}, \mdp)$. In contrast, if $\uv_{(1),1}\ne \uv_{(2),1}$ (i.e., $m = 2$), then as discussed in Example~\ref{ex:strong_ne_2},
    \begin{align*}
        \Sc = \begin{cases}
        {\rm span}(\hat\uv_1, \hat\uv_{n_1}, \mdp) & \text{if~} k_0 = 1, \\
        {\rm span}(\hat\uv_1, \hat\uv_{2}, \mdp) & \text{if~} k_0 = 2, \\
        \end{cases}
    \end{align*}
    is decided at random. In general, $\dim\Sc - m = m_1+m_2-m+1 \ge 1$ and there are multiple directions within $\Sc$ which are orthogonal to both $\Lc_1$ and $\Lc_2$.
\end{itemize}

Theorem~\ref{thm:test data piling} tells that independent test data are asymptotically distributed along parallel $m$-dimensional affine subspaces $\Lc_1$ and $\Lc_2$ in $\Sc$. It implies that if we find a direction $\wv \in \Sc$ such that $\wv$ is asymptotically orthogonal to $\Lc_1$ and $\Lc_2$, then $P_{\wv}\Yc$ yields second data piling and in turn achieves perfect classification of independent test data. Since $\dim{\Sc} - m \ge 1$ for all cases, there always exists a direction $\wv \in \Sc$ which yields second data piling. Meanwhile, any direction $\wv \in \Sc_{\Xc} \setminus \Sc$ also yields second data piling since $\wv$ is asymptotically orthogonal to the common leading eigenspace $\Uc$, resulting in projections of $\Yc_1$ and $\Yc_2$ converging to the same location. Among the many second data piling directions, we will find a second maximal data piling direction, which asymptotically maximizes the distance between the two piles of independent test data.

\section{Estimation of Second Maximal Data Piling Direction}\label{sec:estimation of SMDP}
In this section, we establish that a {second maximal data piling} direction can be obtained by projecting $\mdp$ onto the orthogonal complement of ${\rm span}(\Uv_{1,\Sc})$, which is the nullspace of the common leading eigenspace $\Uc$. This insight is used in estimating the \emph{second maximal data piling} direction based on a data-splitting approach. We then provide classification rules which guarantee asymptotic perfect classification based on such direction.

\subsection{Second Maximal Data Piling Direction for General Heterogeneous Spiked Covariance}\label{sec:SMDP theory}
As a motivating example, we revisit the one-component covariance model (\ref{eq:one-comp model}). Suppose now that two classes share a common leading eigenvector (i.e., $\uv_{(1),1} = \uv_{(2),1} =: \uv_1$) but have unequal tail eigenvalues (i.e., $\tau_1 > \tau_2$). In this case, Theorem~\ref{thm:test data piling} tells that the signal subspace $\Sc$ is ${\rm span}(\hat\uv_1, \hat\uv_{n_1}, \mdp)$ and $\dim{\Sc} - m = 2$. 
It implies that there are (infinitely) many sets of two second data piling directions within $\Sc$ that are orthogonal to each other and to $\uv_{1,\Sc}$. 

Our choices (which turns out to be optimal) of two directions are obtained as follows. Recall that second data piling occurs when a direction $\wv \in \Sc$ is asymptotically orthogonal to $\uv_{1, \Sc} = (\hat\uv_{1}^\top \uv_{1})\hat\uv_1 + (\hat\uv_{n_1}^\top\uv_{1}) \hat\uv_{n_1} + (\mdp^\top \uv_{1})\mdp$. Since $\mdp$ plays a distinct role among the basis of $\Sc$, let $\fv_1 \in \Sc$ be the vector  orthogonal to both $\uv_{1,\Sc}$ and $\mdp$. 
To be specific, 
$\fv_1 \varpropto (-\hat\uv_{n_1}^\top \uv_{1})\hat\uv_1 + (\hat\uv_1^\top\uv_1)\hat\uv_{n_1}$.
The vector $\fv_0 \in \Sc$ parallel to the nullspace of $[\fv_1, \uv_{1,\Sc}]$ completes the pair. It turns out that $\fv_0$ is the orthogonal projection of $\mdp$ onto the orthogonal complement of ${\rm span}(\uv_{1,\Sc})$ within $\Sc$. 

%

\begin{figure}
    \centering \includegraphics[width=0.8\linewidth]{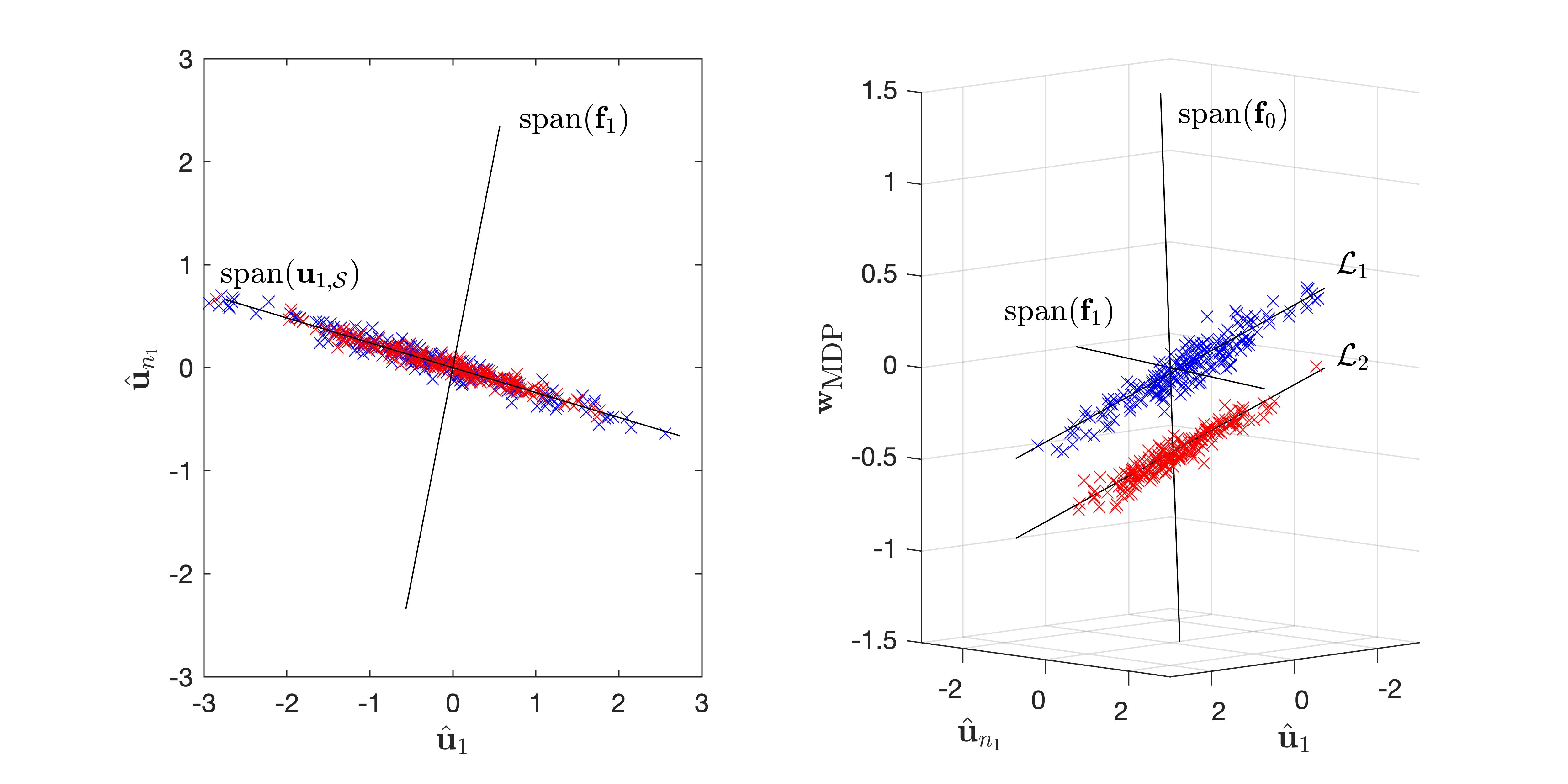}
    \caption{$2$-dimensional projections onto ${\rm span}(\hat\uv_1, \hat\uv_{n_1})$ and $3$-dimensional projections onto $\Sc = {\rm span}(\hat\uv_1, \hat\uv_{n_1}) \oplus {\rm span}(\mdp)$ of independent test data $\Yc$ under the unequal-tail model with $m_1 = m_2 = m = 1$. 
    }
    \label{fig:one_comp_m=1}
\end{figure}

The directions $\fv_0$ and $\fv_1$ are graphically displayed in Figure~\ref{fig:one_comp_m=1}. 
The left panel displays the data on ${\rm span}(\hat\uv_1, \hat\uv_{n_1})$, where the `meaningless' direction $\fv_1$ lies.
Although test data projected on $\fv_1$ exhibit second data piling, the piling gap is nearly zero. 
The right panel of Figure~\ref{fig:one_comp_m=1} shows that the direction $\fv_0 \in \Sc$, which is orthogonal to both of $\uv_{1,\Sc}$ and $\fv_1$, provides a meaningful distance between the two piles of independent test data $\Yc$. It is graphically suggested that $\fv_0$ is a second maximal data piling direction. 

We extend the above discussions to the general cases $m_1 \ge 1$ and $m_2 \ge 1$ and investigate a theoretical second maximal data piling direction. We decompose the signal subspace $\Sc = {\rm span}(\{\hat\uv_i \}_{i \in \Dc}) \oplus {\rm span}(\mdp)$ into three orthogonal parts: 
$$\Sc = {\rm span}(\Uv_{1, \Sc}) \oplus {\rm span}(\Uv_{1,\Sc}^\perp) = {\rm span}(\Uv_{1,\Sc}) \oplus \Tc_p \oplus {\rm span}(\fv_0).$$
Here, ${\rm span}(\Uv_{1,\Sc})$ is the subspace to which the second data piling directions are orthogonal. The orthogonal complement of ${\rm span}(\Uv_{1,\Sc})$ within $\Sc$, ${\rm span}(\Uv_{1,\Sc}^\perp)$ is further decomposed into `useless' data piling part $\Tc_p := {\rm span}(\left\{\hat\uv_i \right\}_{i \in \Dc}) \cap {\rm span}(\Uv_{1,\Sc}^\perp)$ and the rest. While the dimension of $\Tc_p$ is always zero when $\tau_1 = \tau_2$, it becomes $m_1 + m_2 -m$ when $\tau_1 \ne \tau_2$, which is nonnegative but can be zero (depending on the model parameters). The last part, ${\rm span}(\fv_0)$, is  ${\rm span}(\Uv_{1,\Sc}^\perp) \setminus \Tc_p$, for which $\fv_0$ is the orthogonal projection of $\mdp$ onto ${\rm span}(\Uv_{1,\Sc}^\perp)$.


Next result confirms that $\Tc_p$ is indeed not useful, and also characterizes the piling distance on $\fv_0$. 
Recall that $\W_\Xc$ is the collection of all sequences of directions within the sample space $\Sc_\Xc$ (see Definition~\ref{def:SDP}).

\begin{theorem}\label{thm:SMDP piling distance}
Suppose Assumptions \ref{assume:1}---\ref{assume:5} hold. Then,
\begin{itemize}
    \item[(i)] For any $\left\{ \wv \right\} \in \Tc = \left\{\left\{ \wv \right\} \in \W_\Xc : \wv \in \Tc_p \text{ for all } p\right\}$ and any independent observation $Y \in \Yc$, $p^{-1/2}\wv^\top (Y - \bar{X}) \xrightarrow{P} 0$ as $p \to \infty$ whenever $\pi(Y) = 1$ or $\pi(Y) = 2$.
    \item[(ii)] For any independent observation $Y \in \Yc$,
    \begin{align*}
        \frac{1}{\sqrt{p}} \fv_0^\top (Y - \bar{X}) \xrightarrow{P}  \begin{cases} \upsilon_0 (\eta_2 (1-\cos^2\varphi)\delta^2 - (\tau_1^2 - \tau_2^2) /n), & \pi(Y) = 1, \\ \upsilon_0 (-\eta_1 (1-\cos^2\varphi)\delta^2 - (\tau_1^2 - \tau_2^2) /n), & \pi(Y) = 2 \end{cases}
    \end{align*}
    as $p \to \infty$ where $\upsilon_0$ is a strictly positive random variable depending on the true principal component scores of $\Xc$.
\end{itemize}
\end{theorem}

The above result naturally implies that $\fv_0$ is a second maximal data piling direction, as we state next. Recall that $\Ac$ is the collection of all sequences of second data piling directions, and the second maximal data piling direction maximizes the asymptotic piling distance (Definitions~\ref{def:SDP} and~\ref{def:SMDP}).

\begin{theorem}\label{thm:SMDP characterization}
Suppose Assumptions \ref{assume:1}---\ref{assume:5} hold. Then, 
\begin{itemize}
    \item[(i)] For any given $\left\{ \wv \right\} \in \Ac$, there exists a sequence $\left\{ \vv \right\} \in \Bc$ such that $\| \wv - \vv \| \xrightarrow{P} 0$ as $p \to \infty$, where $\Bc = \left\{\left\{ \vv \right\} \in \W_\Xc : \vv \in {\rm span}(\fv_0) \oplus \Tc_p \oplus {\rm span}(\left\{\hat\uv_i \right\}_{i \in \left\{1, \ldots, n-2 \right\} \setminus \Dc}) \right\}$.
    
    \item[(ii)] For any $\left\{\wv \right\} \in \Ac$ such that the limiting piling distance of independent test data, $D(\wv)$, exists, $\left\{ \wv \right\}$ is a sequence of second maximal data piling directions if and only if $\|\wv - \fv_0 \| \xrightarrow{P} 0$ as $p \to \infty$.
\end{itemize}
\end{theorem}

It is important to note that a second maximal data piling direction is a direction $\wv \in \Sc$ which is asymptotically orthogonal to both of ${\rm span}(\Uv_{1,\Sc})$ and $\Tc_p$. A second maximal data piling direction lies in ${\rm span}(\Uv_{1,\Sc}^\perp)$, but is also orthogonal to $\Tc_p$ that is orthogonal to $\mdp$. This observation gives a unified view on the theoretical second maximal data piling direction and the estimates of such directions:

\begin{quote}
\normalsize
    ``Second maximal data piling directions are the projection of $\mdp$ onto the orthogonal complement of ${\rm span}(\Uv_{1,\Sc})$ (or its estimate).''
\end{quote}

A critical issue is that the signal subspace $\Sc$ is decided at random in general cases (e.g., strong spikes with unequal tail eigenvalues). This makes it difficult to determine the signal subspace $\Sc$ as well as the orthogonal complement of ${\rm span}(\Uv_{1,\Sc})$. Nevertheless, the above intuition can be utilized in developing an estimation strategy for the second maximal data piling direction. 


\subsection{Data-splitting Approach for Second Maximal Data Piling Direction}\label{sec:data-splitting approach}
In this section, we propose a data-splitting approach to estimate a second maximal data piling direction.
%
Suppose for now that 
an independent test dataset $\Yc$ is available to us. (Later, we will split the original training data to into training and test data sets.) 
The idea we use is to estimate ${\rm span}(\Uv_{1,\Sc})$ using $\Yc$, and to project $\mdp$ (obtained from $\Xc$) onto  the orthogonal complement of the estimate of ${\rm span}(\Uv_{1,\Sc})$. In the following, we verify that such a data-splitting approach is theoretically sound, and outline algorithmic procedures. 

Denote the horizontally concatenated data matrix of the given independent test dataset $\Yc$ by $\Yv = [Y_{11},\ldots,Y_{1n_1^*},Y_{21},\ldots,Y_{2n_2^*}]$. The $p \times n^*$ data matrix $\Yv$ consists of the $n^*:=n_1^*+n_2^*$ observations independent to $\Xc$ and arranged so that $\pi(Y_{kj}) = k$ for any $k,j$. We assume that $n_k^*$ is fixed and $n_k^* > m_k$ for $k = 1, 2$. Write class-wise sample mean vectors $\bar{Y}_k = n_k^{*-1}\sum_{j=1}^{n_k^*}Y_{kj}$. We define the within-scatter matrix of $\Yc$ as $\Sv_W^* = (\Yv - \bar{\Yv})(\Yv - \bar{\Yv})^\top,$
where $\bar{\Yv}_k = \bar{Y}_k \1v_{n_k^*}^\top$ for $k = 1, 2$ and $\bar{\Yv} = [\bar{\Yv}_1,~\bar{\Yv}_2]$. 
Let $\Vv = [\hat\uv_{1}, \ldots, \hat\uv_{n-2}, \mdp]$, which collects an orthonormal basis of the sample space $\Sc_{\Xc}$ (from the training dataset $\Xc$). 

We are interested in characterizing the second data piling directions $\wv \in \Sc_{\Xc}$ (or the sequence 
$\left\{\wv\right\} \in \W_\Xc$). Observe that $\wv$ is a second data piling direction if  the (scaled) within-class scatter $p^{-1}\wv^\top \Sv_{W}^{*}\wv$ becomes negligible. Thus, we write the collection of sequences of second data piling directions (for the case where both $\Xc$ and $\Yc$ are available) as
%
\begin{align*}
    \bar{\Ac} := \left\{\left\{\wv  \right\} \in \W_\Xc : \frac{1}{p}\wv^\top \Sv_W^* \wv \xrightarrow{P} 0 \text{ as } p \to \infty \right\}.
\end{align*}

For any $\left\{\wv\right\} \in \W_\Xc$, we can write $\wv = \Vv\av$ for some $\av = (a_1, \ldots, a_{n-2}, a_{\textup{MDP}})^\top \in \Real^{n-1}$. Without loss of generality, we assume $a_{\textup{MDP}} \ge 0$ for all $p$. For $\left\{ \wv \right\} \in \W_{\Xc}$, we can write 
\begin{align}\label{eq:scatter}
\frac{1}{p}\wv^\top \Sv_{W}^{*}\wv = \frac{1}{p}\av^\top \Vv^\top \Sv_{W}^{*}\Vv\av = \av^\top \left(\frac{1}{p}\Vv^\top\Sv_{W}^{*}\Vv\right)\av.
\end{align}
Note that the $(n-1) \times (n-1)$ matrix $p^{-1}\Vv^\top \Sv_{W}^*\Vv$ can be understood as the scatter of the independent test data $\Yc$ projected onto the sample space $\Sc_{\Xc}$. Theorem~\ref{thm:scatter of projected test data} shows that independent test data $\Yc$ are asymptotically supported on a $m$-dimensional subspace in $\Sc_{\Xc}$.

\begin{theorem}\label{thm:scatter of projected test data}
Suppose Assumptions \ref{assume:1}---\ref{assume:5} hold. Then $p^{-1}\Vv^\top\Sv_{W}^{*}\Vv$ converges to a rank-$m$ matrix in probability as $p \to \infty$.
\end{theorem}

Write the eigen-decomposition of the scatter matrix as follows: 
$$p^{-1}\Vv^\top \Sv_{W}^{*}\Vv=\hat\Qv\Hv\hat\Qv^\top,$$ where $\Hv=\diag{(h_1,\ldots,h_{n-1})}$ arranged in descending order, and $\hat\Qv =[\hat\Qv_1,~\hat\Qv_2]$ with $\hat\Qv_1=[\hat\qv_1,\ldots,\hat\qv_{m}]$ and $\hat\Qv_2 = [\hat\qv_{m+1},\ldots,\hat\qv_{n-1}]$. Accordingly, the coefficients $\av$ can be expressed as $\av = \hat\Qv \iotav = \sum_{i=1}^{n-1} \iota_i \hat\qv_i$ by $\iotav = (\iota_1, \ldots, \iota_{n-1})^\top$. Since
\begin{align}\label{eq:first_m}
\frac{1}{p}\wv^\top \Sv_{W}^{*}\wv = \av^\top \left(\frac{1}{p}\Vv^\top\Sv_{W}^{*}\Vv\right)\av = \sum_{i=1}^{n-1}h_i\iota_i^2 = \sum_{i=1}^{m}h_i\iota_i^2 + o_p(1)
\end{align}
by Theorem~\ref{thm:scatter of projected test data} and (\ref{eq:scatter}), $\left\{\wv\right\} \in \bar{\Ac}$ if and only if $\iota_1, \ldots, \iota_{m} \xrightarrow{P} 0$ as $p \to \infty$.
In other words, any second data piling direction $\{\wv\} \in \bar{\Ac}$ satisfies $\wv \in {\rm span}(\Vv\hat\Qv_2)$ (asymptotically).
%
This fact plays a crucial role in our next observation: $\left\{\wv \right\} \in \bar{\Ac}$ is indeed asymptotically orthogonal to the common leading eigenspace $\Uc$.  

\begin{theorem}\label{thm:data-splitting approach piling distance}
Suppose Assumptions \ref{assume:1}---\ref{assume:5} hold. Then for any given $\left\{\wv \right\} \in \bar{\Ac}$, (i) $\wv^\top \uv_{j} \xrightarrow{P} 0$ for $1 \leq j \leq m$. (ii) Write $$\wv = \Vv\av = \sum_{k=1}^{n-2}a_k \hat\uv_{k} + a_{\textup{MDP}}\mdp$$ with $\av = (a_1, \ldots, a_{n-2}, a_{\textup{MDP}})^\top$ and assume $a_{\textup{MDP}} \xrightarrow{P} \psi_{\textup{MDP}}$ as $p \to \infty$. Then for any independent observation $Y$, which is independent to both of $\Xc$ and $\Yc$, 
\begin{align*}
    \frac{1}{\sqrt{p}} \wv^\top(Y-\bar{X}) \xrightarrow{P} \begin{cases} \frac{\psi_{\textup{MDP}}}{\kappa}(\eta_2(1-\cos^2\varphi)\delta^2 - (\tau_1^2 - \tau_2^2)/n), & \pi(Y) = 1, \\ \frac{\psi_{\textup{MDP}}}{\kappa}(-\eta_1(1-\cos^2\varphi)\delta^2 - (\tau_1^2 - \tau_2^2)/n), & \pi(Y) = 2, \end{cases}
\end{align*}
as $p\to\infty$, where $\kappa$ is the probability limit of $\kappa_{\textup{MDP}}$ defined in (\ref{eq:kMDP}).
\end{theorem} 

The above theorem also implies that for a carefully chosen $\av$, the direction $\wv$ can be used to yield  perfect classification of new independent observation $Y$, which is independent to both of $\Xc$ and $\Yc$. Requirements for such an $\av$ are $\left\{\wv \right\} \in \bar{\Ac}$ (where $\wv = \Vv\av$) and  $\lim_{p \to \infty} a_{\textup{MDP}} > 0$. 

Theorem~\ref{thm:data-splitting approach piling distance} also implies that the piling gap is maximized when the limiting value $\psi_{\textup{MDP}}$ of $a_{\textup{MDP}}$ is largest. 
Combining this observation with (\ref{eq:first_m}), it is natural to project $\mdp$ onto ${\rm span}(\Vv \hat\Qv_2)$: $\mdp$ is used to maximize $\psi_{\textup{MDP}}$, and ${\rm span}(\Vv \hat\Qv_2)$ is used to ensure second data piling. Next result confirms that such an estimate is indeed maximal. 




\begin{theorem}\label{thm:data-splitting approach SMDP}
Suppose Assumptions \ref{assume:1}---\ref{assume:5} hold. Write $\ev_{\textup{MDP}} = (\0v_{n-2}^\top,~1)^\top$ so that $\mdp = \Vv \ev_{\textup{MDP}}$. Also, let $\left\{\smdp  \right\}$ be a sequence of directions such that $\smdp := \Vv \av_{\textup{SMDP}}$ where $$\av_{\textup{SMDP}} := \frac{P_{{\rm span}(\hat{\Qv}_2)}{\ev_{\textup{MDP}}}}{ \|P_{{\rm span}(\hat{\Qv}_2)}{\ev_{\textup{MDP}}} \|} = \frac{ \hat{\Qv}_2\hat{\Qv}_2^\top \ev_{\textup{MDP}}}{\| \hat{\Qv}_2\hat{\Qv}_2^\top \ev_{\textup{MDP}} \|} \in \Real^{n-1}.$$ Then $\{\smdp \}$ is a sequence of second maximal data piling directions.
\end{theorem} 

We remark that the above procedures are analogous to the optimization problem in (\ref{eq:MDP}). Recall that $\mdp$ is the solution of (\ref{eq:MDP}) which maximizes the training data piling distance (i.e., $\wv^\top \Sv_{B} \wv$) among the (first) data piling directions (i.e., $\wv^\top \Sv_{W}\wv = 0$). Heuristically, our procedures can be seen as finding a sequence of directions $\{\wv\} \in \W_\Xc$, which maximizes the asymptotic distance between two piles of independent test data (i.e., $\lim_{p \to \infty}\wv^\top \Sv_{B}^*\wv$ where $\Sv_B^*$ is the between-class scatter matrix of $\Yc$) among the second data piling directions (i.e., $\lim_{p \to \infty} p^{-1} \wv^\top \Sv_{W}^* \wv = 0$). 

\subsection{Computational Algorithm}\label{sec:computational algorithm}
We have shown that a second maximal data piling direction in the sample space $\Sc_{\Xc}$ can be obtained with a help of independent test data. In practice, we randomly split $\Xc_k$, which is the original training dataset of the $k$th class, into training dataset $\Xc_{k,tr}$ and test dataset $\Xc_{k,te}$ so that the sample size of test data of $k$th class $n_{k,te}$ is larger than $m_k$ for $k = 1, 2$. Then, the estimate $\smdp$ is computed using both $\Xc_{tr} = \Xc_{1,tr} \cup \Xc_{2,tr}$ and $\Xc_{te} = \Xc_{1,te} \cup \Xc_{2,te}$. 

The small sample size characteristic of HDLSS data suggests that classification based on a single data split may be unreliable in practice (albeit theoretically valid).
In order to resolve this concern, we repeat the above procedure several times and set a final estimate of a second maximal data piling direction as the average of estimates of a second maximal data piling direction obtained from each repetition. A detailed algorithm is given in Algorithm~\ref{alg:1}. 



The classification rule $\phi_{\textup{SMDP-I}}$ of Algorithm~\ref{alg:1} ensures perfect classification of independent test data under the HDLSS asymptotic regime (by Theorem~\ref{thm:data-splitting approach piling distance}). In Algorithm~\ref{alg:1}, we also estimate a bias term due to the difference between $\tau_1^2$ and $\tau_2^2$ as $g_{j, \textup{SMDP}}$ in line~\ref{eq:alg1d} for each repetition. We note that $\hat\alpha_k$ in line~\ref{eq:alg1c} is an HDLSS-consistent estimator of $-\tau_k^2$; see Appendix~\ref{app:weak spikes} in the supplemental materials. In fact, we do not need to estimate this term since projections of $\Xc_{te}$ onto $\smdp$ tend to converge two distinct points, one for each class (due to the second data piling). Hence, we propose another classification rule $\phi_{\textup{SMDP-II}}$, which is almost same as $\phi_{\textup{SMDP-I}}$ except that $\phi_{\textup{SMDP-II}}$ computes the the classification threshold by utilizing the simple LDA rule. 
Taking this approach eliminates the need to estimate $m_1$ and $m_2$. A detailed algorithm is given in Algorithm~\ref{alg:2}. 

While we have assumed $\tau_1 \ge \tau_2$ throughout the paper, the proposed algorithms achieve asymptotic perfect classification regardless of whether $\tau_1 \ge \tau_2$ or $\tau_1 < \tau_2$. Also, we do not assume that the tail eigenvalues $\tau_1^2$ and $\tau_2^2$ are known a priori or well-separated. The only requirements are estimates of $m_1$, $m_2$ and $m$, which are the true numbers of leading eigenvalues of $\Sigmav_{(1)}$, $\Sigmav_{(2)}$ and $\Sigmav_{(0)} = \pi_1\Sigmav_{(1)} + \pi_2\Sigmav_{(2)}$, respectively. 
This problem has been extensively studied in the literature; see, e.g., \citet{Kritchman2008,Leek2010,Passemier2014,jung2018number,aoshima2019}.

\begin{algorithm}
\caption{Second Maximal Data Piling (SMDP) algorithm (Type I)}
\begin{algorithmic}[1]
    \Require Original training data matrix of the $k$th class $\Xv_k$ for $k = 1, 2$.
    \Require The number of repetitions $K$, estimated $m_1$, $m_2$ and $m$
    \For{$j = 1, \ldots, K$}
        \State\label{eq:alg1a} Randomly split $\Xv_k$ into $\begin{cases} \Xv_{k, tr} = [X_{k1}, \ldots, X_{kn_{k,tr}}] \\ \Xv_{k, te} = [X_{k1}^*, \ldots, X_{kn_{k,te}}^*] \end{cases}$ so that $n_{k,te} > m_k$ ($k = 1, 2$)
        \State Set $n_{tr} = n_{1,tr} + n_{2,tr}$, $\bar{X}_{k,tr} = n_{k,tr}^{-1}\Xv_{k,tr}\1v_{n_{k,tr}}$ and $\bar{\Xv}_{k,tr} = \bar{X}_{k,tr}\1v_{n_{k,tr}}^\top$ ($k = 1, 2$)
        \State Set $\Sv_{tr} = \sum_{k=1}^{2}(\Xv_{k,tr} - \bar{\Xv}_{k,tr})(\Xv_{k,tr} - \bar{\Xv}_{k,tr})^\top$ and $\dv_{tr} = \bar{X}_{1,tr} - \bar{X}_{2,tr}$
        \State Write an eigen-decomposition of $\Sv_{tr}$ by 
        $\Sv_{tr} = \hat\Uv\hat\Lambdav\hat\Uv^\top = \hat\Uv_{1}\hat\Lambdav_{1}\hat\Uv_{1}^\top$
        where $\hat\Lambdav = \diag{(\hat\lambda_1, \ldots, \hat\lambda_{n_{tr}-2}, 0, \ldots, 0)}$ arranged in descending order, and $\hat\Uv = [\hat\Uv_1,~\hat\Uv_2]$ with $\hat\Uv_1 = [\hat\uv_1, \ldots, \hat\uv_{n_{tr}-2}]$, $\hat\Uv_2 = [\hat\uv_{n_{tr}-1}, \ldots, \hat\uv_p]$
        \State Set $\mdp = \|\hat\Uv_{2}\hat\Uv_{2}^\top \dv_{tr} \|^{-1}\hat\Uv_{2}\hat\Uv_{2}^\top \dv_{tr} $ and $\kappa_{\textup{MDP}} = p^{-1/2}\|\hat\Uv_{2}\hat\Uv_{2}^\top \dv_{tr} \|$
        \State Set $\Vv = [\hat\uv_{1}, \ldots, \hat\uv_{n_{tr} - 2}, \mdp]$
        \State Set $\bar{X}_{k,te} = n_{k,te}^{-1}\Xv_{k,te}\1v_{n_{k,te}}$ and $\bar{\Xv}_{k,te} = \bar{X}_{k,te}\1v_{n_{k,te}}^\top$ ($k = 1, 2$)
        \State Set $\Sv_{te} = \sum_{k=1}^{2}(\Xv_{k,te} - \bar{\Xv}_{k,te})(\Xv_{k,te} - \bar{\Xv}_{k,te})^\top$
        \State Write an eigen-decomposition of $p^{-1}\Vv^\top \Sv_{te}\Vv$ by $p^{-1}\Vv^\top \Sv_{te}\Vv = \hat\Qv\Hv\hat\Qv^\top$ where $\Hv = \diag{(h_1, \ldots, h_{n_{tr}-1})}$ arranged in descending order, and $\hat\Qv = [\hat\Qv_1,~\hat\Qv_2]$ with $\hat\Qv_1 = [\hat\qv_1, \ldots, \hat\qv_m]$ and $\hat\Qv_2 = [\hat\qv_{m+1}, \ldots, \hat\qv_{n_{tr}-1}]$
        \State Set $\av_{j,\textup{SMDP}} =  \| \hat\Qv_2\hat\Qv_2^\top \ev_{\textup{MDP}}  \|^{-1} \hat\Qv_2\hat\Qv_2^\top \ev_{\textup{MDP}}$ where $\ev_{\textup{MDP}} = (\0v_{n-2}^\top, 1)^\top$
        \State\label{eq:alg1b} Set $\wv_{j,\textup{SMDP}} = \Vv \av_{j,\textup{SMDP}}$
        \State Set $\bar{X}_{j, \textup{SMDP}} = n_{tr}^{-1}(n_{1,tr}\bar{X}_{1,tr} + n_{2,tr}\bar{X}_{2,tr})$
        \State\label{eq:alg1c} Set $\hat\alpha_k = -p^{-1}\sum_{i=m_k+1}^{n_{k,tr}-1}\hat\lambda_{(k),i}$ where $\hat\lambda_{(k),i}$ is the $i$th largest eigenvalue of $$\Sv_{k,tr} = (\Xv_{k,tr} - \bar{\Xv}_{k,tr})(\Xv_{k,tr} - \bar{\Xv}_{k,tr})^\top$$
        \State\label{eq:alg1d} Set $g_{j,\textup{SMDP}} = (n_{tr}\kappa_{\textup{MDP}})^{-1} (\ev_{\textup{MDP}}^\top \av_{j,\textup{SMDP}})(\hat\alpha_1 - \hat\alpha_2)$
    \EndFor
    \State Set $\smdp = K^{-1} \sum_{j=1}^{K} {\wv}_{j, \textup{SMDP}}$
    \State Set $\bar{X}_{\textup{SMDP}} = K^{-1} \sum_{j=1}^{K} \wv_{j, \textup{SMDP}}^\top\bar{X}_{j, \textup{SMDP}}$
    \State Set $g_{\textup{SMDP}} = K^{-1} \sum_{j=1}^{K} g_{j, \textup{SMDP}}$
    \State Use the following classification rule:
    \begin{align}\label{eq:SMDP-I-rule}
        \phi_{\textup{SMDP-I}}(Y; \Xc) = \begin{cases}
            1, & p^{-1/2}(\smdp^\top Y-\bar{X}_{\textup{SMDP}}) - g_{\textup{SMDP}} \ge 0, \\
            2, & p^{-1/2}(\smdp^\top Y-\bar{X}_{\textup{SMDP}}) - g_{\textup{SMDP}} < 0.
        \end{cases}
    \end{align}
\end{algorithmic}
\label{alg:1}
\end{algorithm}

\begin{algorithm}
\caption{Second Maximal Data Piling (SMDP) algorithm (Type II)}
\begin{algorithmic}[1]
    \Require Original training data matrix of the $k$th class $\Xv_k$ for $k = 1, 2$.
    \Require The number of repetitions $K$, estimated $m$
    \For{$j = 1, \ldots, K$}
        \State Do lines \ref{eq:alg1a}--\ref{eq:alg1b} in Algorithm~\ref{alg:1}
        \State Apply Linear Discriminant Analysis to $p^{-1/2}\wv_{j,\textup{SMDP}}^\top\Xv_{te}$ where $\Xv_{te} = [\Xv_{1, te},~\Xv_{2, te}]$ and achieve a classification threshold $b_j$
    \EndFor
    
    \State Set $\smdp = K^{-1} \sum_{j=1}^{K} \wv_{j, \textup{SMDP}}$
    \State Set $b_{\textup{SMDP}} = K^{-1} \sum_{j=1}^{K} b_j$
    \State Use the following classification rule:
    \begin{align}\label{eq:SMDP-II-rule}
        \phi_{\textup{SMDP-II}}(Y; \Xc) = \begin{cases}
            1, & p^{-1/2}\smdp^\top Y \ge b_{\textup{SMDP}}, \\
            2, & p^{-1/2}\smdp^\top Y < b_{\textup{SMDP}}.
        \end{cases}
    \end{align}
    
\end{algorithmic}
\label{alg:2}
\end{algorithm}

\section{Ridged Discriminant Directions and Second Maximal Data Piling}\label{sec:estimation ridge} 
In Section~\ref{sec:estimation of SMDP}, we confirmed that the second maximal data piling direction is the direction closest to $\mdp$ in the nullspace of leading eigenspace. This insight was leveraged to estimate $\smdp$ via a data-splitting approach in Algorithms~\ref{alg:1} and~\ref{alg:2}. On the other hand, under the homogeneous covariance model (i.e., $\Sigmav_{(1)} = \Sigmav_{(2)}$), \citet{Chang2021} revealed the close relationship between ridged linear discriminant vectors and the second maximal data piling direction. Formally, the ridged linear discriminant (direction) vector $\wv_{\alpha}$ is defined as 

\begin{equation}\label{eq:ridge}
    \begin{aligned}
    {\wv}_{\alpha} \varpropto \alpha_p(\Sv_W + \alpha_p\Iv_p)^{-1} \dv = \sum\limits_{i=1}^{n-2} \frac{\alpha_p}{\hat\lambda_i + \alpha_p} \hat\uv_i\hat\uv_i^\top\dv + \hat\Uv_2\hat\Uv_2^\top\dv,
    \end{aligned}
\end{equation}
satisfying $\|\wv_{\alpha} \|_2 = 1$, where $\alpha_p = \alpha p$ for a ridge parameter $\alpha \in \Real$. Note that $\mdp$ is a special case of $\wv_{\alpha}$ since $\lim_{\alpha \to 0} \wv_{\alpha} = \mdp$, that is, $\mdp$ is the \textit{ridgeless} minimum-norm estimator in the context of linear classification. \citet{Chang2021} showed that, if $\Sigmav_{(1)} = \Sigmav_{(2)}$, the ridged linear discriminant vector $\wv_{\alpha}$ concentrates toward the signal subspace $\Sc$ in high dimensions. In the next theorem, we observe that this result also holds under heterogeneous covariance assumptions.

\begin{theorem}\label{thm:ridge concentration}
Suppose Assumptions \ref{assume:1}---\ref{assume:5} hold. Let $\Sc = {\rm span}(\left\{\hat\uv_i \right\}_{i \in \Dc}, \mdp)$, where the index set $\Dc$ is specified in Table~\ref{table:D}. For any $\alpha \in \Real \setminus \left\{-\tau_1^2, -\tau_2^2 \right\}$ for which (\ref{eq:ridge}) is defined, ${\rm Angle}{(\wv_{\alpha}, \Sc)} \xrightarrow{P} 0 $
as $p \to \infty$.
\end{theorem}

Based on this, we consider removing noise terms in $\wv_{\alpha}$ and define a \textit{projected} ridged linear discriminant vector $\vv_{\alpha}$,
\begin{align}\label{eq:pridge}
    \vv_\alpha \varpropto \sum\limits_{i \in \Dc} \frac{\alpha_p}{\hat\lambda_i + \alpha_p} \hat\uv_i\hat\uv_i^\top \dv + \hat\Uv_2\hat\Uv_2^\top \dv
\end{align}
satisfying $\|\vv_{\alpha}\| = 1$. 
In case of $\Sigmav_{(1)} = \Sigmav_{(2)}$, \citet{Chang2021} showed that $\vv_{\alpha}$, for $\alpha = \hat\alpha$, is a second maximal data piling direction. Here,  $\hat\alpha$ is any HDLSS-consistent estimator of $-\tau^2$. In this special case, $\smdp$ in Algorithms~\ref{alg:1} and~\ref{alg:2} is essentially equivalent to the negatively ridged classification direction. A natural question is whether the projected ridged linear discriminant vector can also yield second maximal data piling under the generalized heterogeneous covariance model. Let $\Uc_{(k)} = {\rm span}(\Uv_{(k),1})$ be the leading eigenspace of the $k$th class, and recall that $\tau_k^2$ is the average of the non-leading eigenvalues of $\Sigmav_{(k)}$ for $k = 1, 2$. Our short answer is that:
\begin{center}
\begin{itemize}
    \item[(1)] if $\Uc_{(1)} \ne \Uc_{(2)}$ and $\tau_1 = \tau_2 =:\tau$, then the conclusion of \citet{Chang2021} can be extended to this case; $\vv_{\hat\alpha}$ is the second maximal data piling direction and the projected ridge classification rule $\phi_{\textup{PRD},\hat\alpha}$ based on $\vv_{\hat\alpha}$ achieves asymptotic perfect classification. We defer the details to Appendix~\ref{app:strong spikes with equal tails} in the supplemental materials.
    \item[(2)] if $\Uc_{(1)} = \Uc_{(2)}$ and $\tau_1 \ne \tau_2$, then $\vv_{\alpha}$ yields second data piling at two negative ridge parameters $\alpha = -\tau_1^2$ and $\alpha = -\tau_2^2$. However, neither of them is the second maximal data piling direction; see Section~\ref{sec:One-spike Model with Common Leading Eigenspace} and Appendix~\ref{app:strong spikes with unequal tails} in the supplemental materials.
    
    \item[(3)] if $\Uc_{(1)} \ne \Uc_{(2)}$ and $\tau_1 \ne \tau_2$, then $\vv_{\alpha}$ does not yield second data piling for any ridge parameter $\alpha \in \Real$; see Section~\ref{sec:General One-spike Model with Heterogeneous Leading Eigenspace}.
\end{itemize}
\end{center}

Here, the case $\Uc_{(1)} \ne \Uc_{(2)}$ encompasses various scenarios of heterogeneous covariance models (e.g., different numbers of spikes, varying leading eigenvalues and eigenvectors). The above results highlight that the difficulty in high-dimensional binary classification arises more from differences in the tail eigenvalues (i.e., $\tau_1^2$ and $\tau_2^2$) than from differences in the leading eigenstructures (i.e., $\Uc_{(1)}$ and $\Uc_{(2)}$). As discussed in Section~\ref{sec:test data piling}, the asymptotic behavior of the eigenvectors of $\Sv_W$ are quite different depending on the tail eigenvalues, which in turn affects the behavior of the ridge classification directions. Nevertheless, we remark that Algorithms~\ref{alg:1} and~\ref{alg:2} in Section~\ref{sec:estimation of SMDP} yield a second maximal data piling direction in all the cases we have discussed. 

In the remainder of this section, we provide the details for the unequal tail eigenvalues case with $\tau_1 > \tau_2$.
For clearer presentation, we investigate 
two cases of simple one-spike models (with $m_1 = 1$ and $m_2 = 1$).

\subsection{One-spike Model with Common Leading Eigenspace 
}\label{sec:One-spike Model with Common Leading Eigenspace}
Recall the motivating example at the beginning of Section~\ref{sec:SMDP theory} where we assumed $\uv_{(1),1} = \uv_{(2),1} = \uv_1$ but $\tau_1 > \tau_2$. In this case, the signal subspace $\Sc$ was ${\rm span}(\hat\uv_1, \hat\uv_{n_1},\mdp)$ and $\dim\Sc - m = 2$. We have seen that there may exist two second data piling directions that are orthogonal to each other and to $\uv_{1,\Sc}$. 
%
A natural question is whether the projected ridged linear discriminant vector $\vv_{\alpha}$ yields second data piling for some $\alpha \in \Real$. Proposition~\ref{prop:ridge one-spike unequal tails m = 1} below shows that both $\vv_{\hat\alpha_1}$ and $\vv_{\hat\alpha_2}$ are asymptotically orthogonal to $\uv_{1}$. 
Also, the two piles of $P_{\vv_{\hat\alpha_k}}\Yc_1$ and $P_{\vv_{\hat\alpha_k}}\Yc_2$ are apart from each other. 

\begin{proposition}\label{prop:ridge one-spike unequal tails m = 1}
Suppose Assumptions~\ref{assume:1}---\ref{assume:5} hold and assume $\beta_1 = \beta_2 = 1$, $\tau_1 > \tau_2$ and $m_1 = m_2 = m = 1$. Then (i) for $\hat\alpha_k$ chosen as an HDLSS-consistent estimator of $-\tau_k^2$, ${\rm Angle}{(\vv_{\hat\alpha_k}, \uv_{1, \Sc})} \xrightarrow{P} \pi/2$ as $p \to \infty$ for $k = 1, 2$. (ii) Also, ${\rm Angle}{(\vv_{\alpha}, \uv_{1, \Sc})} \xrightarrow{P} \pi/2$ as $p \to \infty$ if and only if $\alpha = -\tau_1^2$ or $\alpha = -\tau_2^2$. (iii) Moreover, for any independent observation $Y \in \Yc$ and $k = 1, 2$,
\begin{equation}\label{eq:ridge unequal tail eigenvalues projection m = 1}
    \begin{aligned}
        \frac{1}{\sqrt{p}} \vv_{\hat\alpha_k}^\top(Y-\bar{X}) \xrightarrow{P} \begin{cases} \gamma_{k}(\eta_2(1-\cos^2\varphi)\delta^2 - (\tau_1^2 - \tau_2^2) / n), & \pi(Y) = 1, \\ 
        \gamma_{k}(-\eta_1(1-\cos^2\varphi)\delta^2 - (\tau_1^2 - \tau_2^2) / n), & \pi(Y) = 2 
        \end{cases}
    \end{aligned}
\end{equation}
as $p\to\infty$ where $\gamma_{k}$ $(k = 1, 2)$ is a strictly positive random variable depending on the true principal component scores of $\Xc$.
\end{proposition}

We note that second data piling occurs only at the negative ridge parameters $\alpha = -\tau_1^2$ or $\alpha = -\tau_2^2$ (or at their HDLSS-consistent estimates).

\begin{remark}
From (\ref{eq:ridge unequal tail eigenvalues projection m = 1}), if $\tau_1 \ne \tau_2$, we can check that the projected ridge classification rule $\phi_{\textup{PRD},\alpha}$ of \citet{Chang2021} (see Appendix~\ref{app:strong spikes with equal tails} of the supplemental materials) may fail to achieve perfect classification due to the bias term. However, we can still achieve perfect classification using $\vv_{\alpha}$ with a bias-correction strategy. This bias correction is similar to $\phi_{\textup{b-MDP}}$ in Appendix~\ref{app:weak spikes} of the supplemental materials used for non-spike case. The bias-corrected projected ridge classification rule $\phi_{\textup{b-PRD},\alpha}$ is given in Appendix~\ref{app:strong spikes with unequal tails} of the supplemental materials, and ensures asymptotic perfect classification under common leading eigenspace conditions. 
\end{remark}

In contrast to the equal-tail case discussed in Appendix~\ref{app:strong spikes with equal tails} of the supplemental materials, $\vv_{\alpha}$ achieves perfect classification with two negative ridge parameters. This naturally raises the question of which, between $\vv_{\hat\alpha_1}$ and $\vv_{\hat\alpha_2}$, results in a larger asymptotic distance between the two piles of independent test data. 

Writing $D(\wv)$ for the asymptotic distance between the two piles of independent test data projected on $\left\{\wv \right\} \in \Ac$, Proposition~\ref{prop:ridge one-spike unequal tails m = 1} establishes that
$D(\vv_{\hat\alpha_k}) = \gamma_{k} (1-\cos^2\varphi)\delta^2$ (for $k = 1,2$).
Although this distance is not intuitive to interpret, 
our next result demonstrates 
that neither $\vv_{\hat\alpha_1}$ nor $\vv_{\hat\alpha_2}$ necessarily gives a larger asymptotic distance of the two piles than the other.

\begin{theorem}\label{thm:compare two ridges}
Suppose Assumptions \ref{assume:1}---\ref{assume:5} hold and assume $\beta_1 = \beta_2 = 1$, $\tau_1 > \tau_2$ and $m_1 = m_2 = m = 1$. Then,
\begin{itemize}
    \item [(i)] $D(\vv_{\hat\alpha_1}) \leq D(\vv_{\hat\alpha_2}) \Leftrightarrow  \tau_1^{-2} \Phiv_1 \leq  \tau_2^{-2} \Phiv_2$ where $\Phiv_{1}$ (and $\Phiv_{2}$) are the sum of squared scores of  the first principal component of the first (and second, respectively) class [see the exact expressions of $\Phiv_1$ and $\Phiv_2$ in Appendix~\ref{app:asymptotic properties of sample covariance matrix} of the supplemental materials].

    \item [(ii)] If the distribution of $X | \pi(X) = k$ is Gaussian, then 
    for $F \sim F(n_1-1, n_2-1)$, 
    %
    $$\mathbb{P}\left( \tau_1^{-2} \Phiv_1 \leq \tau_2^{-2} \Phiv_2 \right) = \mathbb{P}\left(F \leq  \frac{(n_2-1)\tau_2^{-2}\sigma_{2,1}^2}{(n_1-1)\tau_1^{-2}\sigma_{1,1}^2} \right).$$ 
\end{itemize}
\end{theorem}

In Theorem~\ref{thm:compare two ridges}, 
$\tau_k^{-2}\sigma_{k,1}^2$ can be understood as a signal-to-noise ratio of the $k$th class.

Our next question is whether $\vv_{\hat\alpha_1}$ or $\vv_{\hat\alpha_2}$ can be a second maximal data piling direction or not. From Proposition~\ref{prop:ridge one-spike unequal tails m = 1}, we have 
$D(\vv_{\hat\alpha_k}) = \gamma_{k} (1-\cos^2\varphi)\delta^2$ for $k = 1, 2$. We compare them with the asymptotic test data piling distance $D(\fv_0) = \upsilon_0(1-\cos^2\varphi)\delta^2$ of the theoretical second maximal data piling direction $\fv_0$ in Theorem~\ref{thm:SMDP piling distance}. 

\begin{proposition}\label{prop:compare ridge and SMDP}
Suppose Assumptions \ref{assume:1}---\ref{assume:5} hold and assume $\beta_1 = \beta_2 = 1$, $\tau_1 > \tau_2$ and $m_1 = m_2 = m = 1$. Then, $\upsilon_0 > \max\{ \gamma_{1}, \gamma_{2}\}$ with probability $1$, where $\upsilon_0$ is in Theorem~\ref{thm:SMDP piling distance} and $\gamma_{1}, \gamma_2$ are in Proposition~\ref{prop:ridge one-spike unequal tails m = 1}. 
\end{proposition}

Proposition~\ref{prop:compare ridge and SMDP} implies that the piling gap of $\fv_0$ in Theorem~\ref{thm:SMDP piling distance} is larger than those from $\vv_{\hat\alpha_1}$ and $\vv_{\hat\alpha_2}$. Neither $\vv_{\hat\alpha_1}$ nor $\vv_{\hat\alpha_2}$ is a second maximal data piling direction; it can be seen that these classifiers are affected by the meaningless direction $\fv_1$ in Figure~\ref{fig:one_comp_m=1}.

\subsection{General One-spike Model with Heterogeneous Leading Eigenspace 
}\label{sec:General One-spike Model with Heterogeneous Leading Eigenspace}

We now consider the general case of heterogeneous leading eigenvectors. We revisit Example~\ref{ex:strong_ne_2} in Section~\ref{sec:test data piling examples}. In this case, $\uv_{(1),1} \ne \uv_{(2),1}$ and $m = 2$, and the signal subspace $\Sc$ was given at random (either $\Sc = {\rm span}(\hat\uv_1, \hat\uv_2, \mdp)$ or ${\rm span}(\hat\uv_1, \hat\uv_{n_1}, \mdp)$), depending on the true leading principal component scores.
Although a second data piling occurs for a direction orthogonal to the $2$-dimensional common leading eigenspace $P_{\Sc}\Uc := {\rm span}(\uv_{(1),1,\Sc}, \uv_{(2),1,\Sc})$, it is impossible to determine the random $\Sc$. This indeterminacy of $\Sc$ precludes the use of $\vv_\alpha$ (which is the ridge direction projected on $\Sc$). It turns out that even if we knew $\Sc$, there is no $\vv_{\alpha}$,  for any $\alpha \in \Real$, that yields second data piling.


\begin{proposition}\label{prop:ridge one-spike unequal tails m = 2}
Suppose Assumptions~\ref{assume:1}---\ref{assume:5} hold and assume $\beta_1 = \beta_2 = 1$, $\tau_1 > \tau_2$, $m_1 = m_2 = 1$ and $m = 2$. For any given training data $\Xc$ and $\hat\alpha_k$ chosen as an HDLSS-consistent estimator of $-\tau_k^2$, (i) for each $k = 1, 2$, ${\rm Angle}{(\vv_{\hat\alpha_k}, \uv_{(k),1, \Sc})} \xrightarrow{P} \pi/2$ as $p \to \infty$. (ii) However, $\left\{ \vv_{\alpha} \right\} \notin \Ac$ for any ridge parameter $\alpha \in \Real$.
\end{proposition}

Proposition~\ref{prop:ridge one-spike unequal tails m = 2} shows that $\vv_{\hat\alpha_k}$ is asymptotically orthogonal to $\uv_{(k),1}$, which is the leading eigenvector of the $k$th class, but not asymptotically orthogonal to both of $\uv_{(1),1}$ and $\uv_{(2),1}$. Thus, the projected ridge classifier does not necessarily provides perfect classification. 

We note that the conclusion of Proposition~\ref{prop:ridge one-spike unequal tails m = 2}  remains valid for general $m_k$-spike cases. 

\begin{remark} 
In the special case that the leading eigenspace of the one class includes that of the other class (that is, $\Uc = \Uc_{(k)}$), $\vv_{\alpha}$ can yield second data piling with the negative ridge parameter $-\tau_k^2$. Moreover, in such cases, the bias-corrected projected ridge classification rule $\phi_{\textup{b-PRD}, \alpha}$ (defined in Appendix~\ref{app:strong spikes with unequal tails} of the supplemental materials) can achieve asymptotic perfect classification. See Appendix~\ref{app:strong spikes with unequal tails} in the supplemental materials for further discussion.
\end{remark}

\section{Numerical Studies}\label{sec:numerical studies}
In Section~\ref{sec:simulation}, we numerically demonstrate that our proposed classification rules can achieve asymptotic perfect classification under various heterogeneous covariance models via a simulation study. We also show that these methods achieve competitive classification performance in classifying Olivetti faces data  in Section~\ref{sec:real data examples}.

\subsection{Simulation Study}\label{sec:simulation}
We investigate classification performances of the SMDP algorithms ($\phi_{\textup{SMDP-I}}$, $\phi_{\textup{SMDP-II}}$) proposed in Section~\ref{sec:computational algorithm} through a simulation study. The MDP classification rules ($\phi_{\textup{MDP}}$, $\phi_{\textup{b-MDP}}$) in Appendix~\ref{app:weak spikes}, the PRD classification rules ($\phi_{\textup{PRD},\alpha}$, $\phi_{\textup{b-PRD},\alpha}$) in Appendix~\ref{app:strong spikes with equal tails} and~\ref{app:strong spikes with unequal tails} are also considered. For comparison, we consider Distance-Based Discriminant Analysis (DBDA) by \citet{Aoshima2013}, Transformed Distance-Based Discriminant Analysis (T-DBDA) by \citet{aoshima2019} and the distance-based linear classifier using the data transformation procedure (DT) by \citet{ishii2020}, all of which were developed for the classification of HDLSS data with a spiked covariance model. In particular, T-DBDA and DT were proposed for strongly spiked eigenvalue models where $\beta_k \geq 1/2$ $(k = 1, 2)$, which are based on the idea of removing the leading eigenspace (but DT only considers the case of $m_1 = m_2 = 1$). We also compare our methods with Distance Weighted Discrimination (DWD) by \citet{marron2007} and Support Vector Machine (SVM) by \citet{Vapnik1995}, which are well-known to achieve successful classification performances in high-dimensional classification (see \citet{Qiao2009} and \citet{Huang2017}). We do not consider classifiers that require a sparse estimation of the discriminant direction vector, as they are known to perform poorly under a spiked covariance model. We also do not consider nonparametric classification methods, such as nearest neighbors or decision tree-based classifiers, because the small sample size together with the high dimensionality makes tuning practically meaningless.

We generate training data of size $n_1 = n_2 = 20$ from multivariate normal distribution with mean $\muv_{(k)}$ and covariance matrix $\Sigmav_{(k)}$ with $p = 10,000$. We assume $\muv_{(1)} = (\sqrt{8}\1v_{p/8}^\top, \0v_{7p/8}^\top)^\top$, $\muv_{(2)} = \0v_p$. Note that in this case $\delta^2 = 1$. The population covariance matrices $\Sigmav_{(1)}$ and $\Sigmav_{(2)}$ will be given differently in each setting as follows.

\begin{center}
\begin{tabular}{c|c|c|c|c}
    Setting & $\Sigmav_{(1)}$ & $ \Sigmav_{(2)} $ & $(\beta_1, \beta_2)$ & $(\tau_1^2, \tau_2^2)$ \\
    \hline
    I & \multirow{3}{*}{$\sum\limits_{i=1}^{2}\sigma_{1,i}^2\uv_{(1),i}\uv_{(1),i}^\top + \tau_1^2 \Iv_{p}$} & \multirow{5}{*}{$\sum\limits_{i=1}^{3}\sigma_{2,i}^2\uv_{(2),i}\uv_{(2),i}^\top + \tau_2^2 \Iv_{p}$} & $(1/2, 1/2)$ & $(30, 15)$ \\
    \cline{1-1}\cline{4-5}
    II &  &  & $(1, 1)$ & $(30, 30)$ \\
    \cline{1-1}\cline{4-5}
    III &  &  & $(1, 1)$ & $(30, 15)$ \\
    \cline{1-2}\cline{4-5}
    IV & \multirow{2}{*}{$\sum\limits_{i=1}^{3}\sigma_{1,i}^2\uv_{(1),i}\uv_{(1),i}^\top + \tau_1^2 \Iv_{p}$}  &  & $(1, 1)$ & $(30, 30)$ \\
    \cline{1-1}\cline{4-5}
    V &  &  & $(1, 1)$ & $(30, 15)$ \\
\end{tabular}
\end{center}
We set $(\sigma_{1,1}^2, \sigma_{1,2}^2, \sigma_{1,3}^2) = (5p^{\beta_1}, 3p^{\beta_1}, 2p^{\beta_1}), (\sigma_{2,1}^2, \sigma_{2,2}^2, \sigma_{2,3}^2) = (5p^{\beta_2}, 3p^{\beta_2}, 2p^{\beta_2})$, 
\begin{align*}
[\uv_{(1),1}, \uv_{(1),2}, \uv_{(1),3}] = \frac{1}{\sqrt{p}}\begin{bmatrix} \sqrt{2} \1v_{p/4} & \0v_{p/4} & \1v_{p/4} \\ \sqrt{2} \1v_{p/4} & \0v_{p/4} & -\1v_{p/4} \\ \0v_{p/4} & \sqrt{2} \1v_{p/4} & \1v_{p/4} \\ \0v_{p/4} & \sqrt{2} \1v_{p/4} & -\1v_{p/4}\end{bmatrix},
\end{align*}
and
\begin{align*}
    [\uv_{(2),1}, \uv_{(2),2}, \uv_{(2),3}] = \frac{1}{\sqrt{p}}\begin{bmatrix} \1v_{p/4} & \sqrt{2}\1v_{p/4} & \0v_{p/4} \\ \1v_{p/4} & \0v_{p/4} & \sqrt{2}\1v_{p/4} \\ \1v_{p/4} &  -\sqrt{2}\1v_{p/4} & \0v_{p/4} \\ \1v_{p/4} & \0v_{p/4} & -\sqrt{2}\1v_{p/4} \end{bmatrix}.
\end{align*} 
Setting I corresponds to the case of weak spikes, while the other settings correspond to the case of strong spikes. Also, Settings II and III assume $m = m_2$, that is, $\Uc = \Uc_{(2)}$, while Settings IV and V assume $m > \max{(m_1, m_2)}$. Finally, Settings II and IV assume equal tail eigenvalues, while Settings III and V assume unequal tail eigenvalues for $\Sigmav_{(1)}$ and $\Sigmav_{(2)}$.

Independent test data of size $n_1^* = n_2^* = 500$ are generated from the same model with the training data. For each classification rule, we evaluate three metrics for our simulation study. First, we estimate the classification accuracy using independent test data. Next, we calculate the degree of second data piling defined as
\begin{align}\label{eq:DegSDP}
    \text{DegSDP}(\wv) := \frac{\wv^\top \Sv_B^* \wv}{\wv^\top \Sv_W^* \wv} =  \frac{\sum_{k=1}^2 n_k^*\{\wv^\top(\bar{Y}_k - \bar{Y})\}^2}{\sum_{k=1}^2\sum_{j=1}^{n_k^*} \{\wv^\top (Y_{kj} - \bar{Y}_k)\}^2 }
\end{align}
where $\wv$ is the normal vector of the classification rule, $Y_{kj}$ is the $j$th observation of the test data for the $k$th class, $\bar{Y}_k = n_{k}^{* -1} \sum_{j=1}^{n_k^*}Y_{kj}$ and $\bar{Y} = (n_1^* + n_2^*)^{-1}\sum_{k=1}^{2}n_k^*\bar{Y}_k$. Theoretically, $\text{DegSDP}(\wv) \xrightarrow{P} \infty$ as $p \to \infty$ for any $\{\wv\} \in \Ac$. For all settings except Setting I, we evaluate ${\rm Angle}(\wv, \Uc)$, the angle between the normal vector $\wv$ of the classification rule and the common leading eigenspace $\Uc$. Recall that ${\rm Angle}(\wv, \Uc) \xrightarrow{P} \pi/2$ as $p \to \infty$ for any $\{\wv\} \in \Ac$.

To clearly check performances of each classification rule, we use the true numbers of $m_1$, $m_2$ and $m$ for $\phi_{\textup{SMDP-I}}$, $\phi_{\textup{SMDP-II}}$, $\phi_{\textup{PRD}, \hat\alpha}$, $\phi_{\textup{b-PRD}, \hat\alpha_1}$ and $\phi_{\textup{b-PRD}, \hat\alpha_2}$. Here, $\hat\alpha$, $\hat\alpha_1$ and $\hat\alpha_2$ are HDLSS-consistent estimators of $-\tau^2$ (when $\tau_1 = \tau_2 =:\tau$), $-\tau_1^2$ and $-\tau_2^2$, respectively, all of which can be estimated purely from training data; see Appendices~\ref{app:weak spikes} and~\ref{app:strong spikes with equal tails} in the supplemental materials. Similarly, we use the true number of strongly spiked eigenvalues for T-DBDA. Using $\phi_{\textup{b-PRD}, \hat\alpha_1}$ and $\phi_{\textup{b-PRD},\hat\alpha_2}$ requires correctly identifying $\Dc$, which is demanding when $\beta_1 = \beta_2 = 1$, $m > m_1$ and $\tau_1 > \tau_2$ as explained in Section~\ref{sec:test data piling signal subspace} and Appendix~\ref{app:asymptotic properties of sample covariance matrix} of the supplemental materials. In this case, we set $\Dc = \{1, \ldots, m_1+m_2, n_1, \ldots, n_1+m_2-1\}$, including all sample eigenvectors that can potentially capture the variability within $\Uc$ in the signal subspace $\Sc$. It results in adding $m_2$ noisy directions to $\vv_{\alpha}$, but their influence becomes negligible as $p$ increases when $\alpha := \hat\alpha_2$ is used. For $\phi_{\textup{SMDP-I}}$ and $\phi_{\textup{SMDP-II}}$, we set $n_{1,te} = n_{2,te} = 6$ so that $\Xc_{te}$ consists of $30\%$ of original training data $\Xc$ and set the number of repetitions as $K = 10$. 

\begin{table}
\caption{Averaged estimates of classification accuracy, DegSDP in (\ref{eq:DegSDP}), and ${\rm Angle}(\wv, \Uc)$ of each classification rule for each setting of the simulation study. Standard deviations are presented in parentheses.}
\centering
{\scriptsize
\begin{tabular}{cccccccccc}
\noalign{\smallskip}\noalign{\smallskip}

 & Metric & $\phi_{\textup{MDP}}$ &  $\phi_{\textup{PRD}, \hat\alpha}$ & $\phi_{\textup{b-PRD}, \hat\alpha_1}$ & $\phi_{\textup{b-PRD}, \hat\alpha_2}$ & $\phi_{\textup{SMDP-I}}$ & T-DBDA & DT & DWD  \\
\hline
\multirow{4}{*}{I} & \multirow{2}{*}{Accuracy} & 0.890  & 0.890 & \bf{0.999} & \bf{0.999} & \bf{0.999} & \bf{0.999} & \bf{0.999} & 0.893 \\
 & & (0.028) & (0.028) & (0.001) & (0.001) & (0.001) & (0.001) & (0.001) & (0.027)  \\
\cline{2-10}
& \multirow{2}{*}{DegSDP} & {12.819} & 12.819 & 12.819 & 12.819 & {12.574} & 12.535 & \bf{12.956} & 12.596 \\
 & & (1.443) & (1.443) & (1.443) & (1.443) & (1.442) & (1.446) & (1.487) & (1.445) \\
\hline
\multirow{6}{*}{II} & \multirow{2}{*}{Accuracy} & 0.846 & \bf{0.996} & \bf{0.996} & \bf{0.996} & 0.993 & 0.926 & 0.705 & 0.784 \\
 &  & (0.081) & (0.002) & (0.002) & (0.002) & (0.006) & (0.050)  & (0.093) & (0.094) \\ 
\cline{2-10}
& \multirow{2}{*}{DegSDP} & {1.454} & 7.473 & 7.469 & \bf{7.474} & {6.919}  & 2.843 & 0.412 & 0.878 \\
 & & (1.356) & (0.948) & (0.946) & (0.952) & (1.100) & (1.639) & (0.625) & (1.021)  \\ 
\cline{2-10}
& \multirow{2}{*}{${\rm Angle}(\wv, \Uc)$} & {83.124} & 89.390 & 89.388 & \bf{89.393} & {89.158}  & 86.291 & 72.567 & 79.894 \\
 & & (3.476)  & (0.301) & (0.299) & (0.304) & (0.500) & (1.708) & (8.394) & (4.997) \\ 
\hline
\multirow{6}{*}{III} & \multirow{2}{*}{Accuracy} & 0.712 & 0.602 & 0.876 & 0.998  & \bf{0.999} & 0.939 & 0.708 & 0.702 \\
 &  & (0.050) & (0.048) & (0.140)  & (0.009) & (0.001) & (0.049) & (0.092) & (0.047) \\ 
\cline{2-10}
& \multirow{2}{*}{DegSDP} & {2.642} & 6.428 & 3.029 & 11.966 & \bf{12.172} & 3.591 & 0.457 & 1.263 \\
 & & (2.414) & (4.013) & (3.392) & (2.178) & (1.747) & (2.354) & (0.952) & (1.568) \\ 
\cline{2-10}
& \multirow{2}{*}{${\rm Angle}(\wv, \Uc)$} & {84.662} & 86.947 & 82.721 & 89.370 & \bf{89.386} & 85.905 & 70.853 & 80.412 \\
 & & (2.725) & (2.418) & (5.849) & (0.594) & (0.306) & (1.943)  & (8.911) & (4.912) \\ 
\hline
\multirow{6}{*}{IV} & \multirow{2}{*}{Accuracy} & 0.786 & \bf{0.980} & \bf{0.980} & \bf{0.980} & 0.974 & 0.858 & 0.671 & 0.731 \\
 &  & (0.065) & (0.009) & (0.009) & (0.009) & (0.015) & (0.060) & (0.063) & (0.071) \\ 
\cline{2-10}
& \multirow{2}{*}{DegSDP}& {0.684} & \bf{4.550} & \bf{4.550} & 4.546 & {4.135} & 1.347 & 0.195 & 0.411 \\
 & & (0.504) & (0.705) & (0.705) & (0.707) & (0.843) & (0.684) & (0.169) & (0.365) \\ 
\cline{2-10}
& \multirow{2}{*}{${\rm Angle}(\wv, \Uc)$} & {80.471} & \bf{89.138} & \bf{89.138} & \bf{89.138} & {88.832} & 84.489 & 67.875 & 76.374 \\
 & & (3.484) & (0.410) & (0.412) & (0.410) & (0.691) & (2.140) & (7.128) & (4.837) \\ 
\hline
\multirow{6}{*}{V} & \multirow{2}{*}{Accuracy} & 0.688 & 0.578 & 0.833 & 0.947 & \bf{0.992} & 0.872 & 0.672 & 0.680 \\
 & & (0.041)  & (0.050) & (0.155) & (0.046)  & (0.010) & (0.065) & (0.064) & (0.035) \\ 
\cline{2-10}
& \multirow{2}{*}{DegSDP} & {1.003} & 3.071 & 1.826 & 3.956 & \bf{7.113} & 1.598 & 0.189 & 0.492 \\
 & & (0.709) & (1.982) & (2.092)  & (2.175) & (1.441) & (0.940) & (0.156) & (0.397) \\ 
\cline{2-10}
& \multirow{2}{*}{${\rm Angle}(\wv, \Uc)$} & {81.201} & 85.802 & 82.646 & 86.065 & \bf{88.995} & 83.871 & 65.605 & 76.243 \\
 & & (3.417) & (2.353)  & (5.877) & (2.660) & (0.609) & (2.468) & (7.510) & (4.896) \\ 
\end{tabular}
}
\label{table:simulation}
\end{table}

We repeat this procedure $100$ times and report the average of each metric in Table~\ref{table:simulation}. We exclude results of SVM since the performances of $\phi_{\textup{MDP}}$ and SVM are almost similar to each other. Also, we exclude results of $\phi_{\textup{b-MDP}}$ since it is equivalent to $\phi_{\textup{b-PRD},\hat\alpha_1}$ and $\phi_{\textup{b-PRD},\hat\alpha_2}$ in Setting I, and the performances of $\phi_{\textup{b-PRD},\hat\alpha_1}$ and $\phi_{\textup{b-PRD},\hat\alpha_2}$ are uniformly better than that of $\phi_{\textup{b-MDP}}$ in the other settings. Moreover, we exclude results of $\phi_{\textup{SMDP-II}}$ since the performances of $\phi_{\textup{SMDP-I}}$ and $\phi_{\textup{SMDP-II}}$ are almost similar to each other. 

The simulation results in Table~\ref{table:simulation} are consistent with the theoretical findings in this paper. Our results show that both $\phi_{\textup{SMDP-I}}$ and $\phi_{\textup{SMDP-II}}$ achieve successful classification performances in all of the settings. Note that in Setting V, only $\phi_{\textup{SMDP-I}}$ and $\phi_{\textup{SMDP-II}}$ achieve nearly perfect classification among the classification rules. Also, $\smdp$ yields a significant DegSDP and is nearly orthogonal to $\Uc$ in all of the settings. These results confirm that our approach, projecting $\mdp$ onto the nullspace of the common leading eigenspace, successfully works under various heterogeneous covariance models. Note that we also exclude results of DBDA, since T-DBDA shows better performances than these classification rules in all settings. It implies that the data transformation technique, which removes excessive variability within the leading eigenspace, contributes to achieve better classification performances. 

Our results also reveal the conditions under which the MDP classification rules or the PRD classification rules achieve asymptotic perfect classification (as summarized in Table~\ref{table:summary}). For the case of weak spikes (Setting I), asymptotic perfect classification is possible by utilizing $\mdp$ with the bias-correction strategy. For the case of strong spikes with equal tail eigenvalues (Settings II and IV), all PRD classification rules (as well as SMDP algorithms) perform significantly better than the other classification rules. Moreover, the normal vector of each of these classification yields considerable DegSDP and is nearly orthogonal to $\Uc$. In contrast, for the case of strong spikes with unequal tail eigenvalues (Settings III and V), performances of PRD classification rules depend on the structure of leading eigenspaces. To be specific, in Setting III where $m = m_2 > m_1$, $\phi_{\textup{b-PRD}, \hat\alpha_2}$ achieves nearly perfect classification while $\phi_{\textup{PRD}, \hat\alpha}$ and $\phi_{\textup{b-PRD}, \hat\alpha_1}$ do not (see Appendix~\ref{app:strong spikes with unequal tails} in the supplemental materials for a detailed discussion). However, in Setting V where $m > \max{(m_1, m_2)}$, all PRD classification rules do not achieve nearly perfect classification. 

\subsection{Olivetti Faces Data Example}\label{sec:real data examples}
We use the Olivetti faces dataset (available at \url{http://cs.nyu.edu/~roweis/data.html}), which contains ten face images for each of 40 distinct subjects ($n = 400$) in $p = 64 \times 64 = 4096$ pixels. The pixel intensity values are represented as integers ranging from $0$ to $255$, which we normalize to the interval $[-1, 1]$. We focus on the binary classification problem of distinguishing between individuals wearing glasses and those not wearing glasses. There are $119$ images of individuals wearing glasses ($n_1 = 119$) and $281$ images of individuals not wearing glasses ($n_2 = 281$). 

\begin{figure}
    \centering
    \includegraphics[width=1\linewidth]{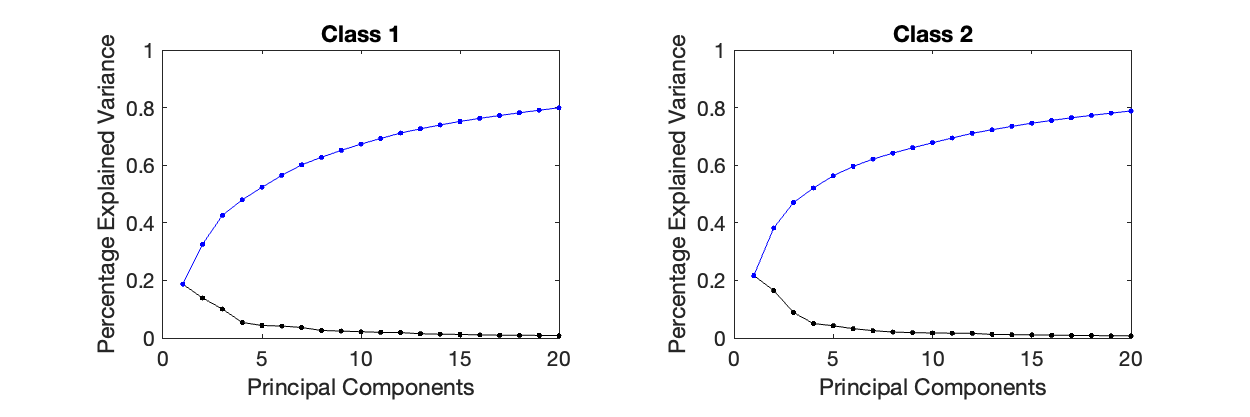}
    \caption{The proportion (black line) and cumulative proportion (blue line) of the total variance explained by each principal component for Class 1 (wearing glasses) and Class 2 (not wearing glasses) in the Olivetti dataset.}
    \label{fig:olivetti screeplot}
\end{figure}

Figure~\ref{fig:olivetti screeplot} shows the scree plots of the eigenvalues for the data from each class, suggesting that it is reasonable to estimate the number of leading eigenvalues for each class as $m_1 = 3$ and $m_2 = 3$. For T-DBDA, we estimate the number of strongly spiked eigenvalues as described in \citet{aoshima2019}, which also yields $m_1 = 3$ and $m_2 = 3$. Note that, in this case, the dimension of the common leading eigenspace, $m$, is between $\max(m_1,m_2) = 3$ and $m_1+m_2 = 6$. Also, it can be checked that both classes have heterogeneous covariance structures. For $k = 1, 2$, let $\hat\Uv_{(k),1} = [\hat\uv_{(k),1}, \ldots, \hat\uv_{(k),m_k}]$ where $\hat\uv_{(k),i}$ is the $i$th eigenvector of the sample covariance matrix of the $k$th class and let $\hat\Uc_{(k)} = {\rm span}(\hat\Uv_{(k),1})$. We find that the largest canonical angle between $\hat\Uc_{(1)}$ and $\hat\Uc_{(2)}$ is $49.198^{\circ}$.   

We randomly split the data into training and test sets by 4:1. This experiment was repeated $100$ times to calculate the average test error rates and $\textup{DegSDP}$, which are presented in Tables~\ref{table:olivetti error} and~\ref{table:olivetti SDPdeg}, respectively. Our results indicate that $\phi_{\textup{SMDP-II}}$ exhibits the lowest classification error rates for all choices of $m = 3, 4, 5, 6$. Moreover, the normal vector $\smdp$ of $\phi_{\textup{SMDP-I}}$ and $\phi_{\textup{SMDP-II}}$ yields the maximum DegSDP among the classification rules, which implies that the SMDP algorithms effectively identify a direction that yields second maximal data piling. Note that $\phi_{\textup{SMDP-I}}$ also achieves competitive classification performance compared to other classification rules. However, $\phi_{\textup{PRD},\hat\alpha}$ exhibits slightly worse performance than $\phi_{\textup{SMDP-I}}$, $\phi_{\textup{SMDP-II}}$ and even $\phi_{\textup{MDP}}$ despite the fact that both classes have heterogeneous covariance structures. This highlights the impact of unequal tail eigenvalues of the covariance matrices on the second data piling phenomenon. Meanwhile, T-DBDA achieves the best classification performance among T-DBDA, DT, and DBDA, followed by DT and then DBDA. Note that T-DBDA removes the first three leading components, DT removes only the first leading component, and DBDA does not remove any leading components for each class. This result is consistent with our observation that removing the leading eigenspace is beneficial.

\begin{table}
\caption{Classification error rates of each classification rule applied to the binary classification of the Olivetti faces dataset. Standard deviations are presented in parentheses.}
\centering
{
\begin{tabular}{ccccccccccc}
\noalign{\smallskip}\noalign{\smallskip}
$m$ & $\phi_{\textup{MDP}}$ &  $\phi_{\textup{PRD}, \hat\alpha}$ & $\phi_{\textup{SMDP-I}}$ & $\phi_{\textup{SMDP-II}}$ & DBDA & T-DBDA & DT & DWD & SVM \\
\hline
\multirow{2}{*}{$3$} & & 0.057 & 0.040 & \bf{0.016} & & & & & \\
 &  & (0.023) & (0.019) & (0.014) & & & & & \\
\cline{1-1}\cline{3-5}
\multirow{2}{*}{$4$} & & 0.079 & 0.041 & \bf{0.016} & & & & & \\
 & 0.042 & (0.028) & (0.019) & (0.015) & 0.198 & 0.140 & 0.176 & 0.074 & 0.023 \\
\cline{1-1}\cline{3-5}
\multirow{2}{*}{$5$} & (0.020) & 0.072 & 0.041 & \bf{0.017} & (0.045) & (0.042) & (0.040) & (0.028) & (0.013) \\
 &  & (0.027) & (0.018) & (0.015) & & & & & \\
\cline{1-1}\cline{3-5}
\multirow{2}{*}{$6$} & & 0.068 & 0.041 & \bf{0.016} & & & & & \\
 &  & (0.027) & (0.018) & (0.014) & & & & & \\
\end{tabular}
}
\label{table:olivetti error}
\end{table}

\begin{table}
\caption{DegSDP in (\ref{eq:DegSDP}) of each classification rule applied to the binary classification of the Olivetti faces dataset. Standard deviations are presented in parentheses.}
\centering
{
\begin{tabular}{ccccccccccc}
\noalign{\smallskip}\noalign{\smallskip}
$m$ & $\phi_{\textup{MDP}}$ &  $\phi_{\textup{PRD}, \hat\alpha}$ & $\phi_{\textup{SMDP-I}}$ & $\phi_{\textup{SMDP-II}}$ & DBDA & T-DBDA & DT & DWD & SVM \\
\hline
\multirow{2}{*}{$3$} & & 3.492 & \multicolumn{2}{c}{\bf{5.377}} & & & & & \\
 &  & (0.930) & \multicolumn{2}{c}{(1.409)} & & & & & \\
\cline{1-1}\cline{3-5}
\multirow{2}{*}{$4$} & & 2.504 & \multicolumn{2}{c}{\bf{5.326}} & & & & & \\
 & 5.048 & (0.734) & \multicolumn{2}{c}{(1.469)} & 0.377 & 1.133 & 0.819 & 1.419 & 3.013 \\
\cline{1-1}\cline{3-5}
\multirow{2}{*}{$5$} & (1.323) & 2.656 & \multicolumn{2}{c}{\bf{5.316}} & (0.169) & (0.262) & (0.287) & (0.468) & (0.584) \\
 &  & (0.770) & \multicolumn{2}{c}{(1.476)} & & & & & \\
\cline{1-1}\cline{3-5}
\multirow{2}{*}{$6$} & & 2.761 & \multicolumn{2}{c}{\bf{5.316}} & & & & & \\
 &  & (0.815) & \multicolumn{2}{c}{(1.471)} & & & & & \\
\end{tabular}
}
\label{table:olivetti SDPdeg}
\end{table}


\begin{funding}
This work was supported by National Research Foundation of Korea (NRF) grants funded by the Korea government (MSIT) (No.  
      2021R1A2C1093526, 
      RS-2022-NR068758, 
      RS-2023-00218231, 
      RS-2023-00301976, 
      RS-2024-00333399).  
\end{funding}

\begin{supplement}
\stitle{Supplemental materials for the paper ``Optimal Test-Data Piling in HDLSS Classification with Covariance Heterogeneity''.}
\sdescription{
The supplement contains technical details, case-by-case discussions of the second data piling phenomenon, and proofs of the main theorems.}
\end{supplement}


\bibliographystyle{imsart-nameyear} 
\bibliography{bib}       

\newpage
The supplemental materials are organized as follows. Additional technical assumptions are given in Appendix~\ref{app:assumptions}. We provide asymptotic properties of high-dimensional sample within-scatter matrix $\Sv_W$ for the case of strong spikes (i.e., $\beta_1 = \beta_2 = 1$) in Appendix~\ref{app:asymptotic properties of sample covariance matrix}. In Appendices~\ref{app:weak spikes}--\ref{app:strong and weak spikes}, we provide case-by-case discussions of the second data piling phenomenon depending on $\beta_1, \beta_2 \in [0,1]$ (see the table below). The proofs of main lemmas and theorems are contained in Appendix~\ref{app:proof of main results}. 

\begin{center}
    \begin{tabular}[h]{cccc}
    \noalign{\smallskip}\noalign{\smallskip}
    {Setting}  & {$0 \leq \beta_1, \beta_2 < 1$}  & {$\beta_1 = \beta_2 = 1$}  & {$0 \leq \beta_2 < \beta_1 = 1$ or $0 \leq \beta_1 < \beta_2 = 1$} \\
    \hline
    $\tau_1 = \tau_2$ & \multirow{2}{*}{Appendix~\ref{app:weak spikes}}  & Appendix~\ref{app:strong spikes with equal tails} & \multirow{2}{*}{Appendix~\ref{app:strong and weak spikes}} \\
    \cline{1-1}\cline{3-3}
    $\tau_1 \ne \tau_2$ &  & Appendix~\ref{app:strong spikes with unequal tails} &  \\
    \end{tabular}
\end{center}
\begin{appendix}
\section{Additional Assumptions}\label{app:assumptions}
We introduce additional assumptions regarding the limiting angles between  leading eigenvectors and the population mean difference vector $\muv$. Assumption~\ref{assume:4} specifies limiting angles between leading eigenvectors of each class and $\muv$. Without loss of generality, we assume $\uv_{(k),i}^\top \muv \ge 0$ for all $k = 1, 2$ and $1 \leq i \leq m_k$.

\begin{Assumption}\label{assume:4}
For $\theta_{k,i}\in [0, \pi/2]$, ${\rm Angle}(\uv_{(k),i}, \muv)\to \theta_{k,i}$ as $p\to\infty$ for $1 \leq i \leq m_k$ and $k = 1, 2$. 
\end{Assumption}

We write an orthogonal basis of $\Uc$ as $\Uv_1 = [\uv_1, \ldots, \uv_m]$. Without loss of generality, we assume $\uv_{i}^\top \muv \ge 0$ for all $1 \leq i \leq m$. Note that there exist orthogonal matrices $\Rv_{k}^{(p)} \in \Real^{m \times m_k}$ satisfying $\Uv_{(k),1} = \Uv_{1}\Rv_{k}^{(p)}$ for $k = 1, 2$. The matrix $\Rv_{k}^{(p)}$ catches the angles between the $m_k$ leading eigenvectors in $\Uv_{(k),1}$ and the $m$ basis in $\Uv_{1}$. We assume the following.

\begin{Assumption}\label{assume:5}
For $\theta_{i}\in [0, \pi/2]$, ${\rm Angle}(\uv_{i}, \muv)\to \theta_{i}$ as $p\to\infty$ for $1 \leq i \leq m$. For an orthogonal matrix $\Rv_{k} \in\Real^{m \times m_{k}}$, $\Rv^{(p)}_{k}\to\Rv_{k}$ as $p\to\infty$ for $k = 1, 2$. Moreover, $\Rv = [\Rv_{1},~\Rv_{2}] \in \Real^{m \times (m_1+m_2)}$ is of rank $m$.
\end{Assumption}

Recall that $\varphi$ denotes the limiting angle between $\muv$ and $\Uc$. Then we have $\cos^2\varphi = \sum_{i=1}^{m} \cos^2\theta_{i}$. Similarly, let $\varphi_k$ denote the limiting angle between $\muv$ and $\Uc_{(k)} = {\rm span}(\Uv_{(k),1})$, which is the leading eigenspace of the $k$th class, for $k = 1, 2$. Then we have $\cos^2\varphi_1 = \sum_{i=1}^{m_1} \cos^2\theta_{1,i}$, $\cos^2\varphi_2 = \sum_{i=1}^{m_2} \cos^2\theta_{2,i}$. 

\section{Asymptotic Properties of High-dimensional Sample Within-scatter Matrix}\label{app:asymptotic properties of sample covariance matrix}
In this section, we investigate the asymptotic behavior of sample eigenvalues and eigenvectors of $\Sv_W$, which plays a critical role in characterizing the signal subspace $\Sc$ in Section~\ref{sec:test data piling signal subspace}. 

Assume that $\beta_1 = \beta_2 = 1$ and $m_1, m_2 \ge 1$. Note that the sphered data matrix of $\Xv_k = [X_{k1}, \ldots, X_{kn_k}]$ is 
$$\Zv_{k} = \Lambdav_{(k)}^{-1/2}\Uv_{(k)}^\top(\Xv_k - \muv_{(k)} \1v_{n_k}^\top) = [\zv_{k,1}, \ldots, \zv_{k,p}]^\top \in \Real^{p \times n_k}$$ 
for $k = 1, 2$. Then the elements of $\Zv_{k}$ are uncorrelated with each other, and have mean zero and unit variance. For each $k = 1, 2$, denote the $n_k \times m_{k}$ matrix of the leading $m_{k}$ principal component scores of the $k$th class as $\Wv_{k}=[\sigma_{k,1} \zv_{k,1}, \ldots, \sigma_{k,m_{k}} \zv_{k,m_{k}}] \in \Real^{n_k \times m_k}$. Also, denote the scaled covariance matrix of the principal component scores in the $m_k$ leading eigenvectors $\{\uv_{(k),i}\}_{i=1}^{m_k}$ of the $k$th class as 
\begin{align}\label{eq:covariance of pc scores}
    \Phiv_{k} = \Wv_{k}^\top\left(\Iv_{n_k} - \frac{1}{n_k}\Jv_{n_k}\right)\Wv_{k} \in \Real^{m_k \times m_k}
\end{align}
where $\Jv_{n_k}$ is the matrix of size $n_k \times n_k$ whose all entries are $1$. Let 
$$\Wv=[\Rv_{1}\Wv_{1}^\top, ~ \Rv_{2}\Wv_{2}^\top]^\top \in \Real^{n \times m}$$ and denote the scaled covariance matrix of the principal component scores in the $m$ common leading eigenvectors $\{\uv_i\}_{i=1}^{m}$ as
\begin{align}\label{eq:covariance of commom pc scores}
    \Phiv = \Wv^\top (\Iv_n - \Jv)\Wv \in \Real^{m \times m}
\end{align}
where $\Jv = \begin{pmatrix} \frac{1}{n_1}\Jv_{n_1} & \Ov_{n_1 \times n_2} \\ \Ov_{n_2 \times n_1} & \frac{1}{n_2}\Jv_{n_2} \end{pmatrix}$. Finally, let
\begin{align}\label{eq:perturbed covariance}
\Phiv_{\tau_1, \tau_2} = \begin{pmatrix} \Phiv_{1} + \tau_1^2\Iv_{m_1} & \Phiv_{1}^{1/2}\Rv_{1}^\top \Rv_{2} \Phiv_{2}^{1/2} \\ \Phiv_{2}^{1/2}\Rv_{2}^\top \Rv_{1}\Phiv_{1}^{1/2} & \Phiv_{2} + \tau_2^2 \Iv_{m_2} \end{pmatrix} \in \Real^{(m_1+m_2) \times (m_1+m_2)}.
\end{align}
Note that $\Phiv_{\tau_1, \tau_2}$ can be understood as a perturbed version of $\Phiv$ with noises $\tau_1^2$ and $\tau_2^2$. For any square matrix $\Mv \in \Real^{l \times l}$ ($l \in \mathbb{N}$), let $\phi_i(\Mv)$ and $v_i(\Mv)$ denote the $i$th largest eigenvalue of $\Mv$ and its corresponding eigenvector, respectively. If $\tau_1 = \tau_2 =: \tau$, then the eigenvalues of $\Phiv$ and $\Phiv_{\tau_1, \tau_2}$ are closely related in the sense that
\begin{align}\label{eq:relation between Phi and Phi2}
    \phi_i(\Phiv_{\tau_1, \tau_2}) = \begin{cases} \phi_i(\Phiv) + \tau^2, & 1 \leq i \leq m, \\ \tau^2, & m+1 \leq i \leq m_1+m_2. \end{cases}
\end{align}
However, (\ref{eq:relation between Phi and Phi2}) does not hold if $\tau_1 \ne \tau_2$. Both of $\Phiv$ and $\Phiv_{\tau_1, \tau_2}$ will play a critical role in explaining the asymptotic behavior of the eigenvalues and eigenvectors of $\Sv_W$, which is quite different depending on whether both covariance matrices have equal tail eigenvalues or unequal tail eigenvalues. 

Lemma~\ref{lem:asymptotic property of Sw equal tails} shows that if $\tau_1 = \tau_2$, then the first $m$ leading eigenvectors of $\Sv_W$ explain the variation within the common leading eigenspace $\Uc$. The other sample eigenvectors are asymptotically orthogonal to $\Uc$, which implies that these eigenvectors do not capture the variability within $\Uc$. Recall that $\hat\lambda_i$ and $\hat\uv_i$ are the $i$th largest eigenvalue and the corresponding eigenvector of $\Sv_W$, respectively. Also, $\Uc$ is the common leading eigenspace of both classes. 

\begin{lemma}\label{lem:asymptotic property of Sw equal tails}
Suppose Assumptions \ref{assume:1}---\ref{assume:5} hold. Also, assume $\beta_1 = \beta_2 = 1$, $\tau_1 = \tau_2 =: \tau$ and $m_1, m_2 \ge 1$. Then conditional to $\Wv_{1}$ and $\Wv_{2}$, the following hold as $p \to \infty$.
\begin{itemize}
    \item[(i)] 
    \begin{align*}
        p^{-1} \hat\lambda_i \xrightarrow{P} \begin{cases}
        \phi_i(\Phiv) + \tau^2, & 1 \leq i \leq m, \\
        \tau^2, & m+1 \leq i \leq n-2.
        \end{cases}
    \end{align*}
    \item[(ii)] 
    \begin{align*}
        \cos{\left({\rm Angle}{(\hat\uv_i, \Uc)}\right)} \xrightarrow{P} \begin{cases}
            C_i, & 1 \leq i \leq m, \\
            0, & m + 1 \leq i \leq n-2
        \end{cases}
    \end{align*}
    where 
    \begin{align}\label{eq:Ci}
        C_i = \sqrt{\frac{\phi_i(\Phiv)}{\phi_i(\Phiv) + \tau^2}} > 0.
    \end{align} 
\end{itemize}
\end{lemma}

In contrast to the case of $\tau_1 = \tau_2$, Lemma~\ref{lem:asymptotic property of Sw unequal tails} shows that if $\tau_1 > \tau_2$, then some non-leading eigenvectors of $\Sv_W$ may capture the variability within $\Uc$ instead of leading eigenvectors.

\begin{lemma}\label{lem:asymptotic property of Sw unequal tails}
Suppose Assumptions \ref{assume:1}---\ref{assume:5} hold. Also, assume $\beta_1 = \beta_2 = 1$ and $\tau_1 > \tau_2$ and $m_1, m_2 \ge 1$. Then conditional to $\Wv_{1}$ and $\Wv_{2}$, the following hold as $p \to \infty$.

\begin{itemize}
    \item[(i)] 
    \begin{align*}
        p^{-1} \hat\lambda_i \xrightarrow{P} \begin{cases}
        \phi_i(\Phiv_{\tau_1, \tau_2}),  & 1 \leq i \leq k_0, \\
        \tau_1^2, & k_0 + 1 \leq i \leq k_0 + (n_1 - m_1 - 1), \\
        \phi_{i-(n_1 - m_1 -1)}(\Phiv_{\tau_1, \tau_2}), & k_0 + (n_1 - m_1) \leq i \leq n_1 + m_2 - 1, \\
        \tau_2^2, & n_1 + m_2 \leq i \leq  n-2,
        \end{cases}
    \end{align*}
    where $k_0$ $(m_1 \leq k_0 \leq m_1+m_2)$ is the (random) integer satisfying $\phi_{k_0}(\Phiv_{\tau_1, \tau_2}) > \tau_1^2 \ge \phi_{k_0+1}(\Phiv_{\tau_1, \tau_2})$ where we use the convention of $\phi_{m_1+m_2+1}(\Phiv_{\tau_1, \tau_2}) = 0$. Furthermore, if $m = m_1$, then $k_0 = m_1$.
    
    \item[(ii)] 
    \begin{align*}
        \cos{\left({\rm Angle}{(\hat\uv_i, \Uc)} \right)} \xrightarrow{P} \begin{cases}
            D_i, & 1 \leq i \leq k_0, \\
            0, & k_0+1 \leq i \leq k_0 + (n_1-m_1-1), \\
            D_{i-(n_1-m_1-1)}, & k_0 + (n_1 - m_1) \leq i \leq n_1+m_2-1, \\ 
            0, & n_1+m_2 \leq i \leq n-2,
        \end{cases}
    \end{align*}
    where $k_0$ is defined in Lemma~\ref{lem:asymptotic property of Sw unequal tails} (i),
    \begin{align}\label{eq:Di}
        D_i = \sqrt{\frac{\|\sum_{k=1}^{2}\Rv_{k}\Phiv_{k}^{1/2}\tilde{v}_{ik}(\Phiv_{\tau_1, \tau_2}) \|_2^2}{\phi_i(\Phiv_{\tau_1, \tau_2})} } > 0
    \end{align}
    and $v_{i}(\Phiv_{\tau_1, \tau_2}) = (\tilde{v}_{i1}(\Phiv_{\tau_1, \tau_2})^\top, \tilde{v}_{i2}(\Phiv_{\tau_1, \tau_2})^\top)^\top$ with $\tilde{v}_{ik}(\Phiv_{\tau_1, \tau_2}) \in \Real^{m_k}$ for $k = 1, 2$.
\end{itemize}
\end{lemma}

\begin{remark}
If $\tau_1 = \tau_2 =: \tau$, then $D_i$ in (\ref{eq:Di}) becomes $C_i$ in (\ref{eq:Ci}) for $1 \leq i \leq m$, and $D_i = 0$ for $m+1 \leq i \leq m_1+m_2$.
\end{remark}

In Lemma~\ref{lem:asymptotic property of Sw unequal tails}, we have seen that, if $\tau_1 > \tau_2$ and $m = m_1$, then $k_0 = m_1$ with probability $1$. However, if $m > m_1$, then $k_0$ is a random number depending on true leading principal component scores $\Wv_{1}$ and $\Wv_{2}$.

\section{The Case of Weak Spikes and the Maximal Data Piling Classification Rule}\label{app:weak spikes}
In this section, we study the special case of weak spikes. Recall that asymptotic results under the weak spikes model (i.e., $\beta_{k} < 1$) are equivalent to those with $m_{k} = 0$ and $\beta_k = 1$ under Assumption~\ref{assume:2} for $k = 1, 2$ (see Remark~\ref{rmk:asymptotic regime}). Hence, we focus on the case of $m_1 = m_2 = 0$, which implies $m = 0$. 

In Theorem~\ref{thm:test data piling}, we have seen that projections of test data $\Yc_1$ and $\Yc_2$ onto $\mdp$ are asymptotically piled on two distinct points, respectively. To be specific, for any independent observation $Y \in \Yc$,
\begin{align}\label{eq:MDP projection}
    \frac{1}{\sqrt{p}} \mdp^\top(Y-\bar{X}) \xrightarrow{P} \begin{cases} \kappa^{-1}(\eta_2\delta^2 - (\tau_1^2 - \tau_2^2) / n), & \pi(Y) = 1, \\ 
    \kappa^{-1}(-\eta_1\delta^2 - (\tau_1^2 - \tau_2^2) / n), & \pi(Y) = 2
    \end{cases}
\end{align}
as $p\to\infty$ where $\kappa$ is the probability limit of $\kappa_{\textup{MDP}}$ defined in (\ref{eq:kMDP}), which implies that $\mdp$ also yields second data piling of independent test data. In fact, it can be shown that $\mdp$ is a second \emph{maximal} data piling direction. 
In this non-spike case, the signal subspace $\Sc$ is ${\rm span}(\mdp)$, and any direction $\wv \in \Sc_{\Xc} \setminus \Sc = {\rm span}(\{ \hat\uv_i \}_{i=1}^{n-2})$ yields asymptotically zero piling distance of independent test data. Hence, $\mdp$ gives an asymptotic maximal distance between two piles of independent test data among second data piling directions.

From (\ref{eq:MDP projection}), we can also check that the original maximal data piling classification rule \citep{Ahn2010},
\begin{align}\label{eq:MDPrule}
    \phi_\textup{MDP}(Y ; \Xc) = \begin{cases}
    1, & \mdp^\top(Y - \bar{X}) \ge 0, \\
    2, & \mdp^\top(Y - \bar{X}) < 0,
    \end{cases}
\end{align}
achieves perfect classification if $\tau_1 = \tau_2$ (even under heterogeneous covariance structures). However, if $\tau_1 \ne \tau_2$, then $\phi_{\textup{MDP}}$ may fail to achieve perfect classification, since the total mean threshold in $\phi_{\textup{MDP}}$ should be adjusted by the bias term in (\ref{eq:MDP projection}). We define the bias-corrected maximal data piling classification rule as
\begin{align}\label{eq:bMDPrule}
    \phi_\textup{b-MDP}(Y ; \Xc) = \begin{cases}
    1, & p^{-1/2} \mdp^\top(Y - \bar{X}) - (\hat\alpha_1 - \hat\alpha_2) / (n \kappa_{\textup{MDP}}) \ge 0, \\
    2, & p^{-1/2} \mdp^\top(Y - \bar{X}) - (\hat\alpha_1 - \hat\alpha_2) / (n \kappa_{\textup{MDP}}) < 0,
    \end{cases}
\end{align}
where $\hat\alpha_k$ is an HDLSS-consistent estimator of $-\tau_k^2$ for $k = 1, 2$. Note that $\hat\lambda_{(k),i} / p \xrightarrow{P} \tau_k^2$ as $p \to \infty$ for $m_k + 1 \leq i \leq n_k -1$ \citep{Jung2012a}, where $\hat\lambda_{(k),i}$ is the $i$th largest eigenvalue of $\Sv_k = (\Xv_k - \bar{\Xv}_k)(\Xv_k - \bar{\Xv}_k)^\top $. Based on this fact, from now on, we fix
\begin{align}\label{eq:hatalphak}
    \hat\alpha_k = -\frac{1}{n_k - m_k - 1}\sum_{i=m_k+1}^{n_k-1} \frac{\hat\lambda_{(k),i}}{p}
\end{align}
for $k = 1, 2$. Then the bias-corrected maximal data piling classification rule achieves asymptotic perfect classification even if $\tau_1 \ne \tau_2$. 

\section{The Case of Strong Spikes with Equal Tails and the Projected Ridge Classification Rule}\label{app:strong spikes with equal tails}
In this section, we discuss the case of strong spikes with equal tail eigenvalues (i.e., $m \ge 1$ and $\tau_1 = \tau_2 =: \tau$). Recall that, in this case, $\Sc = {\rm span}(\hat\uv_1, \ldots, \hat\uv_m, \mdp)$ and second data piling occurs when a direction $\wv \in \Sc$ is asymptotically orthogonal to the common leading eigenspace $\Uc$. Theorem~\ref{thm:ridge equal tail eigenvalues piling distance} confirms that the projected ridged linear discriminant vector $\vv_{\hat\alpha} \in \Sc$ is such a direction, even if two populations do not have a common covariance matrix but equal tail eigenvalues. Here, $\hat\alpha$ is an HDLSS-consistent estimator of $-\tau^2$. Moreover, both $P_{\vv_{\hat\alpha}}\Yc_1$ and $P_{\vv_{\hat\alpha}}\Yc_2$ asymptotically pile on two distinct points, respectively. 

\begin{theorem}\label{thm:ridge equal tail eigenvalues piling distance}
Suppose Assumptions~\ref{assume:1}---\ref{assume:5} hold and assume $\beta_1 = \beta_2 = 1$ and $\tau_1 = \tau_2 =: \tau$. Then (i) for $\hat\alpha$ chosen as an HDLSS-consistent estimator of $-\tau^2$, ${\rm Angle}{(\vv_{\hat\alpha}, \uv_{i, \Sc})} \xrightarrow{P} \pi/2$ as $p \to \infty$ for $1 \leq i \leq m$. (ii) Moreover, for any independent observation $Y \in \Yc$,
\begin{align}\label{eq:ridge equal tail eigenvalues projection}
    \frac{1}{\sqrt{p}} \vv_{\hat\alpha}^\top(Y-\bar{X}) \xrightarrow{P} \begin{cases} \upsilon_0(\eta_2(1-\cos^2\varphi)\delta^2), & \pi(Y) = 1, \\ 
    \upsilon_0(-\eta_1(1-\cos^2\varphi)\delta^2), & \pi(Y) = 2
\end{cases}
\end{align}
as $p\to\infty$ where $\upsilon_0$ is defined in (\ref{eq:app:upsilon0:equal tails}). Hence, $D(\vv_{\hat\alpha}) = D(\fv_0) = \upsilon_0(1-\cos^2\varphi)\delta^2$, where $\fv_0$ is the theoretical second maximal data piling direction in Theorem~\ref{thm:SMDP piling distance}.
\end{theorem}

It should be emphasized that the asymptotic test data piling distance of $\vv_{\hat\alpha}$ is the same as that of the theoretical second maximal data piling direction $\fv_0$ in Theorem~\ref{thm:SMDP piling distance}. Recall that in Section~\ref{sec:SMDP theory}, if $\beta_1 = \beta_2 = 1$ and $\tau_1 = \tau_2$, then $\fv_0$ is the unique direction within $\Sc$ that is orthogonal to the common leading eigenspace $\Uc$ for all $p$. It can be seen that $\fv_0$ is asymptotically equivalent to the projected ridged linear discriminant vector $\vv_{\hat\alpha} \in \Sc$. 

The above argument naturally leads to the conclusion that $\vv_{\hat\alpha}$ is a second maximal data piling direction. The following corollary then directly follows from Theorem~\ref{thm:SMDP characterization}, which states that any second data piling direction is asymptotically close to a linear combination of $\vv_{\hat\alpha}$ and $\left\{ \hat\uv_i \right\}_{i=m+1}^{n-2}$ (the sample eigenvectors  strongly-inconsistent with $\Uc$) in this case. Recall that $\Ac$ (or $\W_\Xc$) represents the collection of all sequences of second data piling directions (or all directions in sample space, respectively).

\begin{corollary}
Suppose Assumptions \ref{assume:1}---\ref{assume:5} hold and assume $\beta_1 = \beta_2 = 1$ and $\tau_1 = \tau_2$. Then,
\begin{itemize}
    \item[(i)] For any given $\left\{ \wv \right\} \in \Ac$, there exists a sequence $\left\{ \vv \right\} \in \Bc$ such that $\| \wv - \vv \| \xrightarrow{P} 0$ as $p \to \infty$, where $\Bc = \left\{\left\{ \vv \right\} \in \W_\Xc : \vv \in {\rm span}(\vv_{\hat\alpha}) \oplus {\rm span}(\left\{\hat\uv_i \right\}_{i=m+1}^{n-2}) \right\}$.
    
    \item[(ii)] For any $\left\{ \wv \right\} \in \Ac$ such that $D(\wv)$ exists, $\left\{ \wv \right\} \in \Ac$ is a sequence of second maximal data piling directions if and only if $\|\wv - \vv_{\hat\alpha}\| \xrightarrow{P} 0$ as $p \to \infty$. 
\end{itemize}
\end{corollary}

We note that $\vv_{\hat\alpha}$ can be obtained purely from the training data $\Xc$. Also, $\hat\alpha$ can be estimated purely from the sample eigenvalues of the pooled within-covariance matrix $\Sv_W$, since $\hat\lambda_i / p \xrightarrow{P} \tau^2$ for $m+1 \leq i \leq n-2$ as $p \to \infty$ (see Lemma~\ref{lem:asymptotic property of Sw equal tails}). Throughout, we use 
\begin{align}\label{eq:hatalpha}
    \hat\alpha = -\frac{1}{n-m-2} \sum_{i=m+1}^{n-2}\frac{\hat\lambda_i}{p}.
\end{align}
From (\ref{eq:ridge equal tail eigenvalues projection}), the original projected ridge classification rule $\phi_{\textup{PRD}, \alpha}(Y; \Xc)$ of \citet{Chang2021} for a given $\alpha$,
\begin{align}\label{eq:PRDrule}
\phi_{\textup{PRD}, \alpha}(Y; \Xc) = \begin{cases} 1, &  \vv_{\alpha}^\top(Y-\bar{X}) \ge 0, \\ 2, &\vv_{\alpha}^\top(Y-\bar{X}) < 0, \end{cases}
\end{align}
with $\Dc = \{1, \ldots, m \}$ also achieves perfect classification when $\alpha := \hat\alpha$, as long as the tail eigenvalues are equal. 
Thus, our result extends the conclusion of \citet{Chang2021} in the sense that $\phi_{\textup{PRD}, \hat\alpha}$ yields perfect classification not only in case of $\Sigmav_{(1)} = \Sigmav_{(2)}$ but also in case of $\Sigmav_{(1)} \ne \Sigmav_{(2)}$ and $\tau_1 = \tau_2$.

\begin{remark} 
It can be shown that $\phi_{\textup{PRD}, \alpha}$ achieves perfect classification only at a negative ridge parameter $\alpha := -\tau^2$ when $\tau_1 = \tau_2 =:\tau$ with a regularizing condition on the distribution of $(\zv_{k,1}, \ldots, \zv_{k, m_k})^\top$ for $k = 1, 2$ as in Theorem 3.7 of \citet{Chang2021}.
\end{remark}

\section{The Case of Strong Spikes with Unequal Tails and the Bias-corrected Projected Ridge Classification Rule}\label{app:strong spikes with unequal tails}
In this section, we focus on the case of strong spikes with unequal tail eigenvalues (i.e., $m \ge 1$ and $\tau_1 \ne \tau_2$). In this case, when $m_1 = m_2 = 1$, Propositions~\ref{prop:ridge one-spike unequal tails m = 1} and~\ref{prop:ridge one-spike unequal tails m = 2} showed that $\vv_{\hat\alpha_k}$ is asymptotically orthogonal to $\uv_{(k),1}$. These results can be extended to the general cases where $m_1 \ge 1$ and $m_2 \ge 1$. In general, Theorem~\ref{thm:ridge unequal tails asymptotically orthogonal} tells that $\vv_{\hat\alpha_k}$ is asymptotically orthogonal to $\Uc_{(k)} = {\rm span}([ \uv_{(k),1}, \ldots, \uv_{(k),m_k}])$, which is the leading eigenspace of the $k$th class. 

\begin{theorem}\label{thm:ridge unequal tails asymptotically orthogonal}
Suppose Assumptions~\ref{assume:1}---\ref{assume:5} hold and assume $\beta_1 = \beta_2 = 1$ and $\tau_1 > \tau_2$. Then for $\hat\alpha_k$ chosen as an HDLSS-consistent estimator of $-\tau_k^2$, ${\rm Angle}{(\vv_{\hat\alpha_k}, \uv_{(k),i, \Sc})} \xrightarrow{P} \pi/2$ as $p \to \infty$ for $k = 1, 2$ and $1 \leq i \leq m_k$. 
\end{theorem}

It implies that $P_{\vv_{\hat\alpha_k}} \Yc_k$ converges to a single point as $p$ increases for each $k = 1, 2$. Hence, in cases where $\Uc_{(1)}$ includes $\Uc_{(2)}$ (or $\Uc_{(2)}$ includes $\Uc_{(1)}$), both of $\Yc_1$ and $\Yc_2$ are piled on $\vv_{\hat\alpha_1}$ (or $\vv_{\hat\alpha_2}$), respectively. 

\begin{theorem}\label{thm:ridge unequal tails piling distance}
Suppose Assumptions \ref{assume:1}---\ref{assume:5} hold and assume $\beta_1 = \beta_2 = 1$, $\tau_1 > \tau_2$. Also, for $1 \leq k \ne s \leq 2$, further assume that $m = m_k \ge m_s$ (that is, $\Uc_{(k)}$ includes $\Uc_{(s)}$). Then for any independent observation $Y \in \Yc$,
    \begin{align*}
        \frac{1}{\sqrt{p}} \vv_{\hat\alpha_k}^{\top}(Y-\bar{X}) \xrightarrow{P} \begin{cases}     \gamma_{k}(\eta_2(1-\cos^2\varphi)\delta^2 - (\tau_1^2 - \tau_2^2) / n), & \pi(Y) = 1, \\ 
        \gamma_{k}(-\eta_1(1-\cos^2\varphi)\delta^2 - (\tau_1^2 - \tau_2^2) / n), & \pi(Y) = 2 
        \end{cases}
    \end{align*}
as $p\to\infty$ where $\gamma_{k}$ is a strictly positive random variable depending on the true principal component scores of $\Xc$.
\end{theorem}

Note that Theorem~\ref{thm:ridge unequal tails piling distance} is a generalized version of Proposition~\ref{prop:ridge one-spike unequal tails m = 1} (iii). For the case of strong spikes with unequal eigenvalues, when $m = m_k$ for a $k \in\{1,2\}$, Theorem~\ref{thm:ridge unequal tails piling distance} tells that projections of independent test data are asymptotically piled on two distinct points on $\vv_{\hat\alpha_k}$.
Even though $\vv_{\hat\alpha_1}$ and $\vv_{\hat\alpha_2}$ are not the second maximal data piling directions (as discussed in Section~\ref{sec:One-spike Model with Common Leading Eigenspace}), they still yield second data piling. Hence, a classification rule utilizing $\vv_{\hat\alpha_k}$ can also achieve asymptotic perfect classification in such cases. We define the bias-corrected projected ridge classification rule as
\begin{align}\label{eq:bPRDrule}
\phi_{\textup{b-PRD}, \alpha}(Y; \Xc) = \begin{cases}
1, & p^{-1/2}\vv_{\alpha}^\top (Y - \bar{X}) - (\hat\alpha_1 - \hat\alpha_2) / (n \|\tilde{\vv}_{\alpha} \|) \ge 0, \\
2, & p^{-1/2}\vv_{\alpha}^\top (Y - \bar{X}) - (\hat\alpha_1 - \hat\alpha_2) / (n \|\tilde{\vv}_{\alpha} \|) < 0,
\end{cases}
\end{align}
where $\tilde{\vv}_{\alpha}$ is given as
\begin{align*}
    \tilde{\vv}_{\alpha} = \sum\limits_{i\in\Dc} \frac{\alpha_p}{\hat\lambda_i + \alpha_p} \hat\uv_i\left( \frac{1}{\sqrt{p}}\hat\uv_i^\top\dv\right) + \frac{1}{\sqrt{p}}\hat\Uv_2\hat\Uv_2^\top\dv.
\end{align*}
Note that $\vv_{\alpha} \varpropto \tilde{\vv}_{\alpha}$. It can be shown that if $k \in\{1,2\}$ is such that $m= m_k$, then $\| \tilde{\vv}_{\hat\alpha_k} \| \xrightarrow{P} \gamma_{k}^{-1}$ as $p \to \infty$, 
and $\phi_{\textup{b-PRD}, \hat\alpha_k}$ achieves asymptotic perfect classification. 

In practice, $\phi_{\textup{b-PRD}, \alpha}$ requires identifying $\Dc$ and determining the signal subspace $\Sc$ onto which the ridged linear discriminant vector $\wv_{\alpha}$ is projected. In case of $m = m_1$, we have seen in Lemma~\ref{lem:asymptotic property of Sw equal tails} that $k_0 = m_1$ with probability $1$ and thus one can set $\Dc = \{1, \ldots, m_1, n_1, \ldots, n_1+m_2-1\}$. However, this becomes challenging when $m = m_2> m_1$ as explained in Section~\ref{sec:test data piling signal subspace} and Appendix~\ref{app:asymptotic properties of sample covariance matrix}. In this case, one possible approach to circumvent this difficulty is to set $\Dc = \{1, \ldots, m_1+m_2, n_1, \ldots, n_1+m_2-1 \}$, including all sample eigenvectors that may potentially capture the variability within the common leading eigenspace $\Uc$ in the signal subspace $\Sc$. Even if the resulting $\vv_{\alpha}$ is affected by $m_2$ additional noisy directions, their influence becomes negligible as $p$ increases when $\alpha := \hat\alpha_2$ (defined in (\ref{eq:hatalphak})) is used. Moreover, $\vv_{\hat\alpha_2}$ yields second data piling and $\phi_{\textup{b-PRD}, \hat\alpha_2}$ achieves asymptotic perfect classification. The simulation results of Setting III in Section~\ref{sec:simulation} numerically confirm the validity of this approach.

However, as described in Section~\ref{sec:General One-spike Model with Heterogeneous Leading Eigenspace}, $\vv_{\alpha}$ does not generally yield second data piling for any ridge parameter $\alpha \in \Real$ when $m > \max(m_1, m_2)$. Hence, $\phi_{\textup{b-PRD},\alpha}$ can only be used in special cases where $m = m_1$ or $m = m_2$.

\section{The Case of Strong and Weak Spikes}\label{app:strong and weak spikes}
In this section, we discuss the cases where either (i) $\beta_1 = 1$ and $\beta_2 < 1$ or (ii) $\beta_1 < 1$ and $\beta_2 = 1$. Recall that, in the main article, we assumed $\beta_1 = \beta_2 = 1$ and defined the common leading subspace $\Uc = {\rm span}([\Uv_{(1),1}, \Uv_{(2),1}])$. However, if (i) $\beta_1 = 1$ and $\beta_2 < 1$ (or (ii) $\beta_1 < 1$ and $\beta_2 = 1$), it is more natural to define $\Uc$ as the subspace spanned by leading eigenspace of the first (or second) class, respectively. That is, if (i) $\beta_1 = 1$ and $\beta_2 < 1$, we define $\Uc = \Uc_{(1)}$ and $m = m_1$. On the other hand, if (ii) $\beta_1 < 1$ and $\beta_2 = 1$, we define $\Uc = \Uc_{(2)}$ and $m = m_2$.

Lemmas~\ref{lem:asymptotic property of Sw with b1 = 1 and b2 < 1} and~\ref{lem:asymptotic property of Sw with b1 < 1 and b2 = 1} below show the asymptotic limits of the sample eigenvalues and eigenvectors of $\Sv_W$ under the assumption of (i) and (ii), respectively.  As in the main article, we assume $\tau_1 \ge \tau_2$. The proofs of these results are quite similar to those of Lemmas~\ref{lem:asymptotic property of Sw equal tails} and~\ref{lem:asymptotic property of Sw unequal tails} (in Appendix~\ref{app:proof of main results}), so we omit them.

\begin{lemma}\label{lem:asymptotic property of Sw with b1 = 1 and b2 < 1}
Suppose Assumptions \ref{assume:1}---\ref{assume:5} hold. Also, assume $\beta_1 = 1$, $0 \leq \beta_2 < 1$ and $\tau_1 \ge \tau_2$. Then conditional to $\Wv_{1}$ and $\Wv_{2}$, as $p \to \infty$,
\begin{align*}
p^{-1} \hat\lambda_i \xrightarrow{P} \begin{cases}
        \phi_i(\Phiv_{1}) + \tau_1^2, & 1 \leq i \leq  m_1, \\
        \tau_1^2, & m_1 + 1 \leq i \leq n_1-1, \\ 
        \tau_2^2, & n_1 \leq i \leq n-2
        \end{cases}
\end{align*}
and 
\begin{align*}
    \cos{\left(\ang{(\hat\uv_i, \Uc)}\right)} \xrightarrow{P} \begin{cases}
            A_i & 1 \leq i \leq m_1, \\
            0, & m_1 + 1 \leq i \leq n-2
        \end{cases}    
\end{align*}
where 
\begin{align*}
    A_i = \sqrt{\frac{\phi_i(\Phiv_1)}{\phi_i(\Phiv_1)+\tau_1^2}}.
\end{align*}
\end{lemma}

\begin{lemma}\label{lem:asymptotic property of Sw with b1 < 1 and b2 = 1}
Suppose Assumptions \ref{assume:1}---\ref{assume:5} hold. Also, assume $0 \leq \beta_1 < 1$, $\beta_2 = 1$ and $\tau_1 \ge \tau_2$. Then conditional to $\Wv_{1}$ and $\Wv_{2}$, as $p \to \infty$,

\begin{align*}
    p^{-1} \hat\lambda_i \xrightarrow{P} \begin{cases}
    \phi_i(\Phiv_{2}) + \tau_2^2, & 1 \leq i \leq k_0, \\
    \tau_1^2, & k_0 + 1 \leq i \leq  k_0 + (n_1 - 1), \\
    \phi_{i-(n_1-1)}(\Phiv_{2}) + \tau_2^2, & k_0 + n_1 \leq i \leq n_1 + m_2 - 1,\\
    \tau_2^2, & n_1 + m_2 \leq i \leq n-2,
    \end{cases}
\end{align*}
where $k_0$ $(0 \leq k_0 \leq m_2)$ is the (random) integer satisfying $\phi_{k_0}(\Phiv_{2}) + \tau_2^2 > \tau_1^2 \ge \phi_{k_0+1}(\Phiv_{2}) + \tau_2^2$ where we use the convention that $\phi_{0}(\Phiv_{2}) = \infty$ and $\phi_{m_2+1}(\Phiv_{2}) = 0$. Moreover,
\begin{align*}
    \cos{\left({\rm Angle}{(\hat\uv_i, \Uc)}\right)} \xrightarrow{P} \begin{cases}
        B_i, & 1 \leq i \leq k_0, \\
        0, & k_0 + 1 \leq i \leq k_0 + (n_1-1), \\
        B_{i-(n_1-1)}, & k_0+n_1 \leq i \leq n_1+m_2-1, \\        
        0, & n_1+m_2 \leq i \leq n-2,
    \end{cases}
\end{align*}
where 
\begin{align*}
    B_i = \sqrt{\frac{\phi_i(\Phiv_2)}{\phi_i(\Phiv_2)+\tau_2^2}}.
\end{align*}
\end{lemma}

\begin{table}
\centering
\begin{tabular}{cccc}
\noalign{\smallskip}\noalign{\smallskip}
$\beta_1, \beta_2$ & $\tau_1, \tau_2$ & {$\Dc$} & {$|\Dc|$}\\
\hline
 $\beta_1 = 1$ & \multirow{2}{*}{$\tau_1 \ge \tau_2$} & \multirow{2}{*}{$\left\{1, \ldots, m \right\}$} & \multirow{2}{*}{$m$} \\
$\beta_2 < 1$ & & & \\
\hline
 $\beta_1 < 1$ & \multirow{2}{*}{$\tau_1 \ge \tau_2$} & \multirow{2}{*}{$\left\{1, \ldots, k_0, k_0+n_1, \ldots, n_1+m-1  \right\}$} & \multirow{2}{*}{$m$} \\
$\beta_2 = 1$ & & & \\
\end{tabular}
\caption{The index set $\Dc$ for each case. $k_0$ is a random number depending on the true principal component scores of training data $\Xc$ defined in Lemma~\ref{lem:asymptotic property of Sw with b1 < 1 and b2 = 1}. }
\label{table:D2}
\end{table}

\begin{remark}
    In Lemma~\ref{lem:asymptotic property of Sw with b1 < 1 and b2 = 1}, if $\tau_1 = \tau_2$, then $k_0 = m = m_2$ with probability $1$. However, if $\tau_1 > \tau_2$, then $k_0$ is a random number depending on the true leading principal component scores of $\Wv_2$.
\end{remark}

If (i) $\beta_1 = 1$ and $0 \leq \beta_2 < 1$, then the leading $m=m_1$ sample eigenvectors can always explain the variation within $\Uc = \Uc_{(1)}$ in the data from the first class, while the other sample eigenvectors are strongly inconsistent with $\Uc = \Uc_{(1)}$. 

If (ii) $0 \leq \beta_1 < 1$ and $\beta_2 = 1$, then $m=m_2$ sample eigenvectors explain the variation within $\Uc = \Uc_{(2)}$ in the second class. In particular, if $\tau_1 = \tau_2$, then the first $m = m_2$ sample eigenvectors account for this variation. However, if $\tau_1 > \tau_2$, some non-leading sample eigenvectors may capture the variability instead of some of the leading $m_2$ sample eigenvectors. For example, if $m_2 = 1$ and $\tau_1 = \tau_2$, then $k_0 = 1$ with probability $1$ and $\hat\uv_1$ always accounts for the variation along $\Uc = {\rm span}(\uv_{(2),1})$. However, if $m_2 = 1$ and $\tau_1 > \tau_2$, then $k_0$ is either $0$ or $1$, which means that either $\hat\uv_1$ or $\hat\uv_{n_1}$ can capture the variability along $\Uc = {\rm span}(\uv_{(2),1})$.

Let $\Sc = {\rm span}(\left\{\hat\uv_i \right\}_{i \in \Dc}, \mdp)$ with $\Dc$ be given in Table~\ref{table:D2} for each case. Then it can be shown that Theorem~\ref{thm:test data piling} also holds for the case of (i) $\beta_1 = 1, \beta_2 < 1$ or (ii) $\beta_1 < 1, \beta_2 = 1$, that is, projections of $\Yc$ onto $\Sc$ are distributed along parallel affine subspaces, one for each class. More precisely, projections of $\Yc_k$ onto $\Sc$ tend to lie on an $m$-dimensional affine subspace $\Lc_k$ if $\beta_k = 1$, or a point in $\Lc_k$ if $\beta_k < 1$. See Figure~\ref{fig:app:SDP} for an illustration.

\begin{figure}
    \centering
    \subfloat[$\beta_1 = 1, 0 \leq \beta_2 < 1$]{{\includegraphics[width=0.5\linewidth]{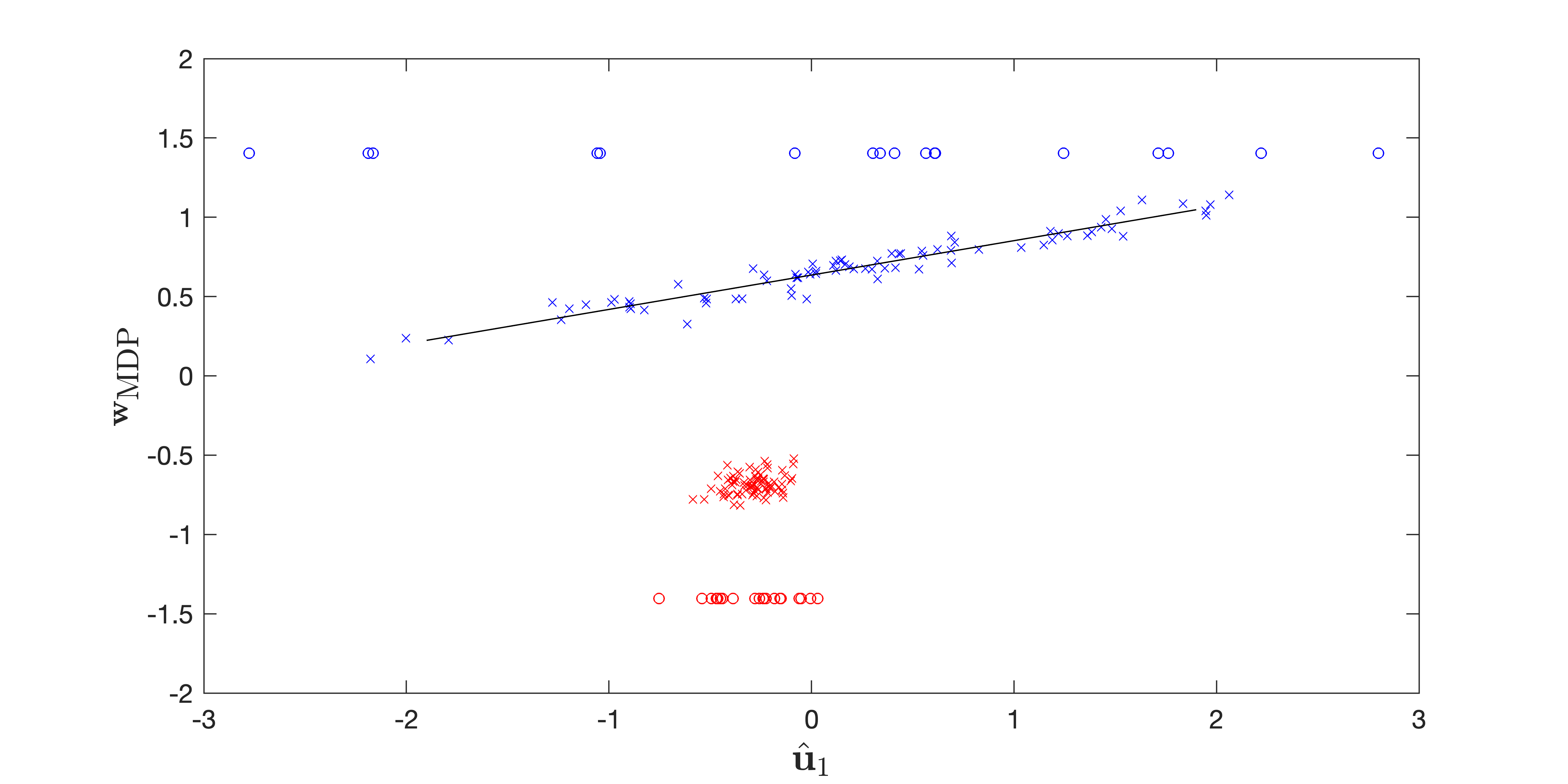} }}%
    \subfloat[$0 \leq \beta_1 < 1, \beta_2 = 1$]{{\includegraphics[width=0.5\linewidth]{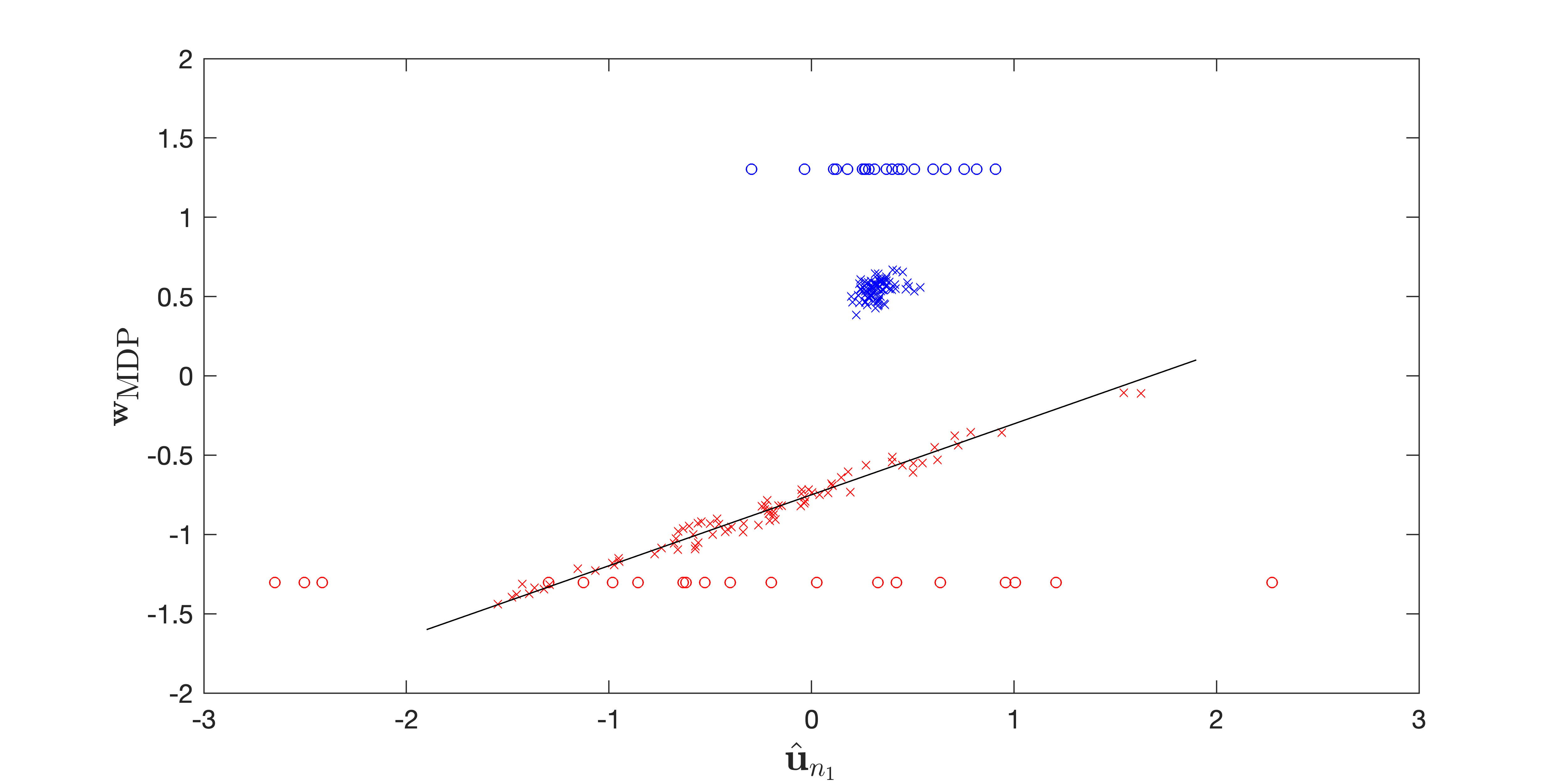} }}%
    \caption{$2$-dimensional projections of training data $\Xc$ (class 1: blue circles, class 2: red circles) and independent test data $\Yc$ (class 1: blue crosses, class 2: red crosses) onto (a) $\Sc_1 = {\rm span}(\hat\uv_1, \mdp)$ and (b) $\Sc_{n_1} = {\rm span}(\hat\uv_{n_1}, \mdp)$. 
    }%
    \label{fig:app:SDP}%
\end{figure}

Furthermore, we can show that the projected ridged linear discriminant vector $\vv_{\alpha}$ defined in Section~\ref{sec:estimation ridge} with a negative ridge parameter can also be a second maximal data piling direction for these cases. Theorem~\ref{thm:ridge unequal tail eigenvalues app} shows that $\vv_{-\tau_1^2}$ (or $\vv_{-\tau_2^2}$) yields second data piling if (i) $\beta_1 = 1, \beta_2 < 1$ (or (ii) $\beta_1 < 1, \beta_2 = 1$), respectively. The proof of this result closely follows the argument used in Theorem~\ref{thm:ridge equal tail eigenvalues piling distance}.

\begin{theorem}\label{thm:ridge unequal tail eigenvalues app}
Suppose Assumptions \ref{assume:1}---\ref{assume:5} hold and assume $\beta_k = 1$, $\beta_s < 1$ $(1 \leq k \ne s \leq 2)$. Then for $\hat\alpha_k$ chosen as an HDLSS-consistent estimator of $-\tau_k^2$, ${\rm Angle}{(\vv_{\hat\alpha_k}, \uv_{(k),i, \Sc})} \xrightarrow{P} \pi/2$ as $p \to \infty$ for all $1 \leq i \leq m_k$. Moreover, for any independent observation $Y \in \Yc$,
    \begin{align}
    \frac{1}{\sqrt{p}} \vv_{\hat\alpha_k}^\top(Y-\bar{X}) \xrightarrow{P} \begin{cases} \gamma_{A,k}(\eta_2(1-\cos^2\varphi)\delta^2 - (\tau_1^2 - \tau_2^2) / n), & \pi(Y) = 1, \\ 
    \gamma_{A,k}(-\eta_1(1-\cos^2\varphi)\delta^2 - (\tau_1^2 - \tau_2^2) / n), & \pi(Y) = 2 
    \end{cases}
    \end{align}
as $p\to\infty$ where $\gamma_{A,k}$ is a strictly positive random variable depending on the true principal component scores of $\Xc_k$.
\end{theorem}

Similarly to the case of $\beta_1 = \beta_2=1$ and $\tau_1 = \tau_2$ in Appendix~\ref{app:strong spikes with equal tails}, we can show that $\left\{\wv \right\} \in \Ac$ is asymptotically close to the subspace spanned by $\vv_{\hat\alpha_1}$ (or $\vv_{\hat\alpha_2}$) and $\left\{\hat\uv_i \right\}_{i \in \left\{1, \ldots, n-2 \right\} \setminus \Dc}$ for case (i) (or (ii)), respectively. This naturally implies that $\vv_{\hat\alpha_1}$ (or $\vv_{\hat\alpha_2}$, respectively) is a second maximal data piling direction. Furthermore, it can be shown that the bias-corrected projected ridge classification rule $\phi_{\textup{b-PRD}, \alpha}$ in (\ref{eq:bPRDrule}) achieves perfect classification only at the negative ridge parameter $\alpha = -\tau_k^2$ if $\beta_k = 1, \beta_s < 1$ ($1 \leq k \ne s \leq 2$).

As discussed in Appendix~\ref{app:strong spikes with unequal tails}, the use of $\vv_{\alpha}$ requires correctly identifying $\Dc$ and determining the signal subspace $\Sc$. This becomes challenging when $\beta_1 < 1, \beta_2 = 1$ and $\tau_1 > \tau_2$ as shown in Lemma~\ref{lem:asymptotic property of Sw with b1 < 1 and b2 = 1}. Similarly done in Appendix~\ref{app:strong spikes with unequal tails} for the case of strong spikes, we can set $\Dc = \{1, \ldots, m_2, n_1, \ldots, n_1+m_2-1 \}$. Even if $\vv_{\alpha}$ is affected by $m_2$ additional noisy directions, it can still yield second data piling when using $\alpha := \hat\alpha_2$, and $\phi_{\textup{b-PRD}, \hat\alpha_2}$ achieves asymptotic perfect classification.

\section{Proofs of Main Results}\label{app:proof of main results}
In this section, we give the proofs of main lemmas and theorems. Unless otherwise stated, we only give the proofs for the case of $\beta_1 = \beta_2 = 1$, $\tau_1 > \tau_2$ for brevity. The proofs for the other cases are quite similar to, but much simpler than, those for this case.

For any vector $\vv \in \Real^{l}$ ($l \in \mathbb{N}$), let $[\vv]_i$ denote the $i$th element of $\vv$. For any matrix $\Mv \in \Real^{l \times l'}$ ($l, l' \in \mathbb{N}$), let $[\Mv]_i$ and $[\Mv]^j$ denote the $i$th row and the $j$th column of $\Mv$, respectively. Also, let $[\Mv]_{i,j}$ denote the $(i,j)$-coordinate of $\Mv$. Let $\1v_{l} \in \Real^{l}$ (and $\0v_{l} \in \Real^{l}$) denote a vector whose all entries are $1$ (and $0$, respectively). Write an $(l \times l)$ identity matrix as $\Iv_{l}$, and an $(l \times l')$ matrix whose entries are all zero as $\Ov_{l \times l'}$. Also, for $k = 1, 2$, let $\Jv_{n_k}$ be the matrix of size $n_k \times n_k$ whose all entries are $1$ and $\Jv = \begin{pmatrix} \frac{1}{n_1}\Jv_{n_1} & \Ov_{n_1 \times n_2} \\ \Ov_{n_2 \times n_1} & \frac{1}{n_2}\Jv_{n_2} \end{pmatrix}$.

\subsection{Preliminary Lemmas}
Recall that the matrix of true principal component scores of $\Xv_k = [X_{k1}, \ldots, X_{kn_k}]$ is
$\Zv_{k} = \Lambdav_{(k)}^{-\frac{1}{2}}\Uv_{(k)}^\top(\Xv_k - \muv_{(k)} \1v_{n_k}^\top) = [\zv_{k,1}, \ldots, \zv_{k,p}]^\top \in \Real^{p \times n_k}$ where $\zv_{k,j}$ is a vector of $j$th principal component scores of the $k$th class. We write $\bar{\zv}_{k,i} = n_k^{-1}\zv_{k,i}^\top\1v_{n_k}$. Also, denote a vector of true principal component scores of independent observation $Y$ by $\zetav = (\zeta_1, \ldots, \zeta_p)^\top$. Then, the elements of $\Zv_{k}$ and $\zetav$ are uncorrelated, and have zero mean and unit variance. The following lemma follows directly from Lemma C.1 of \citet{Chang2021}. 

\begin{lemma}\label{app:lemma:prelim} Suppose Assumptions \ref{assume:1}---\ref{assume:5} hold. For $k = 1, 2$, the following hold as $p \to \infty$.
\begin{itemize}
    \item[(i)] $p^{-1}\muv^\top\Uv_{(k)}\Lambdav_{(k)}^{1/2} \zetav \xrightarrow{P} \begin{cases} \sum_{i=1}^{m_k} \sigma_{k,i} \cos\theta_{k,i}\delta \zeta_i, & \beta_k = 1, \\ 0,  & 0 \leq \beta_k < 1. \end{cases}$ 
    
    \item[(ii)] $p^{-1}\muv^\top\Uv_{(k)}\Lambdav_{(k)}^{1/2} \Zv_{k} \xrightarrow{P} \begin{cases} \sum_{i=1}^{m_k} \sigma_{k,i} \cos\theta_{k,i}\delta \zv_{k,i}^\top, & \beta_k = 1, \\ 0,  & 0 \leq \beta_k < 1. \end{cases}$ 
    
    \item[(iii)] $p^{-1}\Zv_{k}^\top\Lambdav_{(k)} \zetav \xrightarrow{P} \begin{cases} \sum_{i=1}^{m_k} \sigma_{k,i}^2 \zv_{k,i} \zeta_i, & \beta_k = 1, \\ 0,  & 0 \leq \beta_k < 1. \end{cases}$  
    
    \item[(iv)] $p^{-1}\Zv_{k}^\top \Lambdav_{(k)} \Zv_{k}  \xrightarrow{P} \begin{cases} \sum_{i=1}^{m_k} \sigma_{k,i}^2 \zv_{k,i} \zv_{k,i}^\top + \tau_k^2 \Iv_{n_k}, & \beta_k = 1, \\ \tau_k^2\Iv_{n_k},  & 0 \leq \beta_k < 1. \end{cases}$ 
\end{itemize}
\end{lemma}

We will investigate asymptotic properties of the sample within-scatter matrix $\Sv_W = (\Xv - \bar{\Xv})(\Xv - \bar{\Xv})^\top = \sum_{i=1}^{n-2}\hat\lambda_i \hat\uv_i \hat\uv_i^\top$ where $\Xv = [\Xv_1,~\Xv_2]$ and  $\bar{\Xv} = [\bar{\Xv}_1,~\bar{\Xv}_2]$. Since the dimension of $\Sv_W$ grows as $p \to \infty$, we instead use the $n \times n$ dual matrix, $\Sv_D = (\Xv - \bar{\Xv})^\top (\Xv - \bar{\Xv})$, which shares its nonzero eigenvalues with $\Sv_W$. We write the singular-value-decomposition of $\Xv - \bar{\Xv} = \hat\Uv_1 \Dv_1 \hat\Vv_1^\top = \sum_{i=1}^{n-2}d_i\hat\uv_i \hat\vv_i^\top$, where $\hat\uv_i$ is the $i$th eigenvector of $\Sv_W$, $d_i$ is the $i$th largest nonzero singular value, and $\hat\vv_i$ is the vector of normalized sample principal component scores. Write $\hat\vv_{i} = (\hat\vv_{i,1}^\top, \hat\vv_{i,2}^\top)^\top$ where $\hat\vv_{i,1} \in \Real^{n_1}$ and $\hat\vv_{i,2} \in \Real^{n_2}$. Then for $1 \leq i \leq n-2$, we can write
\begin{equation}\label{eq:app:svd}
    \begin{aligned}
        \hat\uv_i = d_i^{-1}(\Xv - \bar{\Xv})\hat\vv_i = \hat\lambda_i^{-1/2} \sum_{k=1}^{2} \Uv_{(k)}\Lambdav_{(k)}^{1/2}\Zv_{k}\left(\Iv_{n_k} - \frac{1}{n_k}\Jv_{n_k}\right)\hat\vv_{i,k}.
    \end{aligned}
\end{equation}
Recall that $\Wv_{k}=[\sigma_{k,1} \zv_{k,1}, \ldots, \sigma_{k,m_{k}} \zv_{k,m_{k}}]$ is  $n_k \times m_{k}$ matrix of the leading $m_{k}$ principal component scores of the $k$th class for each $k = 1, 2$.

\begin{lemma}\label{app:lemma:dual}
Suppose Assumptions \ref{assume:1}---\ref{assume:5} hold. Then,
\begin{align*}
    p^{-1} \Sv_D \xrightarrow{P} \Sv_0 := \begin{pmatrix} \Sv_{0,11} & \Sv_{0, 12} \\ \Sv_{0, 21} & \Sv_{0, 22} \end{pmatrix}
\end{align*}
as $p \to \infty$ where 
\begin{align*}
\Sv_{0, ii} = \begin{cases}
    (\Iv_{n_i} - \frac{1}{n_i}\Jv_{n_i})(\Wv_{i}\Wv_{i}^\top + \tau_i^2\Iv_{n_i})(\Iv_{n_i} - \frac{1}{n_i}\Jv_{n_i} ), & \beta_i = 1, \\
    \tau_i^2(\Iv_{n_i} - \frac{1}{n_i}\Jv_{n_i} ), & 0 \leq \beta_i < 1
    \end{cases}
\end{align*}
for $i = 1, 2$ and
\begin{align*}
\Sv_{0, ij} = \begin{cases}
    (\Iv_{n_i} - \frac{1}{n_i}\Jv_{n_i})(\Wv_{i}\Rv_{i}^\top \Rv_{j}\Wv_{j}^\top)(\Iv_{n_j} - \frac{1}{n_j}\Jv_{n_j} ), & \beta_1 = \beta_2 = 1, \\
    \Ov_{n_i \times n_j}, & o.w
    \end{cases}
\end{align*}    
for $1 \leq i \ne j \leq 2$.
\end{lemma}
\begin{proof}
Observe that $\Xv - \bar{\Xv} = \Xv(\Iv_n - \Jv) = [\Uv_{(1)}\Lambdav_{(1)}^{1/2}\Zv_{1},~\Uv_{(2)}\Lambdav_{(2)}^{1/2}\Zv_{2}](\Iv_{n} - \Jv)$ and
\begin{align*}
    \frac{\Sv_D}{p} &= (\Iv_{n} - \Jv)\begin{pmatrix}
    p^{-1}\Zv_{1}^\top \Lambdav_{(1)} \Zv_{1} & p^{-1}\Zv_{1}^\top \Lambdav_{(1)}^{1/2}\Uv_{(1)}^\top\Uv_{(2)}\Lambdav_{(2)}^{1/2}\Zv_{2} \\ 
    p^{-1}\Zv_{2}^\top \Lambdav_{(2)}^{1/2}\Uv_{(2)}^\top \Uv_{(1)}\Lambdav_{(1)}^{1/2}\Zv_{1} & p^{-1}\Zv_{2}^\top \Lambdav_{(2)}\Zv_{2}
    \end{pmatrix} (\Iv_{n} - \Jv). \\
\end{align*}
By Lemma~\ref{app:lemma:prelim} (iv), we have
\begin{align*}
    p^{-1}\Zv_{i}^\top \Lambdav_{(i)}\Zv_{i} \xrightarrow{P} 
    \begin{cases}
    \Wv_{i}\Wv_{i}^\top + \tau_i^2 \Iv_{n_i}, & \beta_i = 1, \\
    \tau_i^2 \Iv_{n_i} , & 0 \leq \beta_i < 1
    \end{cases}  
\end{align*}
as $p \to \infty$. Thus, it suffices to show that
\begin{align*}
    p^{-1}\Zv_{1}^\top \Lambdav_{(1)}^{1/2}\Uv_{(1)}^\top \Uv_{(2)}\Lambdav_{(2)}^{1/2}\Zv_{2} \xrightarrow{P} \begin{cases}
    \Wv_{1}\Rv_{1}^\top \Rv_{2}\Wv_{2}^\top, & \beta_1 = \beta_2 = 1,\\
    \Ov_{n_1 \times n_2}, & o.w
    \end{cases}
\end{align*}
as $p \to \infty$. Let $\Av_{ij} = \Zv_{1,i}^\top \Lambdav_{(1),i}^{1/2}\Uv_{(1),i}^\top\Uv_{(2),j}\Lambdav_{(2),j}^{1/2}\Zv_{2,j}$ for $i, j = 1, 2$ where $\Uv_{(k),1} = [\uv_{(k),1}, \ldots, \uv_{(k),m_k}]$, $\Uv_{(k),2} = [\uv_{(k),m_k+1}, \ldots, \uv_{(k),p}]$, $\Lambdav_{(k),1} = \diag{(\{\lambda_{(k),i}\}_{i=1}^{m_k})}$, $\Lambdav_{(k),2} = \diag{(\{\lambda_{(k),i}\}_{i=m_k+1}^{p})}$, $\Zv_{k,1} = [\zv_{k,1}, \ldots, \zv_{k,m_k}]^\top$ and $\Zv_{k,2} = [\zv_{k,m_k+1}, \ldots, \zv_{k,p}]^\top$. Then, we can decompose 
\begin{align*}
    \Zv_{1}^\top \Lambdav_{(1)}^{1/2}\Uv_{(1)}^\top \Uv_{(2)}\Lambdav_{(2)}^{1/2}\Zv_{2} = \sum_{i,j=1}^2 \Av_{ij}.
\end{align*}
By Assumption~\ref{assume:2}, there exists $M_k < \infty$ such that $\tau_{k,i} \leq M_k$ for all $i$. Thus,
\begin{align*}
    p^{-2}\mathbb{E}\fnorm{\Av_{12}}^2 &= p^{-2}\mathbb{E}\tr{(\Av_{12}^\top\Av_{12})} \\
&=p^{-2}n_1n_2\tr{(\Uv_{(2),2}^\top\Uv_{(1),1}\Lambdav_{(1),1}\Uv_{(1),1}^\top\Uv_{(2),2}\Lambdav_{(2),2} )} \\
    &=p^{-2}n_1n_2\sum_{l=1}^{m_1}\sum_{l' = m_2+1}^{p} (p^{\beta_1} \sigma_{1,l}^2+ \tau_{1,l}^2)\tau_{2,l'}^2(\uv_{(1),l}^\top \uv_{(2),l'})^2 \\
    &\leq p^{-2}n_1n_2m_1(p^{\beta_1} \sigma_{1,1}^2+M_1^2)M_2^2 \rightarrow 0
\end{align*}
as $p \to \infty$ where $\fnorm{~\cdot~}$ denotes the Frobenius norm of a matrix. Thus, $p^{-1}\Av_{12} \xrightarrow{P} \Ov_{n_1 \times n_2} $ as $p \to \infty$, and $p^{-1}\Av_{21}, p^{-1}\Av_{22} \xrightarrow{P}  \Ov_{n_1 \times n_2} $ can be shown in a similar manner. Then it suffices to show that $p^{-1}\Av_{11} \xrightarrow{P}  \Wv_{1}\Rv_{1}^\top \Rv_{2} \Wv_{2}^\top$ if $\beta_1 = \beta_2 = 1$ and $p^{-1}\Av_{11} \xrightarrow{P} \Ov_{n_1 \times n_2}$ otherwise
as $p \to \infty$. If $\beta_1 < 1$ or $\beta_2 < 1$, then
\begin{align*}
    &p^{-2}\mathbb{E}\fnorm{\Av_{11}}^2 =p^{-2}n_1n_2\sum_{l=1}^{m_1}\sum_{l' = 1}^{m_2} (p^{\beta_1}\sigma_{1,l}^2+\tau_{1,l}^2)(p^{\beta_2}\sigma_{2,l'}^2 +\tau_{2,l'}^2)(\uv_{(1),l}^\top \uv_{(2),l'})^2 \\
    &\leq p^{-2}n_1n_2 m_1m_2(p^{\beta_1}\sigma_{1,1}^2+M_1^2)(p^{\beta_2}\sigma_{2,1}^2+M_2^2)\rightarrow 0
\end{align*}
and $p^{-1} \Av_{11} \xrightarrow{P} \Ov_{n_1 \times n_2}$ as $p \to \infty$ since $\beta_1 + \beta_2 < 2$. If $\beta_1 = \beta_2 = 1$, then $
    p^{-1}\Av_{11} = p^{-1}\Zv_{1,1}^\top \Lambdav_{(1),1}^{1/2}\Rv_{1}^{(p) \top}\Rv_{2}^{(p)}\Lambdav_{(2),1}^{1/2}\Zv_{2,1} \xrightarrow{P} \Wv_{1} \Rv_{1}^\top \Rv_{2} \Wv_{2}^\top
$ 
as $p \to \infty$. 
\end{proof}

\subsection{Proof of Lemmas~\ref{lem:asymptotic property of Sw equal tails} and~\ref{lem:asymptotic property of Sw unequal tails}}
Lemmas~\ref{app:lemma:prelim} and~\ref{app:lemma:dual} play an important role in the proof of Lemmas~\ref{lem:asymptotic property of Sw equal tails} and~\ref{lem:asymptotic property of Sw unequal tails}. Recall that for any square matrix $\Mv$, $\phi_i(\Mv)$ and $v_{i}(\Mv)$ denote the $i$th largest eigenvalue of $\Mv$ and its corresponding eigenvector, respectively. Also, let $v_{ij}(\Mv)$ be the $j$th coefficient of $v_i(\Mv)$. For $\Sv_0$ defined in Lemma \ref{app:lemma:dual}, write $v_i(\Sv_0) = (\tilde{v}_{i1}(\Sv_0)^\top, \tilde{v}_{i2}(\Sv_0)^\top)^\top$ where $\tilde{v}_{i1}(\Sv_0) \in \Real^{n_1}$ and $\tilde{v}_{i2}(\Sv_0) \in \Real^{n_2}$. Also, for $\Phiv_{\tau_1,\tau_2}$ defined in (\ref{eq:perturbed covariance}), write $v_{i}(\Phiv_{\tau_1, \tau_2}) = (\tilde{v}_{i1}(\Phiv_{\tau_1, \tau_2})^\top, \tilde{v}_{i2}(\Phiv_{\tau_1, \tau_2})^\top)^\top$ where $\tilde{v}_{i1}(\Phiv_{\tau_1, \tau_2}) \in \Real^{m_1}$ and $\tilde{v}_{i2}(\Phiv_{\tau_1, \tau_2}) \in \Real^{m_2}$.

\begin{proof}
The proof of Lemma~\ref{lem:asymptotic property of Sw equal tails} closely follows along the lines of Lemmas C.2 and C.3 of \citet{Chang2021} using Lemmas~\ref{app:lemma:prelim} and~\ref{app:lemma:dual}. To prove Lemma~\ref{lem:asymptotic property of Sw unequal tails}, assume $\beta_1 = \beta_2 = 1$ and $\tau_1 > \tau_2$. By Lemma~\ref{app:lemma:dual}, as $p \to \infty$,
\begin{align*}
    \frac{1}{p}\Sv_D \xrightarrow{P} \Sv_0 =  (\Iv_{n} - \Jv) \begin{pmatrix} \Wv_{1}\Wv_{1}^\top + \tau_1^2\Iv_{n_1} & \Wv_{1}\Rv_{1}^\top \Rv_{2} \Wv_{2}^\top \\  \Wv_{2}\Rv_{2}^\top \Rv_{1}\Wv_{1}^\top  & \Wv_{2}\Wv_{2}^\top + \tau_2^2 \Iv_{n_2} \end{pmatrix}  (\Iv_{n} - \Jv).
\end{align*}
To show (i), recall that $\Sv_W$ shares its nonzero eigenvalues with $\Sv_D$, and since $\phi_i$ is a continuous function of elements of a symmetric matrix, we have $p^{-1}\phi_i(\Sv_W) \xrightarrow{P} \phi_i(\Sv_0)$ as $p \to \infty$ for $1 \leq i \leq n-2$. First, let $\hv_0 = (\hv_{01}^\top, \0v_{n_2}^\top)^\top \in \Real^{n}$ be an unit vector satisfying 
\begin{align}\label{eq:g01}
    \Wv_{1}^\top \left(\Iv_{n_1} - \frac{1}{n_1}\Jv_{n_1}\right)\hv_{01} = \0v_{m_1}
\end{align}
and $\1v_{n_1}^\top \hv_{01} = 0$. Then, $\Sv_0 \hv_0 = \tau_1^2 \hv_0$. It implies that $\Sv_0$ has an eigenvalue $\tau_1^2$ of multiplicity $(n_1 - m_1 - 1)$. Likewise, we can show that $\Sv_0$ has an eigenvalue $\tau_2^2$ of multiplicity $(n_2-m_2-1)$. Lastly, let $\hv_i = (\hv_{i1}^\top, \hv_{i2}^\top)^\top \in \Real^{n}$ $(1 \leq i \leq m_1+m_2)$ be an unit vector with 
\begin{align}\label{eq:gik}
    \hv_{ik} = \left(\Iv_{n_k} - \frac{1}{n_k}\Jv_{n_k}\right)\Wv_{k}\Phiv_{k}^{-1/2} \tilde{v}_{ik}(\Phiv_{\tau_1, \tau_2})
\end{align}
for $k = 1, 2$. Then, $\Sv_0 \hv_i = \phi_i(\Phiv_{\tau_1, \tau_2})\hv_i$ for all $1 \leq i \leq m_1+m_2$. Thus, $\Sv_0$ has eigenvalues $\phi_i(\Phiv_{\tau_1, \tau_2})$ for $1 \leq i \leq m_1+m_2$. In summary, $\Sv_0$ has $(n-2)$ non-negative eigenvalues with $\tau_1^2$ of multiplicity $(n_1-m_1-1)$, $\tau_2^2$ of multiplicity $(n_2-m_2-1)$ and $\phi_i(\Phiv_{\tau_1, \tau_2})$ for $1 \leq i \leq m_1+m_2$. Moreover, note that $\Phiv_{1}$ is a positive definite matrix and 
\begin{align*}
    \phi_{m_1}(\Phiv_{\tau_1, \tau_2}) &= \max_{\Vc \subseteq \Real^{m_1+m_2}, \dim\Vc = m_1} \min_{\xv \in \Vc, \| \xv\| = 1} \xv^\top \Phiv_{\tau_1, \tau_2} \xv \\
    &\ge \min_{\xv \in \tilde{\Vc}} \xv^\top \Phiv_{\tau_1, \tau_2} \xv = \phi_{m_1}(\Phiv_{1}) + \tau_1^2 > \tau_1^2
\end{align*}
where 
\begin{align*}
    \tilde{\Vc} = {\rm span}\left(\begin{pmatrix} v_1(\Phiv_1) \\ \0v_{m_2} \end{pmatrix}, \cdots, \begin{pmatrix} v_{m_1}(\Phiv_1) \\ \0v_{m_2} \end{pmatrix} \right).
\end{align*}
Similarly, it can be shown that $\phi_{m_1+m_2}(\Phiv_{\tau_1, \tau_2}) > \tau_2^2$. Hence, $m_1\leq k_0 \leq m_1+m_2$ where $k_0$ is an integer which satisfies $\phi_{k_0}(\Phiv_{\tau_1, \tau_2}) > \tau_1^2 \ge \phi_{k_0+1}(\Phiv_{\tau_1, \tau_2})$ where we use the convention of $\phi_{m_1+m_2+1}(\Phiv_{\tau_1, \tau_2}) = 0$. In particular, if $m = m_1$, then
\begin{align*}
    &\phi_{m_1+1}(\Phiv_{\tau_1, \tau_2}) = \phi_{m+1}(\Phiv_{\tau_1, \tau_2}) \\
    &\leq \phi_{m+1}\begin{pmatrix} \Phiv_{1} & \Phiv_{1}^{1/2}\Rv_{(1)}^\top \Rv_{(2)} \Phiv_{2}^{1/2} \\ \Phiv_{2}^{1/2}\Rv_{(2)}^\top \Rv_{(1)}\Phiv_{1}^{1/2} & \Phiv_{2} \end{pmatrix} + \phi_1\begin{pmatrix} \tau_1^2\Iv_{m_1} & \Ov_{m_1 \times m_2} \\ \Ov_{m_2 \times m_1} & \tau_2^2 \Iv_{m_2} \end{pmatrix} \\
    &= 0+\tau_1^2 = \tau_1^2
\end{align*}
by Weyl's inequality. Since $\phi_{m}(\Phiv_{\tau_1, \tau_2}) = \phi_{m_1}(\Phiv_{\tau_1, \tau_2}) > \tau_1^2$, $k_0 = m = m_1$ in this case.

To show (ii), we will show that 
\begin{equation}\label{eq:uij}
    \begin{aligned}
        \hat\uv_i^\top \uv_j \xrightarrow{P} \begin{cases}
        D_{i,j},  & 1 \leq i \leq k_0, \\
        0, & k_0+1 \leq i \leq k_0 + (n_1-m_1-1), \\
        D_{i-(n_1-m_1-1),j}, & k_0 + (n_1 - m_1) \leq i \leq n_1+m_2-1, \\
        0, & n_1+m_2 \leq i \leq n-2
    \end{cases}
    \end{aligned}
\end{equation}
where $k_0$ $(m_1 \leq k_0 \leq m_1+m_2)$ is defined in (i) and
\begin{align}\label{eq:Dij}
D_{i,j} := \frac{1}{\sqrt{\phi_i(\Phiv_{\tau_1, \tau_2})}} \sum_{k=1}^{2} [\Rv_{k}]_j\Phiv_k^{1/2} \tilde{v}_{ik}(\Phiv_{\tau_1, \tau_2})
\end{align}
for $1 \leq i \leq m_1+m_2$ and $1 \leq j \leq m$. 
Using~(\ref{eq:app:svd}), we can write
\begin{align*}
    \uv_j^\top\hat\uv_i = \left( \frac{\hat\lambda_i}{p} \right)^{-1/2}\sum_{k=1}^{2} \frac{1}{\sqrt{p}}\uv_j^\top \Uv_{(k)}\Lambdav_{(k)}^{1/2}\Zv_{k}\left(\Iv_{n_k} - \frac{1}{n_k}\Jv_{n_k}\right)\hat\vv_{i,k}.
\end{align*}

First, note that $p^{-1/2} \uv_{j}^\top \Uv_{(k)}\Lambdav_{(k)}^{1/2} \Zv_{k}$ can be decomposed into two terms:
\begin{align*}
    \frac{1}{\sqrt{p}}\uv_{j}^\top \Uv_{(k)} \Lambdav_{(k)}^{1/2} \Zv_{k} = \sum_{i=1}^{m_k} \frac{1}{\sqrt{p}}\uv_{j}^\top \uv_{(k),i}\lambda_{(k),i}^{1/2}\zv_{k,i}^\top + \sum_{i=m_k+1}^{p} \frac{1}{\sqrt{p}}\uv_{j}^\top \uv_{(k),i}\lambda_{(k),i}^{1/2}\zv_{k,i}^\top
\end{align*}
for $1 \leq j \leq m$. The first term converges to $[\Rv_{k}]_j\Wv_{k}^\top$ in probability as $p \to \infty$. The second term converges to zero in probability since for any $\epsilon > 0$,
\begin{equation}\label{eq:app:evec:b}
    \begin{aligned}
    &\mathbb{P}\left( \op{\sum_{i=m_k+1}^{p} \frac{1}{\sqrt{p}}\uv_{j}^\top \uv_{(k),i}\lambda_{(k),i}^{1/2}\zv_{k,i} } > \epsilon \right) \leq \frac{1}{p\epsilon^2}\mathbb{E}\left( \op {\sum_{i=m_k+1}^{p} \uv_{j}^\top \uv_{(k),i} \tau_{k,i}  \zv_{k,i}}^2  \right) \\ 
    &= \frac{1}{p\epsilon^2} \mathbb{E} \left( \sum_{i=m_k+1}^{p} (\uv_{j}^\top \uv_{(k),i})^2 \tau_{k,i}^2  \zv_{k,i}^\top \zv_{k,i} \right) \leq \frac{n_k M_k^2}{p\epsilon^2} \rightarrow 0 
    \end{aligned}
\end{equation}
as $p \to \infty$ where, by Assumption~\ref{assume:2}, there exists $M_k < \infty$ such that $\tau_{k,i} \leq M_k$ for $i = m_k+1, \ldots, p$. Then we have 
\begin{align*}
    \frac{1}{\sqrt{p}}\uv_{j}^\top \Uv_{(k)}\Lambdav_{(k)}^{1/2}\Zv_{k} \xrightarrow{P} [\Rv_{k}]_j\Wv_{k}^\top
\end{align*}
as $p \to \infty$ for $k = 1, 2$. 

Next, note that for $i \in \Dc = \{1, \ldots, k_0, k_0 + (n_1-m_1), \ldots, n_1+m_2-1 \}$, $\hat\vv_i = (\hat\vv_{i,1}, \hat\vv_{i,2})^\top \xrightarrow{P} v_i(\Sv_0)$ as $p \to \infty$. Hence, for $i \in \Dc$ and $1 \leq j \leq m$, 
\begin{align*}
    \uv_{j}^\top \hat\uv_i \xrightarrow{P} \phi_i(\Sv_0)^{-1/2}\ev_j^\top \Wv^\top (\Iv_{n} - \Jv) v_i(\Sv_0)
\end{align*}
as $p \to \infty$ where $\ev_{j} \in \Real^{m}$ is a vector whose $j$th coordinate is $1$ and the other coordinates are zero. Using (\ref{eq:gik}), we have $\uv_j^\top \hat\uv_i \xrightarrow{P} D_{i,j}$ for $1\leq i\leq k_0$ and $\uv_j^\top \hat\uv_i \xrightarrow{P} D_{i-(n_1-m_1-1),j}$ for $k_0 + (n_1-m_1) \leq i \leq n_1+m_2-1$ as $p \to \infty$ where $D_{i,j}$ is defined in (\ref{eq:Dij}). In contrast, for $i \in \{1, \ldots, n-2 \} \setminus \Dc$, (\ref{eq:g01}) gives $\Wv_{1}^\top (\Iv_{n_1} - n_1^{-1}\Jv_{n_1})\hat\vv_{i,1} = o_p(1)$ and $\Wv_{2}^\top (\Iv_{n_2} - n_2^{-1}\Jv_{n_2})\hat\vv_{i,2} = o_p(1)$, which imply that $\Wv^\top (\Iv_n - \Jv)\hat\vv_i = o_p(1)$ and $\uv_j^\top \hat\uv_i \xrightarrow{P} 0$ as $p \to \infty$. 

We then conclude the proof by noting that the limit of 
\begin{align*}
    \cos{\left({\rm Angle}{(\hat\uv_i, \Uc)}\right)} =\sqrt{\sum_{j=1}^{m} (\hat\uv_i^\top \uv_{j})^2 }
\end{align*}
can be obtained by (\ref{eq:uij}).
\end{proof}

Note that Lemmas~\ref{lem:asymptotic property of Sw equal tails} and~\ref{lem:asymptotic property of Sw unequal tails} can also be used to investigate the asymptotic behavior of $\hat\uv_i^\top \dv$, where $\dv = \bar{X}_1 - \bar{X}_2$ is the sample mean difference vector. Lemma~\ref{app:lemma:md} will be used frequently in the proof of main theorems.

\begin{lemma}\label{app:lemma:md}
Suppose Assumptions \ref{assume:1}---\ref{assume:5} hold. Then conditional to $\Wv_{1}$ and $\Wv_{2}$, the following hold as $p \to \infty$.
\begin{itemize}
    \item[(i)] If $\beta_1 = \beta_2 = 1$ and $\tau_1 = \tau_2 =: \tau$, then
    \begin{align*}
    p^{-1/2} \dv^\top \hat\uv_i \xrightarrow{P} \begin{cases} 
    \sum_{j=1}^{m} r_{j} C_{i,j}, & 1 \leq i \leq m, \\ 
    0, & m + 1 \leq i \leq n -2 \end{cases}
    \end{align*}
    for $1 \leq j \leq m$ where $r_{j} := \cos\theta_{j}\delta + \sum_{k=1}^{m_1} [\Rv_{1}]_{jk}\sigma_{1,k}\bar{\zv}_{1,k} - \sum_{k=1}^{m_2}[\Rv_{2}]_{jk}\sigma_{2,k}\bar{\zv}_{2,k}$ and $C_{i,j} := \sqrt{\phi_i(\Phiv) / (\phi_i(\Phiv)+\tau^2)}v_{ij}(\Phiv)$.
    
    \item[(ii)] If $\beta_1 = \beta_2 = 1$ and $\tau_1 > \tau_2$, then
    \begin{align*}
    p^{-1/2} \dv^\top \hat\uv_i \xrightarrow{P} \begin{cases} 
    \sum_{j=1}^{m} r_{j} D_{i,j}, & 1 \leq i \leq k_0,  \\ 
    0, & k_0+1 \leq i \leq k_0 + (n_1-m_1-1), \\
    \sum_{j=1}^{m} r_{j} D_{i-(n_1-m_1-1),j}, & k_0 + (n_1-m_1) \leq i \leq n_1+m_2-1, \\
    0, & n_1+m_2 \leq i \leq n-2
    \end{cases}
    \end{align*}
    for $1 \leq j \leq m$ where $k_0$ is defined in Lemma~\ref{lem:asymptotic property of Sw unequal tails} (i), $r_j$ is defined in Lemma~\ref{app:lemma:md} (i) and $D_{i,j}$ is defined in (\ref{eq:Dij}).
\end{itemize}

\end{lemma}

\begin{proof}
Observe that $\dv = \bar{X}_1 - \bar{X}_2 = \muv + \frac{1}{n_1}\Uv_{(1)}\Lambdav_{(1)}^{1/2}\Zv_{1} \1v_{n_1} - \frac{1}{n_2}\Uv_{(2)}\Lambdav_{(2)}^{1/2}\Zv_{2} \1v_{n_2}$ and
\begin{align*}
    \frac{1}{\sqrt{p}}\dv^\top \hat\uv_i = \frac{1}{\sqrt{p}} \muv^\top \hat\uv_i + \frac{1}{n_1 \sqrt{p}}\1v_{n_1}^\top \Zv_{1}^\top \Lambdav_{(1)}^{1/2}\Uv_{(1)}^\top \hat\uv_i - \frac{1}{n_2 \sqrt{p}} \1v_{n_2}^\top \Zv_{2}^\top \Lambdav_{(2)}^{1/2}\Uv_{(2)}\hat\uv_i.
\end{align*}
Using (\ref{eq:app:svd}), we can write
\begin{align*}
    \frac{1}{\sqrt{p}}\muv^\top \hat\uv_i = \left(\frac{\hat\lambda_i}{p}\right)^{-1/2} \left( \frac{\muv^\top \Uv_{(1)}\Lambdav_{(1)}^{1/2}\Zv_{1}}{p},~\frac{\muv^\top \Uv_{(2)}\Lambdav_{(2)}^{1/2}\Zv_{2}}{p}\right) (\Iv_{n} - \Jv)\hat\vv_{i}.
\end{align*}
Write $\cv = (\cos\theta_1, \ldots, \cos\theta_m)^\top \in \Real^{m}$ and $\cv_k = (\cos\theta_{k,1}, \ldots, \cos\theta_{k, m_k})^\top \in \Real^{m_k}$. Then we have $p^{-1}\muv^\top \Uv_{(k)}\Lambdav_{(k)}^{1/2}\Zv_{k} \xrightarrow{P} \cv_k^\top \Wv_{k}^\top \delta$ as $p \to \infty$ from Lemma~\ref{app:lemma:prelim} (ii). Thus,
\begin{equation}\label{eq:app:lemma:md:a}
    \begin{aligned}
        \frac{1}{\sqrt{p}}\muv^\top \hat\uv_i \xrightarrow{P} \begin{cases}
        \phi_i(\Sv_0)^{-1/2} \delta \cv^\top \Wv^\top (\Iv_{n} - \Jv) v_i(\Sv_0), & i \in \Dc, \\
        0, & o.w
        \end{cases}
    \end{aligned}
\end{equation}
as $p \to \infty$. (recall the definition of $\Dc$ in Table~\ref{table:D}). Also, combining Lemma~\ref{app:lemma:prelim} (iv) and Lemma~\ref{app:lemma:dual} gives
\begin{equation}\label{eq:app:lemma:md:b}
    \begin{aligned}
    \frac{1}{n_k\sqrt{p}}\1v_{n_k}^\top \Zv_{k}^\top \Lambdav_{(k)}^{1/2}\Uv_{(k)}^\top \hat\uv_i \xrightarrow{P} \begin{cases}
        \phi_i(\Sv_0)^{-1/2}n_k^{-1}\1v_{n_k}^\top \Wv_{k}\Rv_{k}^\top \Wv^\top (\Iv_{n} - \Jv) v_i(\Sv_0), & i \in \Dc, \\
        0, & o.w
    \end{cases}
    \end{aligned}
\end{equation}
as $p \to \infty$ for $k = 1, 2$. Hence, combining (\ref{eq:app:lemma:md:a}) and (\ref{eq:app:lemma:md:b}) gives
\begin{align*}
    &\frac{1}{\sqrt{p}}\dv^\top\hat\uv_i \xrightarrow{P}
    \phi_i(\Sv_0)^{-1/2} (\delta\cv^\top + n_1^{-1}\1v_{n_1}^\top \Wv_{1}\Rv_{1}^\top - n_2^{-1}\1v_{n_2}^\top \Wv_{2}\Rv_{2}^\top) \Wv^\top (\Iv_{n} - \Jv)v_i(\Sv_0)
\end{align*}
for $i \in \Dc$ and $p^{-1/2}\dv^\top \hat\uv_i \xrightarrow{P} 0$ otherwise as $p \to \infty$.
\end{proof}

\subsection{Proof of Theorem~\ref{thm:test data piling}}
\begin{proof}
We provide the proof for the case of $\beta_1 = \beta_2 =1$ and $\tau_1 > \tau_2$. The other cases can be shown in a similar manner. For $Y \in \Yc$, assume $\pi(Y) = 1$. Recall that in this case, 
\begin{align*}
    \Dc = \left\{i : 1 \leq i \leq k_0, k_0+(n_1-m_1) \leq i \leq n_1 + m_2-1 \right\}
\end{align*}
where $k_0$ is defined in Lemma~\ref{lem:asymptotic property of Sw unequal tails} (i) and $\Sc = {\rm span}(\left\{\hat\uv_i\right\}_{i \in \Dc}, \mdp)$. For notational simplicity, we write $\Dc = \left\{i_1, \ldots, i_{m_1+m_2} \right\}$ so that $i_l < i_{l'}$ if $l < l'$.  Let $\tv^0 = (t_1, \ldots, t_m)^\top \in \Real^{m}$ with 
\begin{align}\label{eq:tj}
    t_j = \eta_2 \cos\theta_j \delta + \sum_{k=1}^{m_1} [\Rv_{1}]_{jk}\sigma_{1,k}(\zeta_k - \eta_1 \bar{\zv}_{1,k}) - \eta_2 \sum_{k=1}^{m_2} [\Rv_{2}]_{jk}\sigma_{2,k}\bar{\zv}_{2,k}
\end{align}
for $1 \leq j \leq m$ and $\nuv^0 = \Uv_{1,\Sc} \tv^{0} + \nu_1 \mdp + \bar{X}_{\Sc}$. Note that $\nuv^0 \in \Lc_1$. We claim that $\|Y_{\Sc} - \nuv^0 \|  \xrightarrow{P} 0$ as $p \to \infty$. For this, we need to show that (a) $\hat\uv_i^\top(Y_\Sc - \nuv^0) \xrightarrow{P} 0$ for $1 \leq i \leq m$ and (b) $\mdp^\top (Y_\Sc - \nuv^0) \xrightarrow{P} 0$ as $p \to \infty$. 

First, we show that (a) $\hat\uv_i^\top(Y_\Sc - \nuv^0) = p^{-1/2} \hat\uv_i^\top (Y - \bar{X}) - \hat\uv_i^\top \Uv_{1,\Sc}\tv^0 \xrightarrow{P} 0$ for $1 \leq i \leq m$ as $p \to \infty$. Note that
\begin{equation}\label{eq:thm:L1}
    \begin{aligned}
    &\frac{1}{\sqrt{p}}\hat\uv_i^\top (Y - \bar{X}) =\frac{1}{\sqrt{p}}\hat\uv_i^\top \left\{\eta_2 \muv + \Uv_{(1)}\Lambdav_{(1)}^{1/2}\left(\zetav - \frac{1}{n}\Zv_{1}\1v_{n_1}\right) - \frac{1}{n} \Uv_{(2)}\Lambdav_{(2)}^{1/2}\Zv_{2}\1v_{n_2}\right\} \\
    &= \frac{\eta_2}{\sqrt{p}}\hat\uv_i^\top \muv + \sum_{k=1}^{m_1}\hat\uv_i^\top \uv_{(1),k}\sigma_{1,k}(\zeta_k - \eta_1 \bar{\zv}_{1,k}) - \eta_2 \sum_{k=1}^{m_2} \hat\uv_i^\top \uv_{(2),k}\sigma_{2,k}\bar{\zv}_{2,k} + o_p(1).
    \end{aligned}
\end{equation}
Using (\ref{eq:uij}) and (\ref{eq:app:lemma:md:a}), we have
\begin{equation}\label{eq:thm:L1:a+b:ne}
    \begin{aligned}
        \frac{1}{\sqrt{p}}\hat\uv_{i_l}^\top (Y - \bar{X}) \xrightarrow{P} \sum_{j=1}^{m} t_j D_{l,j}
    \end{aligned}
\end{equation}
for $1 \leq l \leq m_1+m_2$ and $1 \leq j \leq m$ where $D_{l,j}$ is defined in (\ref{eq:Dij}), and 
\begin{align*}
    \hat\uv_{i_l}^\top \Uv_{1, \Sc}\tv^0 \xrightarrow{P}\sum_{j=1}^{m} t_j D_{l,j}
\end{align*}
as $p \to \infty$. Hence, 
\begin{align*}
    \hat\uv_{i_l}^\top(Y_{\Sc} - \nuv^0) = p^{-1/2}\hat\uv_{i_l}^\top (Y - \bar{X}) - \hat\uv_{i_l}^\top \Uv_{1,\Sc}\tv^0 \xrightarrow{P} 0
\end{align*}
as $p \to \infty$ for $1 \leq l \leq m_1+m_2$.

Next, we show that (b) $\mdp^\top (Y_\Sc - \nuv^0) = p^{-1/2} \mdp^\top (Y - \bar{X}) - \mdp^\top \Uv_{1,\Sc}\tv^0 - \nu_1 \xrightarrow{P} 0$ as $p \to \infty$. We decompose $p^{-1/2}\mdp^\top (Y - \bar{X})$ into the two terms:
\begin{align*}
    \frac{1}{\sqrt{p}}\mdp^\top (Y - \bar{X}) &= \frac{\sqrt{p}}{\|\hat\Uv_2 \hat\Uv_2^\top \dv \|} \left(\frac{\dv^\top (Y - \bar{X})}{p} - \frac{(\hat\Uv_1 \hat\Uv_1^\top \dv)^\top(Y - \bar{X})}{p} \right) \\
    &=\frac{1}{\kappa_{\textup{MDP}}}(K_1 - K_2)
\end{align*}
where $K_1 = \dv^\top(Y - \bar{X}) / p$ and $K_2 = (\hat\Uv_1\hat\Uv_1^\top \dv)^\top(Y - \bar{X})/p$. By Lemmas~\ref{app:lemma:prelim} and~\ref{app:lemma:dual}, 
\begin{equation}\label{eq:K1limit}
\begin{aligned}
    K_1 &= \frac{1}{p}\left(\muv + \frac{1}{n_1}\Uv_{(1)}\Lambdav_{(1)}^{1/2}\Zv_{1}\1v_{n_1} - \frac{1}{n_2}\Uv_{(2)}\Lambdav_{(2)}^{1/2}\Zv_{2}\1v_{n_2} \right)^\top \\
    &\quad\left\{\eta_2 \muv + \Uv_{(1)}\Lambdav_{(1)}^{1/2}\left(\zeta - \frac{1}{n} \Zv_{1}\1v_{n_1}\right) - \frac{1}{n} \Uv_{(2)}\Lambdav_{(2)}^{1/2}\Zv_{2}\1v_{n_2} \right\} \\
    &\xrightarrow{P} \eta_2(1- \cos^2\varphi)\delta^2 - \frac{\tau_1^2 - \tau_2^2}{n} + \sum_{j=1}^{m} t_jr_{j} 
\end{aligned}
\end{equation}
as $p \to \infty$ where $r_{j}$ is defined in Lemma~\ref{app:lemma:md}. Using Lemma~\ref{app:lemma:md}, we also obtain
\begin{equation}\label{eq:K2limit}
    \begin{aligned}
    K_2 &= \frac{(\hat\Uv_1 \hat\Uv_1^\top \dv)^\top (Y - \bar{X})}{p} = \sum_{i \in \Dc} \left(\frac{1}{\sqrt{p}}\hat\uv_i^\top \dv \right) \left(\frac{1}{\sqrt{p}} \hat\uv_i^\top (Y - \bar{X})\right) + o_p(1) \\
    &\xrightarrow{P}\sum_{l=1}^{m_1+m_2}  \sum_{j=1}^{m}\sum_{j'=1}^{m}  t_j r_{j'} D_{l,j}D_{l,j'}.
    \end{aligned}
\end{equation}
as $p \to \infty$. Note that the limit of $\kappa_{\textup{MDP}}^2$ can be obtained from the limit of $p^{-1}\|\dv \|^2$ and $p^{-1}\|\hat\Uv_1\hat\Uv_1^\top \dv  \|^2$:
\begin{equation}\label{eq:kMDPlimit}
    \begin{aligned}
    &\kappa_{\textup{MDP}}^2 = \frac{1}{p}\|\dv\|^2 - \frac{1}{p}\|\hat\Uv_1 \hat\Uv_1^\top \dv \|^2 \\
    &\xrightarrow{P} (1-\cos^2\varphi)\delta^2 + \frac{\tau_1^2}{n_1} + \frac{\tau_2^2}{n_2} + \sum_{j=1}^{m}\sum_{j'=1}^{m} r_{j}r_{j'}\left(I_{jj'} - \sum_{l=1}^{m_1+m_2} D_{l,j}D_{l,j'}\right) \\
    &= (1-\cos^2\varphi)\delta^2 + \frac{\tau_1^2}{n_1} + \frac{\tau_2^2}{n_2} + \rv^\top \left(\Iv_{m} - \begin{pmatrix} \Rv_{1}\Phiv_{1}^{1/2}~\Rv_{2}\Phiv_{2}^{1/2}\end{pmatrix} \Phiv_{\tau_1,\tau_2}^{-1}  \begin{pmatrix}\Phiv_{1}^{1/2} \Rv_{1}^\top \\ \Phiv_{2}^{1/2}\Rv_{2}^\top \end{pmatrix} \right)\rv \\
    &=: \kappa^2
    \end{aligned}
\end{equation}
as $p \to \infty$ where $\rv = (r_{1}, \ldots, r_{m})^\top$ and $I_{jj'} = I(j = j')$. Note that $\kappa^2 \ge (1-\cos^2\varphi)\delta^2 + \tau_1^2 / n_1 + \tau_2^2 / n_2 > 0$. Combining (\ref{eq:K1limit}), (\ref{eq:K2limit}) and (\ref{eq:kMDPlimit}) gives
\begin{equation}\label{eq:L1:MDP:a:2}
    \begin{aligned}
        &\frac{1}{\sqrt{p}}\mdp^\top (Y - \bar{X}) =\frac{1}{\kappa_{\textup{MDP}}}(K_1 - K_2) \\
        &\xrightarrow{P} \frac{1}{\kappa} \left\{\eta_2(1-\cos^2\varphi)\delta^2 - \frac{\tau_1^2 - \tau_2^2}{n} + \sum_{j=1}^m \sum_{j'=1}^{m} t_j r_{j'} \left(I_{jj'} - \sum_{l=1}^{m_1+m_2} D_{l,j}D_{l,j'}  \right) \right\}
    \end{aligned}
\end{equation}
as $p \to \infty$. Also, Lemma~\ref{app:lemma:md} and (\ref{eq:uij}) lead to
\begin{equation}\label{eq:L1:MDP:b:2}
    \begin{aligned}
        \mdp^\top\Uv_{1, \Sc}\tv^0 &= \sum_{j=1}^{m} t_j \mdp^\top \uv_j = \frac{1}{\kappa_{\textup{MDP}}}\sum_{j=1}^{m} t_j \left\{\frac{1}{\sqrt{p}}\uv_j^\top \dv - \uv_j^\top \hat\Uv_1 \left( \frac{1}{\sqrt{p}}\hat\Uv_1^\top \dv \right) \right\} \\
        &\xrightarrow{P} \frac{1}{\kappa}\sum_{j=1}^m \sum_{j'=1}^{m} t_j r_{j'} \left(I_{jj'} - \sum_{l=1}^{m_1+m_2} D_{l,j}D_{l,j'}\right)
    \end{aligned}
\end{equation}
as $p \to \infty$. From~(\ref{eq:L1:MDP:a:2}) and (\ref{eq:L1:MDP:b:2}), we have $\mdp^\top(Y_{\Sc} - \nuv^0) \xrightarrow{P} 0$ as $p \to \infty$. Hence, from (a) and (b), we have $\|Y_\Sc - \nuv^0 \| \xrightarrow{P} 0$ as $p \to \infty$ for $Y \in \Yc$ with $\pi(Y) = 1$. Using similar arguments, we can show for $Y \in \Yc$ with $\pi(Y) = 2$.
\end{proof}

\begin{subsection}{Proof of Theorem~\ref{thm:SMDP piling distance}}
\begin{proof}
We first provide the proof for the case of $\beta_1 = \beta_2 =1$ and $\tau_1 > \tau_2$. As in the proof of Theorem~\ref{thm:test data piling}, for notational simplicity, we write $\Dc = \left\{i_1, \ldots, i_{m_1+m_2} \right\}$ so that $i_l < i_{l'}$ if $l < l'$. Write
\begin{align*}
    \hat{\Vv}_1 = [\hat\uv_{i_1}, \ldots, \hat\uv_{i_{m_1+m_2}}]
\end{align*}
and 
\begin{align*}
    \tilde{\Vv} = [\hat{\Vv}_1,~\mdp]
\end{align*}
so that the columns of $\tilde{\Vv}$ form an orthonormal basis of $\Sc$. Also, write 
\begin{align*}
    \Kv = \hat\Vv_1^\top \Uv_1 \in \Real^{(m_1+m_2) \times m}
\end{align*}
and 
\begin{align*}
    \kv^\top = \mdp^\top \Uv_1 \in \Real^{1 \times m}.
\end{align*}
From (\ref{eq:uij}), we get $\Kv \xrightarrow{P} \tilde{\Omegav}_{1}$ as $p \to \infty$, which is a $(m_1+m_2) \times m$ matrix such that
\begin{align}\label{eq:tildeOmega1}
    [\tilde{\Omegav}_1]_{i,j} = D_{i,j}
\end{align}
for $1 \leq i \leq m_1+m_2$ and $1 \leq j \leq m$ where $D_{i,j}$ is defined in (\ref{eq:Dij}). Also, from (\ref{eq:L1:MDP:b:2}), we get 
\begin{align}\label{eq:tildeOmega12}
    \kv \xrightarrow{P} \omegav_1 = \frac{1}{\kappa}(\Iv_m - \tilde{\Omegav}_1^\top \tilde{\Omegav}_1)\rv \in \Real^m
\end{align}
where $\rv = (r_1, \ldots, r_m)^\top$ with $r_j$ defined in Lemma~\ref{app:lemma:md}. Lastly, write 
\begin{align}\label{eq:tildeOmega}
    \tilde{\Omegav} = [\tilde{\Omegav}_1^\top,~\omegav_1]^\top.
\end{align}
Note that $\tilde{\Omegav}_1$ and $\tilde{\Omegav}$ are of rank $m$.

(i) For any given $\{\wv\} \in \Tc$, $\wv \in \Tc_p = {\rm span}(\{\hat\uv_i\}_{i \in \Dc}) \cap {\rm span}(\Uv_{1,\Sc}^\perp)$ for each $p$. That is, $\wv \in \Tc_p$ can be expressed as $\wv = \hat{\Vv}_1 \av_1$ for some $\av_1 \in \Real^{m_1+m_2}$ such that $\Kv^\top \av_1 = \0v_m$. Hence, $\av_1^\top [\tilde{\Omegav}_{1}]^j = o_p(1)$ for all $1 \leq j \leq m$. For $Y \in \Yc$, assume $\pi(Y) = 1$. From (\ref{eq:thm:L1:a+b:ne}), we have
\begin{align*}
    \frac{1}{\sqrt{p}} \wv ^\top (Y - \bar{X}) = \av_1^\top \frac{1}{\sqrt{p}}\hat{\Vv}_1^\top (Y - \bar{X}) = \sum_{j=1}^{m} t_{j} \av_{1}^\top [\tilde{\Omegav}_1]^j + o_p(1) \xrightarrow{P} 0
\end{align*}
as $p \to \infty$ where $t_j$ is defined in (\ref{eq:tj}). Similarly, we can show the same result for $Y \in \Yc$ with $\pi(Y) = 2$. 

(ii) Recall that $\Sc = {\rm span}(\Uv_{1,\Sc}) \oplus \Tc_p \oplus {\rm span}(\fv_0)$. Write 
\begin{align*}
    \fv_0 = \tilde\Vv\av = \hat{\Vv}_1\av_{0,1} + a_{0, \textup{MDP}}\mdp
\end{align*}
where $\av = (\av_{0,1}, a_{0, \textup{MDP}})^\top \in \Real^{m_1+m_2+1}$ with $\av_{0,1} \in \Real^{m_1+m_2}$ and $a_{0, \textup{MDP}} \in \Real$. Since $\fv_0$ is orthogonal to $\Tc_p = {\rm span}(\{\hat\uv_i\}_{i \in \Dc}) \cap {\rm span}(\Uv_{1,\Sc}^\perp)$, we have $\av_{0,1} = \Kv \pv$ for some $\0v_m \ne \pv \in \Real^{m}$. Also, since $\fv_0$ is orthogonal to ${\rm span}(\Uv_{1, \Sc})$, we have 
\begin{align*}
    \Uv_1^\top \fv_0 =  \Kv^\top \av_{0,1} + \kv a_{0, \textup{MDP}} = \Kv^\top \Kv \pv + \kv a_{0, \textup{MDP}} = \0v_m.
\end{align*}
Lastly, since $\fv_0$ is an unit vector, $\norm{\av}^2 =  \norm{\Kv \pv}^2 + a_{0, \textup{MDP}}^2 = 1$. Combining these, we have
\begin{align*}
    \av_{0,1} = \Kv\pv = -\Kv(\Kv^\top \Kv)^{-1}\kv a_{0,\textup{MDP}}
\end{align*}
and 
\begin{align*}
    a_{0,\textup{MDP}}^2 = \{\kv^\top (\Kv^\top \Kv)^{-1} \kv +1 \}^{-1}.
\end{align*}
Moreover, by (\ref{eq:tildeOmega1}) and (\ref{eq:tildeOmega12}), it holds that
\begin{align*}
     \av_{0,1} &= -\Kv(\Kv^\top \Kv)^{-1}\kv a_{0,\textup{MDP}} \\
     &\xrightarrow{P} {\psiv}_{0,1} := \frac{\psi_{0,\textup{MDP}}}{\kappa} \tilde{\Omegav}_1 (\Iv_{m} - (\tilde{\Omegav}_1^\top\tilde{\Omegav}_1)^{-1} )\rv
\end{align*}
as $p \to \infty$ where $\psi_{0, \textup{MDP}}$ is the probability limit of $a_{0, \textup{MDP}}$ defined in the next equation:
\begin{equation}\label{eq:app:psi0MDP}
    \begin{aligned}
         a_{0,\textup{MDP}} &= \{\kv^\top (\Kv^\top \Kv)^{-1} \kv +1 \}^{-1/2} \\
        &\xrightarrow{P} \psi_{0,\textup{MDP}} := \frac{\kappa}{\sqrt{\kappa^2 + \rv^\top (\Iv_{m} - \tilde{\Omegav}_1^\top\tilde{\Omegav}_1) (\tilde{\Omegav}_1^\top\tilde{\Omegav}_1)^{{-1}} (\Iv_{m} - \tilde{\Omegav}_1^\top\tilde{\Omegav}_1)\rv}}.
    \end{aligned}
\end{equation}
Note that $\psiv_0 := (\psiv_{0,1}^\top, \psi_{0, \textup{MDP}})^\top$ satisfies ${\psiv}_0^\top [\tilde{\Omegav}]^j = 0$ for all $1 \leq j \leq m$ where $\tilde{\Omegav}$ is defined in (\ref{eq:tildeOmega}). Now, for $Y \in \Yc$, assume $\pi(Y) = 1$. Combining (\ref{eq:thm:L1:a+b:ne}) and (\ref{eq:L1:MDP:a:2}) gives
\begin{align*}
    &\frac{1}{\sqrt{p}}\fv_0^\top (Y - \bar{X}) = \av_{0,1}^\top \frac{1}{\sqrt{p}}\hat{\Vv}_1^\top (Y - \bar{X}) + a_{0,\textup{MDP}}\frac{1}{\sqrt{p}}\mdp^\top (Y - \bar{X}) \\
    &\xrightarrow{P} \sum_{j=1}^{m}t_{j} {\psiv}_{0}^\top [\tilde{\Omegav}]^{j} + \frac{\psi_{0,\textup{MDP}}}{\kappa}\left\{\eta_2(1-\cos^2\varphi)\delta^2 - \frac{\tau_1^2 - \tau_2^2}{n} \right\} \\
    &=\frac{\psi_{0,\textup{MDP}}}{\kappa}\left\{\eta_2(1-\cos^2\varphi)\delta^2 - \frac{\tau_1^2 - \tau_2^2}{n} \right\} =: \upsilon_0\left\{\eta_2(1-\cos^2\varphi)\delta^2 - \frac{\tau_1^2 - \tau_2^2}{n} \right\}
\end{align*}
as $p \to \infty$ where $t_j$ is defined in (\ref{eq:tj}) and 
\begin{align}\label{eq:app:upsilon0}
    \upsilon_0 := \frac{\psi_{0,\textup{MDP}}}{\kappa} = \left\{\kappa^2 + \rv^\top (\Iv_{m} - \tilde{\Omegav}_1^\top\tilde{\Omegav}_1) (\tilde{\Omegav}_1^\top\tilde{\Omegav}_1)^{{-1}} (\Iv_{m} - \tilde{\Omegav}_1^\top\tilde{\Omegav}_1)\rv\right\}^{-1/2}    
\end{align}
is a strictly positive random variable. Similarly, we can show that $p^{-1/2}\fv_0^\top (Y - \bar{X}) \xrightarrow{P} \upsilon_0 \{-\eta_1(1-\cos^2\varphi)\delta^2 - (\tau_1^2 - \tau_2^2)/n\}$ as $p \to \infty$ for $Y \in \Yc$ with $\pi(Y) = 2$.

Following the above arguments, we can also prove the same results for the case where $\beta_1 = \beta_2 = 1$ and $\tau_1 = \tau_2 =: \tau$ (with $\Dc = \{1, \ldots, m \}$ and $\hat\Vv_1 = [\hat\uv_1, \ldots, \hat\uv_m]$). We remark that, in this case, 
\begin{align*}
    \hat\uv_i^\top \uv_j \xrightarrow{P} [\tilde{\Omegav}_1]_{i,j}= C_{i,j} = \sqrt{\frac{\phi_i(\Phiv)}{\phi_i(\Phiv) + \tau^2}} v_{ij}(\Phiv)
\end{align*}
as $p \to \infty$ (see Lemma~\ref{lem:asymptotic property of Sw equal tails} and this can be obtained in a similar manner to $D_{i,j}$ in the proof of Lemma~\ref{lem:asymptotic property of Sw unequal tails}). Then, from (\ref{eq:app:upsilon0}), it can be checked that
\begin{align}\label{eq:app:upsilon0:equal tails}
    \upsilon_0 =  \kappa^2  + \sum_{i=1}^{m} \frac{\tau^4}{\phi_i(\Phiv)(\phi_i(\Phiv) + \tau^2)} \left(\sum_{j=1}^{m}r_{j} v_{ij}(\Phiv) \right)^2.
\end{align}
Later, in the proof of Theorem~\ref{thm:ridge equal tail eigenvalues piling distance}, we will show that the asymptotic test data piling distance of $\fv_0$, $D(\fv_0) = \upsilon_0(1-\cos^2\varphi)\delta^2$, is the same as that of the projected ridged linear discriminant vector $\vv_{\hat\alpha}$ in Theorem~\ref{thm:ridge equal tail eigenvalues piling distance}.
\end{proof}
\end{subsection}

\begin{subsection}{Proof of Theorem~\ref{thm:SMDP characterization}}
\begin{proof}
We continue to use the same notations as in the proof of Theorem~\ref{thm:SMDP piling distance}. We assume $\beta_1 = \beta_2 = 1$ and $\tau_1 > \tau_2$. The other cases can be shown in a similar manner.

(i) For each $p$, let $\left\{\fv_i\right\}_{i=1}^{m_1+m_2-m}$ forms an orthonormal basis of $\Tc_p$. Also, choose $\{\gv_{i} \}_{i=1}^{m}$ so that $\{\fv_0 \} \cup \{\fv_i\}_{i=1}^{m_1+m_2-m} \cup \{\gv_{i}\}_{i=1}^{m}$  forms an orthonormal basis of $\Sc$. Then $\{\fv_0 \} \cup \{\fv_i\}_{i=1}^{m_1+m_2-m} \cup \{\gv_{i}\}_{i=1}^{m} \cup \{\hat\uv_i \}_{i \in \{1, \ldots, n-2 \} \setminus \Dc}$ forms an orthonormal basis of $\Sc_{\Xc}$. 

Write $\Bc_p = {\rm span}(\fv_0) \oplus \Tc_p\oplus {\rm span}(\left\{\hat\uv_i \right\}_{i \in \left\{1, \ldots, n-2 \right\} \setminus \Dc})$. For any given $\left\{\wv \right\} \in \Ac$, write $\wv = a_0\fv_0 + \sum_{i=1}^{m_1+m_2-m}a_i\fv_i + \sum_{i=1}^{m}b_i\gv_i +\sum_{i \in \left\{1, \ldots, n-2 \right\}\setminus\Dc}c_i\hat\uv_i$. Note that for all $0 \leq i \leq m_1+m_2-m$, $\fv_i$ is orthogonal to $\uv_j$ for $1 \leq j \leq m$. Also, from (\ref{eq:uij}) we have $\hat\uv_i^\top \uv_j \xrightarrow{P} 0$ as $p \to \infty$ for $i \in \{1, \ldots, n-2\} \setminus \Dc$ and $1 \leq j \leq m$. Using Lemma 3.3 (i) of \citet{Chang2021}, it can be shown that $\{\wv\} \in \Ac$ is equivalent to $\wv^\top \uv_j \xrightarrow{P} 0$ as $p \to \infty$ for $1 \leq j \leq m$. Hence, $b_i = o_p(1)$ for $1 \leq i \leq m$. Let $\{ \vv \} \in \Bc_p$ with $\vv = P_{\Bc_p} \wv / \| P_{\Bc_p}\wv \| \in \Bc_p$ for all $p$. Then $\| \wv - \vv\|  \xrightarrow{P} 0$ since $\|P_{\Bc_p}\wv \|^2 = \| \wv\|^2 - \sum_{i=1}^{m}b_i^2 \xrightarrow{P} 1$ as $p \to \infty$. 

(ii) Theorem~\ref{thm:SMDP piling distance} tells that $D(\fv_i) = 0$ for all $1 \leq i \leq m_1+m_2-m$ while $D(\fv_0) > 0$. Also, since $p^{-1/2}\wv^\top \muv \xrightarrow{P} D(\wv)$ as $p \to \infty$ for any $\{\wv \} \in \Ac$ such that $D(\wv)$ exists, (\ref{eq:app:lemma:md:a}) leads to $D(\hat\uv_i) = 0$ for $i \in \left\{1, \ldots, n-2 \right\} \setminus \Dc$. Then $p^{-1/2}\wv^\top(Y_1 - Y_2) \leq \abs{a_0}D(\fv_0) + o_p(1)$ for $Y_i \in \Yc$ with $\pi(Y_i) = i$ $(i = 1, 2)$. Now $\abs{a_0} \leq 1$ implies that $D(\wv) \leq D(\fv_0)$ and the equality holds if and only if $a_0 \xrightarrow{P} 1$ as $p \to \infty$.
\end{proof}
\end{subsection}

\begin{subsection}{Proof of Theorem~\ref{thm:scatter of projected test data}}
To distinguish between the notations for the training dataset $\Xc$ and the independent test dataset $\Yc$ in Section~\ref{sec:data-splitting approach}, we use a superscript for the latter. Specifically, any quantity defined with a superscript * refers to the independent test dataset $\Yc$, whereas quantities without the superscript pertain to the training dataset $\Xc$. For example, $\Sv_W^*$ is the within-class scatter matrix of $\Yc$ and $\hat\uv_i^*$ is the $i$th eigenvector of $\Sv_W^*$. Also, $\Dc^* \subset \{1, \ldots, n^*-2 \}$ is the index set that is defined analagously to $\Dc$ in Table~\ref{table:D} in Section~\ref{sec:test data piling signal subspace} for the independent test dataset $\Yc$. 

We assume $\beta_1 = \beta_2 = 1$ and $\tau_1 > \tau_2$. The other cases can be shown in a similar manner to this case but are much simpler. For notational simplicity, we write $\Dc = \left\{i_1, \ldots, i_{m_1+m_2}\right\}$ and ${\Dc}^* = \left\{j_1, \ldots, j_{m_1+m_2} \right\}$ so that $i_l < i_{l'}$ and $j_l < j_{l'}$ if $l < l'$.

\begin{proof}
First, we will find the probability limit of $\hat\uv_{i}^\top \hat\uv_{j}^*$ for $1 \leq i \leq n-2$ and $1 \leq j \leq n^*-2$. Using (\ref{eq:app:svd}) and Lemmas~\ref{lem:asymptotic property of Sw unequal tails} and \ref{app:lemma:dual}, we have
\begin{align*}
    &\hat\uv_i^\top \hat\uv_j^* \\
    &\xrightarrow{P} \begin{cases} (\phi_i(\Sv_0)\phi_j(\Sv_0^*))^{-1/2} \{\Wv^\top (\Iv_{n} - \Jv)v_{i}(\Sv_0)\}^\top \{\Wv^{* \top} (\Iv_{n^*} - \Jv^*) v_{j}(\Sv_0^*)\}, & i \in \Dc, j \in \Dc^* \\
    0, & o.w
    \end{cases}
\end{align*}
as $p \to \infty$ where $\Sv_0^*$ is the probability limit of $p^{-1}(\Yv - \bar{\Yv})^\top(\Yv - \bar{\Yv})$, which is the scaled dual matrix of $\Sv_W^*$ (see Lemma~\ref{app:lemma:dual}). To be specific, for $i_l \in \Dc$ and $j_{l'} \in {\Dc}^*$ $(1 \leq l, l' \leq m_1+m_2)$, we have
\begin{equation}\label{eq:app:sampleevec:2}
    \begin{aligned}
        \hat\uv_{i_l}^\top \hat\uv_{j_{l'}}^* \xrightarrow{P}&~  
            \sum_{k=1}^{m} D_{l,k}D_{l',k}^*
    \end{aligned}
\end{equation}
as $p \to \infty$ where $D_{l,k}$ is defined in (\ref{eq:Dij}) and $D_{l',k}^*$ is defined analogously:
\begin{align*}
    D_{l',k}^* = \frac{1}{\sqrt{\phi_{l'}(\Phiv_{\tau_1, \tau_2}^*)}} \sum_{i=1}^{2} [\Rv_{i}]_k{\Phiv_i^*}^{1/2} \tilde{v}_{l'i}(\Phiv_{\tau_1, \tau_2}^*)
\end{align*}
for $k = 1, \ldots, m$. 

Next, to obtain the probability limit of $\mdp^\top \hat\uv_{j}^*$, note that 
\begin{align*}
        \mdp^\top \hat\uv_j^* &= \hat\uv_j^{* \top} \frac{\hat\Uv_2 \hat\Uv_2^\top \dv}{\|\hat\Uv_2 \hat\Uv_2^\top \dv \|} = \frac{1}{\kappa_{\textup{MDP}}} \left(\frac{1}{\sqrt{p}}\hat\uv_j^{* \top} \dv - \frac{1}{\sqrt{p}} \hat\uv_j^{* \top}\hat\Uv_1 \hat\Uv_1^\top \dv \right).
\end{align*}
Using similar arguments to the proof of Lemma~\ref{app:lemma:md}, for $j_{l'} \in {\Dc}^*$, we have
\begin{equation}\label{eq:app:sampleK1:2}
    \begin{aligned}
        \frac{1}{\sqrt{p}}\hat\uv_{j_{l'}}^{* \top} \dv \xrightarrow{P} \sum_{k=1}^{m} r_k D^*_{l',k}
    \end{aligned}
\end{equation}
and 
\begin{equation}\label{eq:app:sampleK2:2}
    \begin{aligned}
      \frac{1}{\sqrt{p}}\hat\uv_{j_{l'}}^{* \top} \hat\Uv_1 \hat\Uv_1^\top \dv &= \sum_{i \in \Dc}(\hat\uv_{j_{l'}}^{* \top}\hat\uv_i)\left(\frac{1}{\sqrt{p}}\hat\uv_i^\top \dv \right) + o_p(1)\\
      &\xrightarrow{P}  \sum_{i=1}^{m_1+m_2} \sum_{k=1}^{m} \sum_{k'=1}^{m} r_{k'}D_{i,k}D_{i,k'} D_{l',k}^* 
    \end{aligned}
\end{equation}
as $p \to \infty$ where $r_k$ is defined in Lemma~\ref{app:lemma:md}. Combining (\ref{eq:app:sampleK1:2}) and (\ref{eq:app:sampleK2:2}) gives
\begin{equation}\label{eq:app:angMDP:2}
    \begin{aligned}
        \mdp^\top \hat\uv_{j_{l'}}^* \xrightarrow{P} \frac{1}{\kappa}  \sum_{k=1}^{m}\sum_{k'=1}^{m} \left(I_{kk'} - \sum_{i=1}^{m_1+m_2} D_{i,k}D_{i,k'}\right) r_{k'} D_{l',k}^*
    \end{aligned}
\end{equation}
as $p \to \infty$ where $I_{kk'} = I(k = k')$. In contrast, for $j \notin {\Dc}^*$, we can show that $\mdp^\top \hat\uv_j^* \xrightarrow{P} 0$ as $p \to \infty$.

Now, let $\xi_{i,j}$ be the probability limit of $\hat\uv_{i}^\top \hat\uv_{j}^*$ and $\xi_{\textup{MDP},j}$ be the probability limit of $\mdp^\top \hat\uv_{j}^*$, and write 
\begin{align*}
    \xiv_{j} = (\xi_{1,j}, \ldots, \xi_{n-2,j}, \xi_{\textup{MDP}, j})^\top 
    \in \Real^{n-1}
\end{align*}
for $1 \leq j \leq n-2$. Also, write
\begin{align*}
    \Vv = [\hat\uv_1, \ldots, \hat\uv_{n-2}, \mdp] \in \Real^{p \times (n-1)}
\end{align*}
and denote the probability limit of the $(n-1) \times (n-1)$ matrix $p^{-1}\Vv^\top \Sv_W^* \Vv$ by $\Lv$. Since $\hat\uv_i^\top \hat\uv_j^* \xrightarrow{P} 0$ and $\mdp^\top \hat\uv_j^* \xrightarrow{P} 0 $ as $p \to \infty$ for $j \notin {\Dc}^*$, we have $\xiv_j = \0v_{n-1}$ for $j \notin {\Dc}^*$ and 
\begin{align}\label{eq:L}
    {\Lv} = \sum_{l'=1}^{m_1+m_2} \phi_{l'}(\Phiv_{\tau_1,\tau_2}^*) \xiv_{j_{l'}} \xiv_{j_{l'}}^{\top}.
\end{align}
We will show that $\rank{(\Lv)} = m$. Recall that $\phi_{l'}(\Phiv_{\tau_1, \tau_2}^*) > \tau_2^2 > 0$ for all $l' = 1, \ldots, m_1+m_2$. Now, we define the $(n-2) \times m$ matrix $\Omegav_1$ such that
\begin{align}\label{eq:Omega1}
    [\Omegav_1]_{i_l,j} = [\tilde{\Omegav}_{1}]_{l,j},
\end{align}
for $1 \leq l \leq m_1+m_2$ and $1 \leq j \leq m$ where $\tilde{\Omegav}_1$ is defined in (\ref{eq:tildeOmega1}), and 
\begin{align*}
    [\Omegav_1]_{i,j} = 0
\end{align*}
for $i \notin \Dc$ and $1 \leq j \leq m$. Also, we define the $(n-1) \times m$ matrix 
\begin{align}\label{eq:Omega}
    \Omegav = [\Omegav_1^\top, ~\omegav_1]^\top \in \Real^{(n-1) \times m}
\end{align}
where $\omegav_1$ is the $m \times 1$ vector defined in (\ref{eq:tildeOmega12}). Lastly, analogously to $\tilde{\Omegav}_1$ in (\ref{eq:tildeOmega1}), we define the $(m_1+m_2) \times m$ matrix $\tilde{\Omegav}_1^*$ such that
\begin{align}\label{eq:Omega1star}
    [\tilde{\Omegav}_1^*]_{i,j} = D_{i,j}^*
\end{align}
for $1 \leq i \leq m_1+m_2$ and $1 \leq j \leq m$. Then from (\ref{eq:app:sampleevec:2}) and (\ref{eq:app:angMDP:2}), 
\begin{align}\label{eq:tildeQ}
    {\Xiv} := [\xiv_{j_1}, \ldots, \xiv_{j_{m_1+m_2}}] = \Omegav \tilde{\Omegav}_1^{* \top} \in \Real^{(n-1) \times (m_1+m_2)}.
\end{align}
Since both of $\Omegav$ and $\tilde{\Omegav}_1^*$ are of rank $m$, we have $\rank{(\Lv)} = \rank{(\Xiv)} = m$.
\end{proof}
\end{subsection}

\begin{subsection}{Proof of Theorem~\ref{thm:data-splitting approach piling distance}}
\begin{proof}
(i) For any given $\left\{ \wv \right\} \in \bar{\Ac}$, we have already seen that there exists $\left\{\vv \right\}$ such that $\vv \in {\rm span}(\Vv\hat\Qv_2)$ for all $p$ and $\|\wv - \vv \| \xrightarrow{P} 0$ as $p \to \infty$. Thus it suffices to show that $\vv^\top \uv_j \xrightarrow{P} 0$ as $p \to \infty$. Let $\vv = \Vv\bv$ such that $\bv \in {\rm span}(\hat\Qv_2)$. Note that $\bv$ is asymptotically orthogonal to ${\rm span}(\Lv) = {\rm span}(\Omegav)$ where $\Lv$ and $\Omegav$ are in (\ref{eq:L}) and (\ref{eq:Omega}), respectively. Using (\ref{eq:uij}) and (\ref{eq:L1:MDP:b:2}), we have
\begin{align}\label{eq:basis:orthogonal}
    \vv^\top \uv_j &= \bv^\top \Vv^\top \uv_j = \bv^\top [\Omegav]^j + o_p(1) \xrightarrow{P} 0
\end{align}
as $p \to \infty$ for $1 \leq j \leq m$. 

(ii) For any given $\left\{ \wv \right\} \in \bar{\Ac}$, we assume $\wv = \Vv\av$ where $\av = (a_1, \ldots, a_{n-2}, a_{\textup{MDP}})^\top$ satisfies $a_{\textup{MDP}} \xrightarrow{P} \psi_{\textup{MDP}}$ as $p \to \infty$. Recall that there exists $\left\{\vv \right\}$ such that $\vv \in {\rm span}(\Vv\hat\Qv_2)$ for all $p$ and $\|\wv - \vv \| \xrightarrow{P} 0$ as $p \to \infty$. Write $\vv = \Vv\bv$ with $\bv = (b_1, \ldots, b_{n-2}, b_{\textup{MDP}})^\top$. Then $\|\wv-\vv\| = \|\av-\bv \| \xrightarrow{P} 0$ and $b_{\textup{MDP}} \xrightarrow{P} \psi_{\textup{MDP}}$ as $p \to \infty$. Note that $p^{-1/2}\wv^\top (Y -\bar{X}) = p^{-1/2}\vv^\top (Y -\bar{X}) + o_p(1)$ and it suffices to obtain the probability limit of $p^{-1/2}\vv^\top (Y -\bar{X})$. For any independent observation $Y$, which is independent to both of $\Xc$ and $\Yc$, assume that $\pi(Y) = 1$. Combining (\ref{eq:thm:L1:a+b:ne}), (\ref{eq:L1:MDP:a:2}) and (\ref{eq:basis:orthogonal}) gives
\begin{align*}
    \frac{1}{\sqrt{p}}\vv^\top(Y - \bar{X}) &=  \sum_{j=1}^{m} t_j \bv^\top [\Omegav]^j + \frac{\psi_{\textup{MDP}}}{\kappa}\left\{\eta_2(1-\cos^2\varphi)\delta^2 - \frac{\tau_1^2 - \tau_2^2}{n}  \right\} + o_p(1) \\
    &\xrightarrow{P} \frac{\psi_{\textup{MDP}}}{\kappa}\left\{ \eta_2(1-\cos^2\varphi)\delta^2 - \frac{\tau_1^2 - \tau_2^2}{n}  \right\}
\end{align*}
as $p \to \infty$ where $t_j$ is defined in (\ref{eq:tj}). Similarly, we can show that
\begin{align*}
    \frac{1}{\sqrt{p}}\wv^\top(Y - \bar{X}) \xrightarrow{P} \frac{\psi_{\textup{MDP}}}{\kappa}\left\{-\eta_1(1-\cos^2\varphi)\delta^2 - \frac{\tau_1^2 - \tau_2^2}{n} \right\}
\end{align*}
as $p \to \infty$ for any independent observation $Y$, which is independent to both of $\Xc$ and $\Yc$, with $\pi(Y) = 2$.
\end{proof}
\end{subsection}

\begin{subsection}{Proof of Theorem~\ref{thm:data-splitting approach SMDP}}
\begin{proof}
Theorem~\ref{thm:data-splitting approach piling distance} tells that, for $\left\{ \wv\right\} \in \bar{\Ac}$ with $\wv = \Vv\av$ and $a_{\textup{MDP}} = \ev_{\textup{MDP}}^\top \av \xrightarrow{P} \psi_{\textup{MDP}}$ as $p \to \infty$, the asymptotic distance between the two piles of independent test data is given by 
\begin{align*}
    D(\wv) = \frac{\psi_{\textup{MDP}}}{\kappa}(1-\cos^2\varphi)\delta^2.
\end{align*}
Let $\smdp = \Vv\av_{\textup{SMDP}}$ where
$\av_{\textup{SMDP}} = \|\hat\Qv_2\hat\Qv_2^\top\ev_{\textup{MDP}} \|^{-1}\hat\Qv_2\hat\Qv_2^\top\ev_{\textup{MDP}}$
and $\ev_{\textup{MDP}} = (\0v_{n-2}^\top, 1)^\top$. Note that
\begin{equation}\label{eq:SMDPloading}
    \begin{aligned}
    \ev_{\textup{MDP}}^\top \av_{\textup{SMDP}} &= 
\|\hat\Qv_2\hat\Qv_2^\top\ev_{\textup{MDP}}\|^{-1}\ev_{\textup{MDP}}^\top \hat\Qv_2\hat\Qv_2^\top\ev_{\textup{MDP}} =\|\hat\Qv_2^\top\ev_{\textup{MDP}} \|.
    \end{aligned}
\end{equation}
As shown in the proof of Theorem~\ref{thm:scatter of projected test data}, $\Lv$, which is the probability limit of the $(n-1)\times (n-1)$ matrix $p^{-1}\Vv^{\top}\Sv_{W}^*\Vv = \hat\Qv\Hv \hat\Qv^\top$, is of rank $m$. Moreover, from (\ref{eq:tildeQ}), we have ${\rm span}(\Lv) = \rank{(\Xiv)} = {\rm span}(\Omegav) = m$ where $\Omegav$ is defined in (\ref{eq:Omega}). Then using similar arguments in the proof of Theorem~\ref{thm:SMDP piling distance}, we can show that the probability limit of (\ref{eq:SMDPloading}) is $\kappa\upsilon_0$, where $\upsilon_0$ is defined in (\ref{eq:app:upsilon0}). Hence, we have 
\begin{align*}
    D(\smdp) = \upsilon_0 (1-\cos^2\varphi)\delta^2 > 0
\end{align*}
with probability $1$. Then, by applying similar arguments as in the proof of Theorem~\ref{thm:SMDP characterization}, we conclude the proof.

\end{proof}
\end{subsection}

\subsection{Proof of Theorem~\ref{thm:ridge concentration}}
\begin{proof}
Write
\begin{align*}
    \tilde{\wv}_{\alpha} = \sum_{i=1}^{n-2} \frac{\alpha_p}{\hat\lambda_i + \alpha_p}\hat\uv_i \left(\frac{1}{\sqrt{p}}\hat\uv_i^\top \dv \right) + \frac{1}{\sqrt{p}}\|\hat\Uv_2 \hat\Uv_2^\top \dv\| \mdp.
\end{align*}
Then $\wv_{\alpha} \varpropto \tilde{\wv}_{\alpha}$ for any $\alpha \in \Real$ where $\wv_{\alpha}$ is defined. Also, note that ${\rm Angle}(\wv_{\alpha}, \Sc) = \arccos{(\|P_{\Sc}\tilde{\wv}_{\alpha} \| / \|\tilde{\wv}_{\alpha}\| )}$. We decompose $\| \tilde{\wv}_{\alpha}\|^2$ into the two terms:
\begin{equation}\label{eq:concetarationRidge}
    \begin{aligned}
        \|\tilde{\wv}_{\alpha} \|^2 = \|P_{\Sc}\tilde{\wv}_{\alpha} \|^2 + \sum_{i \in \left\{1, \ldots, n-2 \right\} \setminus \Dc} \left(\frac{\alpha_p}{\hat\lambda_i + \alpha_p} \right)^2 \left(\frac{1}{\sqrt{p}}\hat\uv_i^\top \dv \right)^2.
    \end{aligned}
\end{equation}
Using Lemmas~\ref{lem:asymptotic property of Sw unequal tails} and~\ref{app:lemma:md}, the second term in (\ref{eq:concetarationRidge}) converges to zero in probability for $\alpha \in \Real \setminus \left\{-\tau_1^2, -\tau_2^2\right\}$. Then it suffices to show that $\|P_{\Sc}\tilde{\wv}_{\alpha} \|$ is stochastically bounded. Note that
\begin{align*}
    \|P_{\Sc}\tilde{\wv}_{\alpha} \|^2 = \sum_{i\in \Dc} \left(\frac{\alpha_p}{\hat\lambda_i + \alpha_p} \right)^2 \left(\frac{1}{\sqrt{p}}\hat\uv_i^\top \dv \right)^2 + \kappa_{\textup{MDP}}^2.
\end{align*}
Again from Lemmas~\ref{lem:asymptotic property of Sw unequal tails} and~\ref{app:lemma:md}, the first term converges in probability since 
\begin{align*}
    \sum_{i \in \Dc}\left(\frac{\alpha_p}{\hat\lambda_i + \alpha_p} \right)^2 \left(\frac{1}{\sqrt{p}}\hat\uv_i^\top \dv \right)^2 \xrightarrow{P} \sum_{l=1}^{m_1+m_2}\left(\frac{\alpha}{\phi_l(\Phiv_{\tau_1, \tau_2}) + \alpha} \right)^2\left(\sum_{j=1}^{m} r_{j}D_{l,j} \right)^2
\end{align*}
as $p \to \infty$ where $D_{l,j}$ is defined in (\ref{eq:Dij}). We have already shown that $\kappa_{\textup{MDP}}^2$ converges to $\kappa^2 > 0$ in the proof of Theorem~\ref{thm:test data piling}. These complete the proof. 
\end{proof}

\subsection{Proof of Proposition~\ref{prop:ridge one-spike unequal tails m = 1}}
\begin{proof}
Assume that $\beta_1 = \beta_2 = 1$ and $\tau_1 > \tau_2$. Assume further that $m_1 = m_2 = m = 1$ (that is, $\uv_{(1),1} = \uv_{(2),1} = \uv_1$). (i) is a special case of Theorem~\ref{thm:ridge unequal tails asymptotically orthogonal}. To show (ii), note that for any given $\alpha \in \Real$, $\left\{\vv_{\alpha} \right\} \in \Ac$ if and only if $\vv_{\alpha}^\top \uv_{1} \xrightarrow{P} 0$ as $p \to \infty$. From Lemmas~\ref{lem:asymptotic property of Sw unequal tails} and~\ref{app:lemma:md}, we have 
\begin{align*}
    \tilde{\vv}_{\alpha}^\top \uv_{1} \xrightarrow{P} r_1 \left(1 - \sum_{l=1}^{2} \frac{\phi_l(\Phiv_{\tau_1,\tau_2})}{\phi_l(\Phiv_{\tau_1,\tau_2}) + \alpha}  D_{l,1}^2 \right)
\end{align*}
as $p \to \infty$ where $\tilde{\vv}_{\alpha} \varpropto \vv_{\alpha}$ is defined in (\ref{eq:app:pridge:conditioned:eq}) and $D_{l,1}$ is defined in (\ref{eq:Dij}). Using (i), we have 
\begin{align}\label{eq:app:tau1ntau2}
     r_1 \left(1 - \sum_{l=1}^{2} \frac{\phi_l(\Phiv_{\tau_1,\tau_2})}{\phi_l(\Phiv_{\tau_1,\tau_2}) - \tau_1^2} D_{l,1}^2 \right) = 0 \text{ and } r_1 \left(1 - \sum_{l=1}^{2} \frac{\phi_l(\Phiv_{\tau_1,\tau_2})}{\phi_l(\Phiv_{\tau_1,\tau_2}) - \tau_2^2} D_{l,1}^2 \right) = 0.
\end{align}
where $r_1$ is defined in Lemma~\ref{app:lemma:md}. Hence, we can write
\begin{align*}
    \tilde{\vv}_{\alpha}^\top \uv_1 \xrightarrow{P} \frac{r_1(\alpha+\tau_1^2)(\alpha+\tau_2^2)}{(\alpha+\phi_1(\Phiv_{\tau_1,\tau_2}))(\alpha+\phi_2(\Phiv_{\tau_1,\tau_2}))}
\end{align*}
as $p \to \infty$. Since $\|\tilde{\vv}_{\alpha}\|$ is stochastically bounded, $\vv_{\alpha}^\top \uv_1 \xrightarrow{P} 0$ as $p \to \infty$ if and only if $\alpha = -\tau_1^2$ or $\alpha = -\tau_2^2$. (iii) is a special case of Theorem~\ref{thm:ridge unequal tails piling distance}.
\end{proof}

\begin{subsection}{Proof of Theorem~\ref{thm:compare two ridges}}
\begin{proof}
We continue to use the notations in the proof of Proposition~\ref{prop:ridge one-spike unequal tails m = 1}.  From Proposition~\ref{prop:ridge one-spike unequal tails m = 1}, we can check that $D(\vv_{\hat\alpha_k})$, the asymptotic distance between $P_{\vv_{\hat\alpha_k}}\Yc_1$ and $P_{\vv_{\hat\alpha_k}}\Yc_2$, is $\gamma_{k}(1-\cos^2\varphi)\delta^2$ for each $k = 1, 2$, where $\gamma_{1}$ and $\gamma_{2}$ are defined in (\ref{eq:gamma1}) and (\ref{eq:gamma2}), respectively. To be specific, when $m_1 = m_2 = m = 1$,
\begin{align}\label{eq:gamma1:m=1}
    \gamma_{k} = \left(r_1^2 \sum_{l=1}^{2} \frac{\tau_k^4}{(\phi_l(\Phiv_{\tau_1, \tau_2}) - \tau_k^2)^2}D_{l,1}^2 + \kappa^2 \right)^{-1/2}
\end{align}
as $p \to \infty$. It can be checked that
\begin{equation}\label{eq:app:Phi11}
    \begin{aligned}
        D_{1,1}^2 = \frac{\phi_{1}(\Phiv_{\tau_1, \tau_2}) - \tau_2^2 - v_{11}^2(\Phiv_{\tau_1, \tau_2})(\tau_1^2 - \tau_2^2)}{\phi_1(\Phiv_{\tau_1,\tau_2})} \\
        D_{2,1}^2 = \frac{\phi_{2}(\Phiv_{\tau_1, \tau_2}) - \tau_1^2 - v_{22}^2(\Phiv_{\tau_1, \tau_2})(\tau_2^2 - \tau_1^2)}{\phi_{2}(\Phiv_{\tau_1,\tau_2})}
    \end{aligned}
\end{equation}
and
\begin{equation}\label{eq:app:v11}
    \begin{aligned}
        v_{11}^2(\Phiv_{\tau_1,\tau_2}) = v_{22}^2(\Phiv_{\tau_1,\tau_2}) &=\frac{(\phi_1(\Phiv_{\tau_1,\tau_2}) - \tau_2^2)(\tau_1^2 - \phi_{2}(\Phiv_{\tau_1, \tau_2}))}{(\tau_1^2 - \tau_2^2)(\phi_1(\Phiv_{\tau_1, \tau_2}) - \phi_2(\Phiv_{\tau_1,\tau_2}))} \\
        &= \frac{1}{2} + \frac{\Phiv_{1}-\Phiv_{2}+\tau_1^2 -\tau_2^2}{2\sqrt{(\Phiv_{1}-\Phiv_{2}+\tau_1^2 -\tau_2^2)^2 + 4\Phiv_1 \Phiv_2}}.
    \end{aligned}
\end{equation}
Combining (\ref{eq:app:Phi11}) and (\ref{eq:app:v11}) gives
\begin{align*}
    &D(\vv_{\hat\alpha_1}) \leq D(\vv_{\hat\alpha_2}) \Leftrightarrow \gamma_{1}^2 \leq \gamma_{2}^2 \\
    &\Leftrightarrow \sum_{l=1}^{2} \frac{\tau_1^4}{(\phi_l(\Phiv_{\tau_1, \tau_2}) - \tau_1^2)^2}D_{l,1}^2 \ge \sum_{l=1}^{2} \frac{\tau_2^4}{(\phi_l(\Phiv_{\tau_1, \tau_2}) - \tau_2^2)^2}D_{l,1}^2 \\
    &\Leftrightarrow \tau_1^4\left(\frac{\phi_2(\Phiv_{\tau_1,\tau_2})(\phi_1(\Phiv_{\tau_1,\tau_2}) - \tau_2^2)}{\phi_1(\Phiv_{\tau_1,\tau_2})-\tau_1^2} + \frac{\phi_1(\Phiv_{\tau_1,\tau_2})(\tau_2^2- \phi_2(\Phiv_{\tau_1,\tau_2}))}{\phi_2(\Phiv_{\tau_1,\tau_2})-\tau_1^2} \right) \\
    &\quad \ge \tau_2^4\left(\frac{\phi_2(\Phiv_{\tau_1,\tau_2})(\phi_1(\Phiv_{\tau_1,\tau_2}) - \tau_1^2)}{\phi_1(\Phiv_{\tau_1,\tau_2})-\tau_2^2} + \frac{\phi_1(\Phiv_{\tau_1,\tau_2})(\tau_1^2- \phi_2(\Phiv_{\tau_1,\tau_2}))}{\phi_2(\Phiv_{\tau_1,\tau_2})-\tau_2^2}  \right) \\
    &\Leftrightarrow \tau_1^4 \left(\frac{\Phiv_2 + \tau_2^2}{\Phiv_1} \right) \ge \tau_2^4 \left(\frac{\Phiv_{1}+ \tau_1^2}{\Phiv_{2}} \right) \Leftrightarrow \frac{\Phiv_{2}}{\tau_2^2} \ge \frac{\Phiv_{1}}{\tau_1^2}.
\end{align*}

Now, assume further that $X |\pi(X) = k \sim \Nc_p(\muv_{(k)}, \Sigmav_{(k)})$. For each $k = 1, 2$, note that $\Phiv_{k} = \Wv_{k}^\top \left(\Iv_{n_k} - n_k^{-1}\Jv_{n_k}\right)\Wv_{k} \sim \sigma_{k,1}^2V_k$ where $V_{k} \sim \chi^2(n_k-1)$ and $V_1$ and $V_2$ are independent to each other. Then,
\begin{align*}
    \mathbb{P}(D(\vv_{\hat\alpha_1}) \leq D(\vv_{\hat\alpha_2})) &= \mathbb{P}(\tau_1^{-2}\Phiv_{1} \leq \tau_2^{-2}\Phiv_{2}) \\
    &= \mathbb{P}\left(\frac{(n_1-1)^{-1}V_1}{(n_2-1)^{-1}V_2} \leq \frac{(n_2-1)\tau_2^{-2}\sigma_{2,1}^2}{(n_1-1)\tau_1^{-2}\sigma_{1,1}^2}\right) \\
    &= \mathbb{P}\left(F \leq \frac{(n_2-1)\tau_2^{-2}\sigma_{2,1}^2}{(n_1-1)\tau_1^{-2}\sigma_{1,1}^2}\right) 
\end{align*}
where $F \sim F(n_1-1, n_2-1)$.
\end{proof}
\end{subsection}

\begin{subsection}{Proof of Proposition~\ref{prop:compare ridge and SMDP}}
\begin{proof}
We continue to use the notations in the proofs of Proposition~\ref{prop:ridge one-spike unequal tails m = 1} and Theorem~\ref{thm:compare two ridges}. If $m_1 = m_2 = m = 1$, then  $\upsilon_0$ defined in (\ref{eq:app:upsilon0}) is given by
\begin{align*}
    \upsilon_0 = \left\{r_1^2\left(1 - D_{1,1}^2 - D_{2,1}^2 \right)^2\left(D_{1,1}^2+D_{2,1}^2  \right)^{-1} + \kappa^2 \right\}^{-1/2}.
\end{align*}
To show that $\upsilon_0$ is greater than $\gamma_{1}$ and $\gamma_{2}$ in (\ref{eq:gamma1:m=1}), it suffices to show that 
\begin{align*}
    &\left(1 - D_{1,1}^2 - D_{2,1}^2 \right)^2\left(D_{1,1}^2 + D_{2,1}^2 \right)^{-1} \\
    &~<  \frac{\tau_k^4}{(\phi_1(\Phiv_{\tau_1, \tau_2}) - \tau_k^2)^2}D_{1,1}^2 + \frac{\tau_k^4}{(\phi_2(\Phiv_{\tau_1, \tau_2}) - \tau_k^2)^2} D_{2,1}^2
\end{align*}
for $k = 1, 2$. From (\ref{eq:app:tau1ntau2}), we have
\begin{align*}
    1 - \frac{\phi_1(\Phiv_{\tau_1, \tau_2})}{\phi_1(\Phiv_{\tau_1, \tau_2}) - \tau_k^2}D_{1,1}^2 - \frac{\phi_2(\Phiv_{\tau_1,\tau_2})}{\phi_2(\Phiv_{\tau_1,\tau_2}) - \tau_k^2}D_{2,1}^2 = 0
\end{align*}
for $k = 1, 2$ and thus
\begin{align*}
    \left(1 - D_{1,1}^2 - D_{2,1}^2 \right)^2 &= \tau_k^4 \left(\sum_{i=1}^{2} \frac{D_{i,1}^2}{\phi_i(\Phiv_{\tau_1, \tau_2}) - \tau_k^2 }\right)^2 \\
    &\leq \tau_k^4\left( \sum_{i=1}^{2}D_{i,1}^2 \right)\left(\sum_{i=1}^{2} \frac{D_{i,1}^2 }{(\phi_i(\Phiv_{\tau_1, \tau_2}) - \tau_k^2)^2}\right)
\end{align*}
by the Cauchy-Schwarz inequality. Note that the equality does not hold with probability $1$ since $\phi_1(\Phiv_{\tau_1,\tau_2}) > \phi_2(\Phiv_{\tau_1,\tau_2})$. Hence, $\upsilon_0 > \gamma_k$ for all $k = 1, 2$ with probability $1$.
\end{proof}
\end{subsection}

\begin{subsection}{Proof of Proposition~\ref{prop:ridge one-spike unequal tails m = 2}}
\begin{proof}
Assume that $\beta_1 = \beta_2 = 1$ and $\tau_1 > \tau_2$. Assume further that $m_1 = m_2 = 1$ and $m = 2$ (that is, $\uv_{(1),1} \ne \uv_{(2),1}$). For any given $\alpha \in \Real$, $\left\{\vv_{\alpha} \right\} \in \Ac$ if and only if $\vv_{\alpha}^{\top} \uv_{(1),1} \xrightarrow{P} 0$ and $\vv_{\alpha}^{\top} \uv_{(2),1} \xrightarrow{P} 0$ as $p \to \infty$. Since $\|\tilde{\vv}_{\alpha}\|$ (defined in (\ref{eq:app:pridge:conditioned:eq})) is stochastically bounded for all $\alpha \in \Real$ where $\tilde{\vv}_{\alpha}$ is defined, it suffices to show that there is no ridge parameter $\alpha \in \Real$ such that both of $\tilde{\vv}_{\alpha}^{\top} \uv_{(1),1}$ and $\tilde{\vv}_{\alpha}^{\top} \uv_{(2),1}$ converge to zero as $p \to \infty$. It can be checked that
\begin{equation}\label{eq:app:v11:m=2}
    \begin{aligned}
        v_{11}^2(\Phiv_{\tau_1,\tau_2}) = v_{22}^2(\Phiv_{\tau_1,\tau_2}) &= \frac{1}{2} + \frac{\Phiv_{1} - \Phiv_{2} + \tau_1^2 - \tau_2^2}{2\sqrt{(\Phiv_{1} - \Phiv_{2} + \tau_1^2 - \tau_2^2)^2 + 4(\Rv_{1}^\top \Rv_{2})^2 \Phiv_{1}\Phiv_{2}}}
    \end{aligned}
\end{equation}
and 
\begin{equation}\label{eq:app:v12:m=2}
    \begin{aligned}
        v_{11}(\Phiv_{\tau_1, \tau_2})v_{12}(\Phiv_{\tau_1, \tau_2}) = \frac{(\Rv_{1}^\top\Rv_{2})\Phiv_{1}^{1/2}\Phiv_{2}^{1/2}}{\sqrt{(\Phiv_1 - \Phiv_2 + \tau_1^2 - \tau_2^2)^2 + 4(\Rv_{1}^\top\Rv_{2})\Phiv_{1}\Phiv_{2}}}.
    \end{aligned}
\end{equation}
Combining (\ref{eq:app:v11:m=2}), (\ref{eq:app:v12:m=2}) and (\ref{eq:app:ridge_innerprouct_2}) in the proof of Theorem~\ref{thm:ridge unequal tails asymptotically orthogonal} gives
\begin{align*}
        \tilde{\vv}_{\alpha}^{\top} \uv_{1,1} &\xrightarrow{P} \rv^\top \{\Iv_{2} - (\Rv_{1}\Phiv_{1}^{1/2}\tilde{\Vv}_1(\Phiv_{\tau_1,\tau_2}) + \Rv_{2}\Phiv_{2}^{1/2}\tilde{\Vv}_2(\Phiv_{\tau_1,\tau_2}))(\Dv(\Phiv_{\tau_1,\tau_2})+\alpha\Iv_{2})^{-1}\\
        &\quad\quad\quad(\Rv_{1}\Phiv_{1}^{1/2}\tilde{\Vv}_1(\Phiv_{\tau_1,\tau_2}) + \Rv_{2}\Phiv_{2}^{1/2}\tilde{\Vv}_2(\Phiv_{\tau_1,\tau_2}))^\top\}\Rv_{1} \\
        &= \frac{(\rv^\top\Rv_{1})(\alpha+\tau_1^2)(\alpha +\phi_1(\Phiv_{\tau_1,\tau_2})v_{21}^2(\Phiv_{\tau_1,\tau_2})+\phi_2(\Phiv_{\tau_1,\tau_2})v_{11}^2(\Phiv_{\tau_1,\tau_2}))}{(\alpha+\phi_1(\Phiv_{\tau_1,\tau_2}))(\alpha+\phi_2(\Phiv_{\tau_1,\tau_2}))} \\
        &- \frac{(\rv^\top\Rv_{2})(\alpha+\tau_1^2)\{\Phiv_{1}^{-1/2}\Phiv_{2}^{1/2}(\phi_{1}(\Phiv_{\tau_1,\tau_2}) - \phi_{2}(\Phiv_{\tau_1,\tau_2}))v_{11}(\Phiv_{\tau_1, \tau_2})v_{12}(\Phiv_{\tau_1, \tau_2}) \} }{(\alpha+\phi_1(\Phiv_{\tau_1,\tau_2}))(\alpha+\phi_2(\Phiv_{\tau_1,\tau_2}))} \\
        &= \frac{(\rv^\top\Rv_{1})(\alpha+\tau_1^2)\{\alpha + \Phiv_{2} + \tau_2^2 - (\rv^\top \Rv_{1})^{-1} (\rv^\top \Rv_{2} )(\Rv_{1}^\top \Rv_{2})\Phiv_{2}\}}{(\alpha+\phi_1(\Phiv_{\tau_1,\tau_2}))(\alpha+\phi_2(\Phiv_{\tau_1,\tau_2}))} 
\end{align*}
as $p \to \infty$ where $\rv = (r_1, r_2)^\top$ with $r_i$ ($i = 1, 2$) defined in Lemma~\ref{app:lemma:md}, $\Dv(\Phiv_{\tau_1, \tau_2})$ and $\tilde{\Vv}_i(\Phiv_{\tau_1, \tau_2})$ are defined in (\ref{eq:perturbed covariance notations}). Hence, $\tilde{\vv}_{\alpha}^\top \uv_{(1),1} \xrightarrow{P} 0$ as $p \to \infty$ if and only if $\alpha = -\tau_1^2$ and $\alpha = -\Phiv_{2} - \tau_2^2 + (\rv^\top \Rv_{1})^{-1} (\rv^\top \Rv_{2})(\Rv_{1}^\top \Rv_{2})\Phiv_{2}$. Similarly,
\begin{align*}
        \tilde{\vv}_{\alpha}^{\top} \uv_{(2),1} &\xrightarrow{P} \frac{(\rv^\top\Rv_{2})(\alpha+\tau_2^2)\{\alpha + \Phiv_{1} + \tau_1^2 - (\rv^\top \Rv_{2})^{-1} (\rv^\top \Rv_{1})(\Rv_{1}^\top \Rv_{2})\Phiv_{1}\}}{(\alpha+\phi_1(\Phiv_{\tau_1,\tau_2}))(\alpha+\phi_2(\Phiv_{\tau_1,\tau_2}))}
\end{align*}
as $p \to \infty$. Hence, $\tilde{\vv}_{\alpha}^\top \uv_{(2),1} \xrightarrow{P} 0$ as $p \to \infty$ if and only if $\alpha = -\tau_2^2$ and $\alpha = -\Phiv_{1} - \tau_1^2 + (\rv^\top \Rv_{2})^{-1} (\rv^\top \Rv_{1})(\Rv_{1}^\top \Rv_{2})\Phiv_{1}$ and there is no ridge parameter $\alpha \in \Real$ such that both of $\tilde{\vv}_{\alpha}^\top \uv_{(1),1}$ and $\tilde{\vv}_{\alpha}^\top\uv_{(2),1}$ converge to zero in probability as $p \to \infty$.  
\end{proof}
\end{subsection}

\begin{subsection}{Proof of Theorem~\ref{thm:ridge equal tail eigenvalues piling distance}}
\begin{proof}
Assume that $\beta_1 = \beta_2 = 1$ and $\tau_1 = \tau_2 =: \tau$. For $Y \in \Yc$, assume that $\pi(Y) = 1$ and denote $M(\alpha) = p^{-1/2} \tilde{\vv}_{\alpha}^\top (Y - \bar{X})$ where 
\begin{align*}
    \tilde{\vv}_{\alpha} = \sum\limits_{i\in\Dc} \frac{\alpha_p}{\hat\lambda_i + \alpha_p} \hat\uv_i\left( \frac{1}{\sqrt{p}}\hat\uv_i^\top\dv\right) + \frac{1}{\sqrt{p}}\hat\Uv_2\hat\Uv_2^\top\dv
\end{align*}
so that $\vv_{\alpha} \varpropto \tilde{\vv}_{\alpha}$. Also, suppose that $\hat\alpha$ is an HDLSS-consistent estimator of $-\tau^2$. Recall that in this case $\Dc = \{1, \ldots, m\}$. Using Lemmas~\ref{lem:asymptotic property of Sw equal tails} and~\ref{app:lemma:md} (i), we obtain
\begin{align*}
    &M(\hat\alpha) = \frac{1}{\sqrt{p}}\tilde{\vv}_{\hat\alpha}^\top (Y - \bar{X}) \\
    &= \sum_{i=1}^{m} \frac{\hat\alpha}{\hat\lambda_i / p + \hat\alpha} \left(\frac{1}{\sqrt{p}}\hat\uv_i^\top \dv \right) \left(\frac{1}{\sqrt{p}}\hat\uv_i^\top (Y - \bar{X}) \right) + \frac{1}{\sqrt{p}} \|\hat\Uv_2 \hat\Uv_2^\top \dv \|  \left(\frac{1}{\sqrt{p}}\mdp^\top (Y - \bar{X})\right) \\
    &\xrightarrow{P} \sum_{i=1}^{m}\frac{-\tau^2}{\phi_i(\Phiv)}\sum_{j=1}^{m}\sqrt{\frac{\phi_i(\Phiv)}{\phi_i(\Phiv) + \tau^2}}r_{j}v_{ij}(\Phiv) \sum_{j'=1}^{m}\sqrt{\frac{\phi_i(\Phiv)}{\phi_i(\Phiv)+\tau^2}} t_{j'}v_{ij'}(\Phiv) \\
    &+\kappa \times \frac{1}{\kappa}\left( \eta_2(1-\cos^2\varphi)\delta^2 - \frac{\tau_1^2 - \tau_2^2}{n} +  \sum_{i=1}^{m}\sum_{j=1}^{m}\sum_{j'=1}^{m}r_{j} t_{j'} \frac{\tau^2}{\phi_i(\Phiv) + \tau^2}v_{ij}(\Phiv)v_{ij'}(\Phiv) \right) \\
    &= \eta_2(1- \cos^2\varphi)\delta^2 > 0
\end{align*}
as $p \to \infty$ where $\Phiv$ is defined in (\ref{eq:covariance of commom pc scores}), $t_{j} = \eta_2 \cos\theta_j \delta + \sum_{k=1}^{m_1} [\Rv_{1}]_{jk}\sigma_{1,k}(\zeta_k - \eta_1 \bar{\zv}_{1,k}) - \eta_2 \sum_{k=1}^{m_2} [\Rv_{2}]_{jk}\sigma_{2,k}\bar{\zv}_{2,k}$ and $r_{j}$ is defined in Lemma~\ref{app:lemma:md} (i). Next, again using Lemmas~\ref{lem:asymptotic property of Sw equal tails} and~\ref{app:lemma:md} (i), we have
\begin{align*}
    \|\tilde{\vv}_{\hat\alpha} \|^2 &= \sum_{i=1}^{m} \left(\frac{\hat\alpha}{\hat\lambda_i/p + \hat\alpha}\right)^2\left( \frac{1}{\sqrt{p}}\hat\uv_i^\top\dv\right)^2 + \kappa_{\textup{MDP}}^2 \\
    &\xrightarrow{P} \sum_{i=1}^{m} \frac{\tau^4}{\phi_i(\Phiv)(\phi_i(\Phiv) + \tau^2)} \left(\sum_{j=1}^{m}r_{j} v_{ij}(\Phiv) \right)^2 + \kappa^2 = \upsilon_0^{-2} > 0
\end{align*}
where $\upsilon_0$ is in (\ref{eq:app:upsilon0:equal tails}). Hence, we have
\begin{align*}
    \frac{1}{\sqrt{p}}\vv_{\hat\alpha}^\top(Y - \bar{X}) = \frac{1}{\|\tilde{\vv}_{\hat\alpha} \|} \frac{1}{\sqrt{p}}\tilde{\vv}_{\hat\alpha}^\top(Y - \bar{X}) \xrightarrow{P} \upsilon_0\eta_2(1-\cos^2\varphi)\delta^2
\end{align*}
as $p \to \infty$. Similarly, we can show that $p^{-1/2}\vv_{\hat\alpha}^\top (Y - \bar{X}) \xrightarrow{P} \upsilon_0(-\eta_1(1-\cos^2\varphi)\delta^2)$ as $p \to \infty$ for $Y \in \Yc$ with $\pi(Y) = 2$. 
\end{proof}
\end{subsection}

\begin{subsection}{Proof of Theorem~\ref{thm:ridge unequal tails asymptotically orthogonal}}
\begin{proof}
Assume that $\beta_1 = \beta_2 = 1$ and $\tau_1 > \tau_2$. For notational simplicity, we write $\Dc = \left\{i_1, \ldots, i_{m_1+m_2} \right\}$ so that $i_l < i_{l'}$ if $l < l'$. Write 
\begin{equation}\label{eq:app:pridge:conditioned:eq}
    \begin{aligned}
    \tilde{\vv}_{\alpha} = \sum_{l=1}^{m_1+m_2} \frac{\alpha}{\hat\lambda_{i_l} / p + \alpha} \left( \frac{1}{\sqrt{p}} \hat\uv_{i_l}^\top \dv \right) \hat\uv_{i_l} + \frac{1}{\sqrt{p}}\|\hat\Uv_2\hat\Uv_2^\top \dv \| \mdp.
    \end{aligned} 
\end{equation}
Assume that for each $k = 1, 2$, $\hat\alpha_k$ is an HDLSS-consistent estimator of $-\tau_k^2$. First, we shall show that ${\rm Angle}{(\vv_{\hat\alpha_1}, \uv_{(1),\iota, \Sc})} \xrightarrow{P} \pi / 2$ as $p \to \infty$ for $1 \leq \iota \leq m_1$. Using Lemmas~\ref{lem:asymptotic property of Sw unequal tails} and~\ref{app:lemma:md}, the inner product between $\tilde{\vv}_{\hat\alpha_1}$ and $\uv_{(1),\iota,\Sc}$ becomes
\begin{align*}
    &\tilde{\vv}_{\hat\alpha_1}^{\top} \uv_{(1),\iota, \Sc} = \sum_{l=1}^{m_1+m_2} \frac{\hat\alpha_1 }{\hat\lambda_{i_l} / p+ \hat\alpha_1} \left(\frac{1}{\sqrt{p}}\hat\uv_{i_l}^\top \dv \right) (\uv_{(1),\iota}^\top \hat\uv_{i_l}) + \frac{1}{\sqrt{p}} \|\hat\Uv_2 \hat\Uv_2^\top \dv \| (\uv_{(1),\iota}^\top \mdp) \\
    &= \sum_{j=1}^{m}\left\{\sum_{l=1}^{m_1+m_2} \frac{\hat\alpha_1 }{\hat\lambda_{i_l}/p + \hat\alpha_1} \left(\frac{1}{\sqrt{p}}\hat\uv_{i_l}^\top \dv \right) (\uv_{j}^\top \hat\uv_{i_l}) + \frac{1}{\sqrt{p}} \|\hat\Uv_2 \hat\Uv_2^\top \dv \| (\uv_{j}^\top \mdp) \right\} [\Rv_{1}^{(p)}]_{j\iota} \\
    &\xrightarrow{P} \sum_{j=1}^{m}  \left[  \sum_{j'=1}^{m}\sum_{l=1}^{m_1+m_2} \frac{-\tau_1^2}{(\phi_l(\Phiv_{\tau_1, \tau_2}) - \tau_1^2)} r_{j'} D_{l,j} D_{l,j'}   \right. \\
    &~\quad \left.+~\kappa \times \frac{1}{\kappa}\left\{ \sum_{j'=1}^{m} r_{j'} \left(I_{jj'} - \sum_{l=1}^{m_1+m_2}D_{l,j}D_{l,j'} \right) \right\} \right][\Rv_{1}]_{j\iota} \\
    &= \sum_{j=1}^{m} \sum_{j'=1}^{m} r_{j'} \left(I_{jj'} - \sum_{l=1}^{m_1+m_2} \frac{\phi_l(\Phiv_{\tau_1,\tau_2})}{\phi_l(\Phiv_{\tau_1, \tau_2}) - \tau_1^2} D_{l,j}D_{l,j'}  \right)  [\Rv_{1}]_{j\iota}
\end{align*}
as $p \to \infty$ where $r_j$ is defined in Lemma~\ref{app:lemma:md}, $I_{jj'} = I(j = j')$ and $D_{l,j}$ is defined in (\ref{eq:Dij}).
Write 
\begin{equation}\label{eq:perturbed covariance notations}
    \begin{aligned}
        &\Dv(\Phiv_{\tau_1,\tau_2}) = \diag{(\phi_1(\Phiv_{\tau_1, \tau_2}), \ldots, \phi_{m_1+m_2}(\Phiv_{\tau_1, \tau_2}))} \\
        &\tilde{\Vv}_i(\Phiv_{\tau_1,\tau_2}) = [\tilde{v}_{1i}(\Phiv_{\tau_1,\tau_2}), \ldots, \tilde{v}_{(m_1+m_2)i}(\Phiv_{\tau_1,\tau_2})] ~(i = 1, 2) \\
        &\Vv(\Phiv_{\tau_1, \tau_2}) = [v_1(\Phiv_{\tau_1, \tau_2}), \ldots, v_{m_1+m_2}(\Phiv_{\tau_1, \tau_2})] = [\tilde{\Vv}_1(\Phiv_{\tau_1,\tau_2})^\top,~\tilde{\Vv}_{2}(\Phiv_{\tau_1,\tau_2})^\top]^\top.
    \end{aligned}
\end{equation}
Then we can write 
\begin{equation}\label{eq:app:ridge_innerprouct_1}
    \begin{aligned}
        \tilde{\vv}_{\hat\alpha_1}^{\top} \uv_{1,\iota,\Sc} \xrightarrow{P} \rv^\top \{\Iv_{m} - & (\Rv_{1}\Phiv_{1}^{1/2}\tilde{\Vv}_1(\Phiv_{\tau_1,\tau_2}) + \Rv_{2}\Phiv_{2}^{1/2}\tilde{\Vv}_2(\Phiv_{\tau_1,\tau_2}))(\Dv(\Phiv_{\tau_1,\tau_2})-\tau_1^2\Iv_{m_1+m_2})^{-1}\\
        &(\Rv_{1}\Phiv_{1}^{1/2}\tilde{\Vv}_1(\Phiv_{\tau_1,\tau_2}) + \Rv_{2}\Phiv_{2}^{1/2}\tilde{\Vv}_2(\Phiv_{\tau_1,\tau_2}))^\top\}[\Rv_{1}]^\iota
    \end{aligned}
\end{equation}
as $p \to \infty$ where $\rv = (r_{1}, \ldots, r_{m})^\top$. Note that 
\begin{align*}
    \tilde{\Vv}_k(\Phiv_{\tau_1,\tau_2})\tilde{\Vv}_k(\Phiv_{\tau_1,\tau_2})^\top = \Iv_{m_k}
\end{align*}
for $k = 1, 2$ and 
\begin{align*}
    \tilde{\Vv}_1(\Phiv_{\tau_1,\tau_2})\tilde{\Vv}_2(\Phiv_{\tau_1,\tau_2})^\top = \Ov_{m_1 \times m_2}.
\end{align*}
Also, since $\Phiv_{\tau_1,\tau_2}\Vv(\Phiv_{\tau_1,\tau_2}) = \Vv(\Phiv_{\tau_1,\tau_2})\Dv(\Phiv_{\tau_1,\tau_2})$, we have
\begin{align*}
    \Phiv_{1}\tilde{\Vv}_1(\Phiv_{\tau_1,\tau_2}) + \Phiv_{1}^{1/2}\Rv_{1}^\top \Rv_{2}\Phiv_{2}^{1/2} \tilde{\Vv}_2(\Phiv_{\tau_1,\tau_2}) = \tilde{\Vv}_1(\Phiv_{\tau_1,\tau_2})(\Dv(\Phiv_{\tau_1,\tau_2}) - \tau_1^2\Iv_{m_1+m_2}).     
\end{align*}
Hence,
\begin{equation}\label{eq:app:ridge_innerprouct_2}
    \begin{aligned}
    &(\Rv_{1}\Phiv_{1}^{1/2}\tilde{\Vv}_1(\Phiv_{\tau_1,\tau_2}) + \Rv_{2}\Phiv_{2}^{1/2}\tilde{\Vv}_2(\Phiv_{\tau_1,\tau_2}))^\top\Rv_{1} \\
    &= (\Dv(\Phiv_{\tau_1,\tau_2}) - \tau_1^2\Iv_{m_1+m_2})\tilde{\Vv}_1(\Phiv_{\tau_1,\tau_2})^\top \Phiv_{1}^{-1/2}.
    \end{aligned}
\end{equation}
Combining (\ref{eq:app:ridge_innerprouct_1}) and (\ref{eq:app:ridge_innerprouct_2}) gives
\begin{equation}\label{eq:app:R1}
    \begin{aligned}
    &\tilde{\vv}_{\hat\alpha_1}^{\top} \uv_{(1),\iota,\Sc}  \\
    &\xrightarrow{P} \rv^\top [\Rv_{1} - (\Rv_{1}\Phiv_{1}^{1/2}\tilde{\Vv}_1(\Phiv_{\tau_1,\tau_2}) + \Rv_{2}\Phiv_{2}^{1/2}\tilde{\Vv}_2(\Phiv_{\tau_1,\tau_2}))\tilde{\Vv}_1(\Phiv_{\tau_1,\tau_2})^\top \Phiv_{1}^{-1/2} ]^\iota \\
    &=\rv^\top [\Rv_{1} - \Rv_{1}]^\iota = \rv^\top [\Ov^{m \times m_1}]^\iota = 0
    \end{aligned}
\end{equation}
as $p \to \infty$ for all $1 \leq \iota \leq m_1$. We can easily check that $\|\tilde{\vv}_{\hat\alpha_1}\|$ and $\|\uv_{(1),\iota,\Sc} \|$ converge to strictly positive random variables. In particular,
\begin{equation}\label{eq:gamma1}
    \begin{aligned}
        \|\tilde{\vv}_{\hat\alpha_1}\|^2 \xrightarrow{P} \sum_{l=1}^{m_1+m_2} \frac{\tau_1^4}{(\phi_l(\Phiv_{\tau_1, \tau_2}) - \tau_1^2)^2}\left(\sum_{j=1}^{m} r_{j}D_{l,j} \right)^2 + \kappa^2 =: \frac{1}{\gamma_{1}^2} > 0.
    \end{aligned}
\end{equation}
Hence, ${\rm Angle}{(\vv_{\hat\alpha_1}, \uv_{(1),\iota, \Sc})} \xrightarrow{P} \pi / 2$ as $p \to \infty$ for all $1 \leq \iota \leq m_1$. We will make use of $\gamma_{1}$ in the proof of Theorem~\ref{thm:ridge unequal tails piling distance}. Similarly, we can show that ${\rm Angle}{(\vv_{\hat\alpha_2}, \uv_{(2),\iota,\Sc})} \xrightarrow{P} \pi/2$ as $p \to \infty$ for all $1 \leq \iota \leq m_2$ and
\begin{equation}\label{eq:gamma2}
    \begin{aligned}
        \|\tilde{\vv}_{\hat\alpha_2}\|^2 \xrightarrow{P} \sum_{l=1}^{m_1+m_2} \frac{\tau_2^4}{(\phi_l(\Phiv_{\tau_1, \tau_2}) - \tau_2^2)^2}\left(\sum_{j=1}^{m} r_{j}D_{l,j}  \right)^2 + \kappa^2 =: \frac{1}{\gamma_{2}^2} > 0.
    \end{aligned}
\end{equation}
We will also make use of $\gamma_{2}$ in the proof of Theorem~\ref{thm:ridge unequal tails piling distance}.
\end{proof}
\end{subsection}

\begin{subsection}{Proof of Theorem~\ref{thm:ridge unequal tails piling distance}}
\begin{proof}
We continue to use the same notations as in Theorem~\ref{thm:ridge unequal tails asymptotically orthogonal}. For $k = 1, 2$, assume that $\pi(Y_k) = k$ for $Y_k \in \Yc$ and denote 
\begin{align*}
    M_k(\alpha) = \frac{1}{\sqrt{p}}\tilde{\vv}_{\alpha}^{\top} (Y_k - \bar{X}).
\end{align*}
Assume that $m = m_1$. Using Lemmas~\ref{lem:asymptotic property of Sw unequal tails} and \ref{app:lemma:md} with (\ref{eq:thm:L1:a+b:ne}) and (\ref{eq:L1:MDP:a:2}), we have
\begin{align*}
    &M_1(\hat\alpha_1) = \frac{1}{\sqrt{p}}\tilde{\vv}_{\hat\alpha_1}^{\top} (Y_1 - \bar{X}) \\
    &= \sum_{l=1}^{m_1+m_2} \frac{\hat\alpha_1}{\hat\lambda_{i_l} /p+ \hat\alpha_1}\left(\frac{1}{\sqrt{p}}\hat\uv_{i_l}^\top \dv \right) \left\{ \frac{1}{\sqrt{p}}\hat\uv_{i_l}^\top (Y_1 - \bar{X}) \right\} + \frac{1}{\sqrt{p}} \|\hat\Uv_2 \hat\Uv_2^\top \dv \| \left\{ \frac{1}{\sqrt{p}}\mdp^\top (Y_1 - \bar{X})\right\} \\
    &\xrightarrow{P} \sum_{l=1}^{m_1+m_2}\sum_{j=1}^{m}\sum_{j'=1}^{m} \frac{-\tau_1^2}{(\phi_l(\Phiv_{\tau_1,\tau_2}) - \tau_1^2)}t_{j}r_{j'} D_{l,j}D_{l,j'}  \\
    &~+\kappa \times \frac{1}{\kappa} \left\{\eta_2(1-\cos^2\varphi)\delta^2 - \frac{\tau_1^2 - \tau_2^2}{n} + \sum_{j=1}^m \sum_{j'=1}^{m} t_j r_{j'} \left(I_{jj'} - \sum_{l=1}^{m_1+m_2}D_{l,j}D_{l,j'}  \right) \right\} \\
    &= \eta_2(1-\cos^2\varphi)\delta^2 - \frac{\tau_1^2 - \tau_2^2}{n} + \sum_{j=1}^{m}\sum_{j'=1}^{m}t_jr_{j'}\left(I_{jj'} -\sum_{l=1}^{m_1+m_2} \frac{\phi_l(\Phiv_{\tau_1,\tau_2})}{\phi_l(\Phiv_{\tau_1, \tau_2}) - \tau_1^2}D_{l,j}D_{l,j'} \right)
\end{align*}
as $p \to \infty$. Note that in this case $\Rv_{1}$ is invertible since $m = m_1$. Then using (\ref{eq:app:R1}), we have
$$I_{jj'} -\sum_{l=1}^{m_1+m_2} \frac{\phi_l(\Phiv_{\tau_1,\tau_2})}{\phi_l(\Phiv_{\tau_1, \tau_2}) - \tau_1^2}D_{l,j}D_{l,j'} = 0$$
for all $1 \leq j, j' \leq m$. Therefore, we have
\begin{align*}
        \frac{1}{\sqrt{p}}\vv_{\hat\alpha_1}^{\top} (Y_1 - \bar{X}) = \frac{1}{\|\tilde{\vv}_{\hat\alpha_1}\|} \frac{1}{\sqrt{p}}\tilde{\vv}_{\hat\alpha_1}^{\top}(Y_1 - \bar{X}) \xrightarrow{P} \gamma_{1}\left( \eta_2 (1- \cos^2\varphi)\delta^2 - \frac{\tau_1^2 - \tau_2^2}{n} \right)
\end{align*}
as $p \to \infty$ where $\gamma_{1}$ is defined in (\ref{eq:gamma1}). In a similar way, we can show that
\begin{align*}
    &M_2(\hat\alpha_1) = \frac{1}{\sqrt{p}}\tilde{\vv}_{\hat\alpha_1}^{ \top} (Y_2 - \bar{X}) \\
    &\xrightarrow{P} -\eta_1(1-\cos^2\varphi)\delta^2 - \frac{\tau_1^2 - \tau_2^2}{n} + \sum_{j=1}^{m}\sum_{j'=1}^{m}s_{j}r_{j'}\left(I_{jj'} -\sum_{l=1}^{m_1+m_2} \frac{\phi_l(\Phiv_{\tau_1,\tau_2})}{\phi_l(\Phiv_{\tau_1, \tau_2}) - \tau_1^2}D_{l,j}D_{l,j'} \right) \\
    &= -\eta_1(1-\cos^2\varphi)\delta^2 - \frac{\tau_1^2 - \tau_2^2}{n}
\end{align*}
as $p \to \infty$ where $s_{j} = -\eta_1\cos\theta_j \delta  - \eta_1 \sum_{k=1}^{m_1}[\Rv_{1}]_{jk}\sigma_{1,k}\bar{\zv}_{1,k} + \sum_{k=1}^{m_2}[\Rv_{2}]_{jk}\sigma_{2,k}(\zeta_k - \eta_2\bar{\zv}_{2,k})$ for $1 \leq j \leq m$. Hence, we have
\begin{align*}
        \frac{1}{\sqrt{p}}\vv_{\hat\alpha_1}^{\top} (Y_2 - \bar{X}) = \frac{1}{\|\tilde{\vv}_{\hat\alpha_1} \|} \frac{1}{\sqrt{p}}\tilde{\vv}_{\hat\alpha_1}^{\top} (Y_2 - \bar{X}) \xrightarrow{P} \gamma_{1} \left( -\eta_1 (1- \cos^2\varphi)\delta^2 - \frac{\tau_1^2 - \tau_2^2}{n} \right)
\end{align*}
as $p \to \infty$. Using similar arguments, we can show the results for the case of $m = m_2$.
\end{proof}
\end{subsection}
\end{appendix}
\end{document}